\documentclass[10pt]{article}

\usepackage{mathtools}
\usepackage{amsthm}
\usepackage{amssymb}
\usepackage{mathrsfs}

\usepackage{parskip} 
\makeatletter
\def\thm@space@setup{%
  \thm@preskip=\parskip \thm@postskip=\parskip 
}
\makeatother

\usepackage{setspace}

\usepackage{multirow}
\usepackage{array}
\newcolumntype{C}{>{$}c<{$}} 

\usepackage[margin=25mm]{geometry}

\allowdisplaybreaks

\usepackage{tikz}
\usetikzlibrary{matrix,arrows,decorations.pathmorphing}
\tikzset{
	nom/.style={circle, inner sep=2pt, color=black, fill=brown!25!},
	rnom/.style={rectangle, inner sep=2pt, color=black, fill=brown!25!},
	atyp/.style={circle, draw=black, fill=black, inner sep=2pt, outer sep=0pt, minimum size=5pt}, 
	wt/.style={circle, fill=black!30!white, inner sep=2pt, outer sep=0pt, minimum size=5pt}, 
}

\usepackage{enumitem}
\setitemize{leftmargin=*}   
\setenumerate{leftmargin=*, 
  label=\textup{(\roman*)}} 

\usepackage{hyperref}
\definecolor{darkblue}{rgb}{0,0.1,0.5}
\definecolor{darkred}{RGB}{200,26,26}
\definecolor{darkgreen}{RGB}{28,133,26}

\definecolor{verma}{RGB}{200,26,26} 
\definecolor{dverma}{RGB}{0,0,255}  
\definecolor{other}{RGB}{28,133,26} 

\hypersetup{
    colorlinks=true, 
    linkcolor=darkblue, 
    urlcolor=blue, 
    linktoc=all, 
    citecolor=darkgreen
    }

\usepackage[capitalise,noabbrev]{cleveref}

\newtheorem{theorem}{Theorem}[section]
\newtheorem{corollary}[theorem]{Corollary}
\newtheorem{lemma}[theorem]{Lemma}
\newtheorem{proposition}[theorem]{Proposition}
\newtheorem{conjecture}[theorem]{Conjecture}
\theoremstyle{definition}

\newtheorem{definition}[theorem]{Definition}

\Crefname{conjecture}{Conjecture}{Conjectures}

\numberwithin{equation}{section}
\numberwithin{theorem}{section}

\DeclareMathOperator{\rad}{rad}
\DeclareMathOperator{\soc}{soc}
\DeclareMathOperator{\spn}{span}

\newcommand{\blank}{{-}}

\newcommand{\into}{\hookrightarrow}
\newcommand{\onto}{\twoheadrightarrow}

\newcommand{\corrto}{\longleftrightarrow} 

\newcommand{\set}[2]{\left\{ #1 \middle\bracevert #2 \right\}}

\newcommand{\ee}{\mathsf{e}} 
\newcommand{\ii}{\mathsf{i}} 

\newcommand{\CC}{\mathbb{C}}
\newcommand{\QQ}{\mathbb{Q}}
\newcommand{\RR}{\mathbb{R}}
\newcommand{\ZZ}{\mathbb{Z}}
\newcommand{\NN}{\ZZ_{\ge0}}

\newcommand{\mfsl}[1]{\mathfrak{sl}_{#1}}

\newcommand{\sltwo}{\mathfrak{sl}_2}
\newcommand{\slthree}{\mathfrak{sl}_3}
\newcommand{\aslthree}{\widehat{\mathfrak{sl}}_3}

\newcommand{\csub}{\mathfrak{h}}
\newcommand{\killing}[2]{\langle #1, #2 \rangle}
\newcommand{\norm}[1]{\lVert #1 \rVert}

\newcommand{\weylaut}[1]{\mathrm{w}_{#1}}           
\newcommand{\conjaut}{\mathrm{c}}                   
\newcommand{\sfaut}[1]{\sigma^{#1}}                 

\newcommand{\weylmod}[2]{\weylaut{#1}(#2)}    
\newcommand{\conjmod}[1]{\conjaut(#1)}        
\newcommand{\sfmod}[2]{\sfaut{#1}(#2)}        

\newcommand{\Uq}[1]{\mathcal{U}_q(#1)}
\newcommand{\Uqbar}[1]{\overline{\mathcal{U}}_q(#1)}
\newcommand{\UHbar}[1]{\overline{\mathcal{U}}{}_q^H\!(#1)}

\newcommand{\qgrp}{\overline{\mathcal{U}}{}_{\ii}^H\!(\slthree)} 

\newcommand{\qgrpnilm}{\overline{\mathcal{U}}{}_{\ii}^H\!(\mfn^-)} 
\newcommand{\qgrpnilpm}{\overline{\mathcal{U}}{}_{\ii}^H\!(\mfn^{\pm})}

\newcommand{\affvoa}{L_{-\frac{3}{2}}(\slthree)}
\newcommand{\singlet}{W^0_{\!A_1}(p)}
\newcommand{\triplet}{W_{\!A_1}(p)}
\newcommand{\singvoa}{W^0_{\!A_2}(2)}
\newcommand{\octvoa}{W_{\!A_2}(2)}

\newcommand{\maxi}[1]{J_{#1}}
\newcommand{\irred}[1]{L_{#1}}
\newcommand{\verma}[1]{M_{#1}}
\newcommand{\proj}[1]{P_{#1}}

\newcommand{\irr}[2]{\irred{#2}^{(#1)}} 
\newcommand{\pro}[2]{\proj{#2}^{(#1)}}  
\newcommand{\quo}[2]{N^{(#1)}_{#2}}

\newcommand{\airr}[1]{\widehat{\mathcal{L}}_{#1}}       
\newcommand{\asemi}[1]{\widehat{\mathcal{S}}_{[#1]}}    
\newcommand{\arel}[1]{\widehat{\mathcal{R}}_{[#1]}}     
\newcommand{\airrpro}[1]{\widehat{\mathcal{P}}_{#1}}    
\newcommand{\asemipro}[1]{\widehat{\mathcal{Q}}_{[#1]}} 

\newcommand{\fock}[1]{F_{#1}}

\newcommand{\singirred}[1]{E_{#1}}
\newcommand{\singirr}[2]{\singirred{#2}^{(#1)}}
\newcommand{\singpro}[2]{Q^{(#1)}_{#2}}

\newcommand{\octirr}[2]{\mathbb{O}^{(#1)}_{[#2]}}
\newcommand{\octpro}[2]{\mathbb{P}^{(#1)}_{[#2]}}
\newcommand{\octquo}[2]{\mathbb{N}^{(#1)}_{[#2]}}

\newcommand{\Grfuse}{\mathbin{\widehat{\boxtimes}}}       
\newcommand{\singfuse}{\mathbin{\boxtimes}}               
\newcommand{\afffuse}{\mathbin{\widehat{\boxtimes}}}      
\newcommand{\octfuse}{\mathbin{\widetilde{\boxtimes}}}    

\newcommand{\qgrpcat}{\mathscr{C}}
\newcommand{\affcat}{\mathscr{D}}
\newcommand{\singcat}{\mathscr{E}}
\newcommand{\fockcat}{\mathscr{F}}

\newcommand{\induct}[1]{\operatorname{Ind} \bigl( #1 \bigr)} 

\DeclareMathOperator{\chmap}{ch}
\newcommand{\ch}[1]{\chmap[#1]} 

\DeclareFontFamily{U}{mathx}{\hyphenchar\font45}
\DeclareFontShape{U}{mathx}{m}{n}{
      <5> <6> <7> <8> <9> <10>
      <10.95> <12> <14.4> <17.28> <20.74> <24.88>
      mathx10
      }{}
\DeclareSymbolFont{mathx}{U}{mathx}{m}{n}
\DeclareFontSubstitution{U}{mathx}{m}{n}
\DeclareMathAccent{\widecheck}{0}{mathx}{"71}

\newcommand{\dual}[1]{\widecheck{#1}}
\newcommand{\outdual}[1]{\left( #1 \right)^{\raisebox{-1ex}{$\!\!\dual{\hphantom{V}}$}}} 

\newcommand{\dverma}[1]{\dual{M}_{#1}}
\newcommand{\dproj}[1]{\dual{P}_{#1}}

\newcommand{\mfg}{\mathfrak{g}}
\newcommand{\mfh}{\mathfrak{h}}
\newcommand{\mfn}{\mathfrak{n}}

\begin{document}

\title{A Kazhdan--Lusztig correspondence for 
$\affvoa$}

\author{Thomas Creutzig\footnote{Department of Mathematics, University of Alberta. \href{mailto:creutzig@ualberta.ca}{creutzig@ualberta.ca}}
	\and David Ridout\footnote{School of Mathematics and Statistics, University of Melbourne.  \href{mailto:david.ridout@unimelb.edu.au}{david.ridout@unimelb.edu.au}}
	\and Matthew Rupert\footnote{Department of Mathematics and Statistics, Utah State University. \href{mailto:matthew.rupert@usu.edu}{matthew.rupert@usu.edu}}}

\date{}

\maketitle

\begin{abstract}
	The abelian and monoidal structure of the category of smooth weight modules over a non-integrable affine vertex algebra of rank greater than one is an interesting, difficult and essentially wide open problem. Even conjectures are lacking. This work details and tests such a conjecture for $\affvoa$ via a logarithmic Kazhdan--Lusztig correspondence.

	We first investigate the representation theory of $\qgrp$, the unrolled restricted quantum group of $\slthree$ at fourth root of unity. In particular, we analyse its finite-dimensional weight category, determining Loewy diagrams for all projective indecomposables and decomposing all tensor products of irreducibles. Our motivation is that this category is conjecturally braided tensor equivalent to a category of $\singvoa$-modules. Here, $\singvoa$ is an orbifold of the octuplet vertex algebra $\octvoa$ of Semikhatov, the latter being the natural $\slthree$-analogue of the well known triplet algebra. Moreover, $\singvoa$ is the parafermionic coset of the affine vertex algebra $\affvoa$.

	We formulate an explicit conjecture relating the representation theory of $\singvoa$ and $\qgrp$ and work out the resulting structures of the corresponding $\affvoa$-modules. In particular, we obtain conjectural Loewy diagrams for the latter's projective indecomposables and decompositions for the fusion products of its irreducibles. These products coincide with those recently computed via Verlinde's formula. Finally, we give analogous results for $\octvoa$.
\end{abstract}

\tableofcontents

\onehalfspacing

\section{Introduction}

Two excellent sources of braided tensor categories are the representation categories of modules of quantum groups and of vertex operator algebras. Here, we are mostly interested in categories that are not semisimple. Besides being interesting in their own right, such categories give rise to rich invariants of $3$-manifolds \cite{CGP} and they naturally emerge from quantum field theories \cite{CDGG}. In this work, we study the case of the unrolled restricted quantum group of $\slthree$ at fourth root of unity, $\qgrp$, and its connection to the simple affine vertex operator algebra of $ \mathfrak{sl}_3$ at level $-\frac{3}{2}$, $\affvoa$. Let us start by reviewing the state of the art.

\subsection{Affine vertex algebras at admissible level}

Let $\mfg$ be a finite-dimensional complex simple Lie algebra. Then, to any $k \in \CC$, one associates the simple affine vertex algebra of $\mfg$ at level $k$, which we denote by $L_k(\mfg)$. If $k$ is a non-negative integer, then $L_k(\mfg)$ is rational and lisse, hence its module category is a semisimple modular tensor category \cite{HUang-MTC}. It is a major research goal to thoroughly understand the module categories of $L_k(\mfg)$ for general $k$. A relatively accessible, though still extremely challenging, case is when $k$ is an admissible level \cite{KacMod88}.

This goal entails first classifying the irreducible modules, second understanding extensions between them and third establishing a vertex tensor category structure and determining its properties. There are of course different types of categories of modules that one can study. The simplest one is the category of ordinary modules, these being the positive-energy modules whose conformal weight spaces are finite-dimensional. For $k$ admissible, this category is a semisimple vertex tensor category \cite{Ara-adm, CHY} for which rigidity is also mostly established \cite{C-adm, CKL-B, CGL}.

It turns out that when $k$ is not a non-negative integer, $L_k(\mfg)$ usually has many modules that are not positive-energy or have infinite-dimensional conformal weight spaces. Among other reasons, one must then extend the category of ordinary modules to include some of these in order to obtain a modular-invariant partition function \cite{KacMod88,CR1,CR2}. There is therefore significant current effort to understand these extended categories of $L_k(\mfg)$-modules, particularly those that include the so-called relaxed highest-weight modules \cite{FeiEqu97,RidRel15}, see for example \cite{A,RidBos14,AraWei16,RidAdm17,Ada-log,KR1,CreCos18,AdaFus19,KR2,AllBos20,FutSim20,FK,AdaRea20,FehCla20,FutAdm21,KRW}.

In the case $\mfg = \mathfrak{sl}_2$, the irreducible positive-energy modules were classified in \cite{AM-sl2}. Reducible positive-energy modules were subsequently considered in many works, including \cite{GabFus01,RidFus10,CR1,CR2,KR1}. Moreover, modules that are not positive-energy are often obtained by applying spectral flow functors to those that are. In fact, all irreducible weight modules with finite-dimensional weight spaces can be obtained in this fashion \cite{FutCla01}. Non-trivial extensions between such spectrally flowed modules were introduced in \cite{GabFus01,RidFus10} and a general construction of Adamovi\'{c} \cite{Ada-log} covers all the indecomposable projective modules \cite{ACK} (in the category of weight modules).

Beyond $\mfg = \mathfrak{sl}_2$, the structures of indecomposable projective modules are currently unknown. There are however many conjectural correspondences between categories of modules of certain quantum groups and vertex operator algebras. These are often called logarithmic Kazhdan--Lusztig correspondences and we will review them in a moment. One such correspondence relates the unrolled restricted quantum group $\qgrp$, at a fourth root of unity, and a vertex operator algebra related to $\affvoa$ \cite{CR}. (This correspondence is in fact a variant of the conjectured  logarithmic Kazhdan--Lusztig correspondence between Feigin--Tipunin algebras and restricted quantum groups \cite{FT}.) Since the representation theory of quantum groups is reasonably accessible, this is a promising direction to deduce otherwise unimaginable structural conjectures for indecomposable projective $\affvoa$-modules.

Eventually, one would not only like to understand the linear categories of an affine vertex algebra, but also their monoidal structures. For example, one would like to obtain fusion rules for the irreducible modules. Verlinde's formula \cite{VerFus88} is a particularly nice way to obtain (Grothendieck) fusion rules, though it is only established for rational and lisse vertex operator algebras \cite{Hu-ver}. There is nevertheless a generalised Verlinde conjecture \cite{CreLog13,RidVer14} that has been successfully applied to $L_k(\sltwo)$ at admissible levels \cite{CR1, CR2}. This conjectural formula has been recently lifted to $\affvoa$ in a companion paper \cite{KRW}, so that the tensor product decompositions of $\qgrp$-modules and the fusion rules of the corresponding $\affvoa$-modules may be compared.

\subsection{Logarithmic Kazhdan--Lusztig correspondences}

The category of ordinary modules of the affine vertex operator algebra $L_k(\mfg)$, with $\mfg$ simple and $k+h^\vee \notin \QQ_{\ge0}$, is equivalent as a braided tensor category to the category of finite-dimensional modules of Lusztig's quantum group $\Uq{\mfg}$ at $q = \ee^{\pi \ii/\ell(k+h^\vee)}$, where $h^\vee$ and $\ell$ are the dual Coxeter and lacing numbers, respectively, of $\mfg$.  This is the original Kazhdan--Lusztig correspondence \cite{KL1, KL2, KL3, KL4} established in the 1990s.

Related progress on the mathematics of logarithmic conformal field theory was made in the mid-2000s \cite{Feigin:2005zx, Feigin:2005xs, Feigin:2006iv, Feigin:2006xa}. These papers conjectured a logarithmic Kazhdan--Lusztig correspondence between the representation categories of the triplet algebras $\triplet$ (parameterised by an integer $p \ge 2$) and the restricted quantum group $\Uqbar{\sltwo}$ at $q = \ee^{\pi \ii/p}$. It was later noticed that the correspondence is in fact between the triplet algebras and a quasi-Hopf algebra whose underlying algebra is the quantum group. One calls this a quasi-Hopf modification of the quantum group \cite{Gainutdinov:2015lja,CGR}. In this instance, the quantum group category was quickly understood and it thus provided reliable conjectures about $\triplet$-module categories, in particular their abelian structures and fusion rules. It then took many works \cite{AM-triplet,AM-singlet,KS-11,Nagatomo:2009xp,TsuExt13,CGR,mcrae2021structure,CLR,GN} to completely settle these conjectures.

A related conjectural logarithmic Kazhdan--Lusztig correspondence involves the unrolled restricted quantum group $\UHbar{\sltwo}$ at $q = \ee^{\pi \ii/p}$ and the singlet algebra $\singlet$ \cite{Costantino:2014sma}, the latter being the $U(1)$-orbifold of $\triplet$. This in turn means that $\triplet$ is a simple current extension of $\singlet$ \cite{RidMod13}. The corresponding relation between the two quantum groups $\Uqbar{\sltwo}$ and $\UHbar{\sltwo}$ is sometimes called uprolling and was first worked out in \cite{CGR}. Moreover, a full understanding of the tensor category of $\singlet$ was recently achieved in \cite{CMY-singlet, CKMY-singlet}.

This unrolled/singlet correspondence is of significant interest. One reason is that very interesting vertex algebras, called $\mathcal B_p$-algebras \cite{Creutzig:2013pda}, may be realised as simple current extensions of $\singlet$ and a rank-one Heisenberg vertex algebra. For $p=2$, one gets the $\beta\gamma$-vertex algebra \cite{RidSL210}; for $p=3$, one gets $L_{-4/3}(\sltwo)$ \cite{Ada-sl2}; while for larger $p$, one gets certain subregular W-algebras of $\mathfrak{sl}_{p-1}$ \cite{Adamovic:2020lvj}. We remark that the $\mathcal B_p$-algebras are important in physics as they arise as chiral algebras of certain Argyres--Douglas theories \cite{Beem:2017ooy, C1, Adamovic:2020lvj}, these being a type of four-dimensional superconformal quantum field theory.

The quantum group correspondence was then used to give a complete, but of course still conjectural, description of the category of smooth weight modules of all the $\mathcal{B}_p$-algebras \cite{ACKR}. In particular, fusion rules were conjectured via Verlinde's formula and, even better, the normalised entries of the modular S-transformation corresponding to the generalised Verlinde formula were found to coincide with the normalised modified traces of open Hopf links, as was already anticipated in \cite{CMR, Creutzig:2016fms, Gainutdinov:2016qhz}. In summary, these logarithmic Kazhdan--Lusztig correspondences between quantum group categories and categories of modules of vertex algebras, both associated to $\sltwo$, have been invaluable for better understanding the representation theory of the vertex algebras. It is thus reasonable to extend this program to higher ranks.

In this direction, there is a natural generalisation of the triplet algebras called the Feigin--Tipunin algebras \cite{FT}. They are parameterised by a simply laced simple Lie algebra $\mfg$ and an integer $p\ge2$. The corresponding quantum groups are expected to be quasi-Hopf modifications of the restricted quantum groups $\Uqbar{\mfg}$ at $q=\ee^{\pi\ii/p}$. The unrolled version should moreover correspond to the orbifold of the Feigin--Tipunin algebra by the maximal torus of the Lie group of $\mfg$. Unfortunately, very few results have thus far been established for these vertex algebras, see \cite{SemVir11, SemNot13, CM2, Sugimoto:2021lok, AdaPar20, sugimoto2021simplicity}.

As in the case of $\mfg = \sltwo$, there exist extensions of the tensor product of these Feigin--Tipunin orbifolds by Heisenberg vertex algebras of the same rank as $\mfg$ \cite{C2}. These extensions also arise in physics as chiral algebras of Argyres--Douglas theories \cite{Buican:2017uka, C2}. The very first higher-rank example is $\mfg = \slthree$ and $p=2$. In this case, the extension was identified by Adamovi\'{c} as $\affvoa$ \cite{A}. It is also worth pointing out that a different extension, obtained by first tensoring with a rank-one Heisenberg vertex algebra, gives the small $N=4$ superconformal algebra at central charge $-9$.

The quantum group associated to the orbifold $\singvoa$ of the Feigin--Tipunin algebra for $\slthree$ and $p=2$ is $\qgrp$. General results for these unrolled restricted quantum groups and their uprollings were obtained in \cite{GP, CGP, Costantino:2014sma, R, CR}, but the explicit structures of the projective indecomposables and decompositions of tensor products were not known. We thus decided to study this example in order to compare our findings with the recent results of \cite{KRW} on the representation theory of $\affvoa$, via the conjectural logarithmic Kazhdan--Lusztig correspondence.

\subsection{Overview}

We begin in \cref{sec:bgrnd} by recalling the definition of the unrolled restricted quantum group $\qgrp$ and outlining some of the broader features of its representation theory, following \cite{R}.  We also outline the theory of Loewy series and diagrams to fix notation and explain the subtle relation between this theory and duality.

\cref{qgrpmod} is devoted to a detailed study of the category of finite-dimensional weight $\qgrp$-modules. The simple objects of this category are realised as irreducible quotients $\irred{\lambda}$ of Verma modules $\verma{\lambda}$. It follows from \cite[Propositions~4.4 and 4.6]{R} that the irreducibility of $M_{\lambda}$ is determined by the ``typicality'' of the scalars $\lambda_k=\killing{\lambda+\rho}{\alpha_k}$, $k=1,2,3$, where $\alpha_1$ and $\alpha_2$ are the simple roots of $\slthree$ and $\alpha_3 = \rho = \alpha_1 + \alpha_2$.  We call $\lambda_k \in \CC$ typical if $\lambda_k \notin 2\ZZ+1$ and a weight $\lambda \in \mfh^*$ typical if all $\lambda_j$ are typical.  Weights that are not typical are called atypical and are further subdivided by the degree of atypicality, this being the number of $\lambda_k$ that are atypical (in $2\ZZ+1$).

We find that the structures and invariants of the irreducibles $\irred{\lambda}$, Vermas $\verma{\lambda}$ and projective covers $\proj{\lambda}$ are captured by four different classes of atypicality for the weight $\lambda$:
\begin{itemize}
	\item $\lambda$ typical (atypical of degree $0$);
	\item $\lambda$ atypical of degree $1$;
	\item $\lambda$ atypical of degree $2$ with $\lambda_3$ atypical;
	\item $\lambda$ atypical of degree $2$ with $\lambda_3$ typical.
\end{itemize}
For example, the dimensions of the $L_{\lambda}$ corresponding to these four classes are $8$, $4$, $3$ and $1$, respectively.  This, along with the characters of the $L_{\lambda}$, is the content of \cref{irreduciblecharacter}.

From here, we establish in \cref{LoewyM} the Loewy diagrams that summarise the structures of the reducible Verma modules, see \cref{fig:LoewyM}. The quite non-trivial generalisation for the indecomposable projectives is reported in \cref{thm:LoewyP,fig:LoewyP}. Both results again split into cases along the lines of the atypicality classes introduced above. This identification of projectives is moreover extremely helpful for decomposing many of the tensor products of irreducible $\qgrp$-modules, this being the subject of \cref{tensorsec}.

The implications of these results for the representation theories of the vertex operator algebras $\singvoa$, $\octvoa$ and $\affvoa$ is discussed in \cref{sec:KL}. The latter here serves as a partial check as some features of the representation theory of $\affvoa$ have been recently established in \cite{KRW}. These results are reviewed in \cref{sec:KLaff}.  However, the conjectural logarithmic Kazhdan--Lusztig correspondence involves the Feigin--Tipunin orbifold $\singvoa$, not $\affvoa$, so \cref{sec:KLcoset} details the relationship between these two vertex operator algebras. In particular, we determine the branching rules of the irreducible $\affvoa$-modules, that is their decomposition into modules of $\singvoa \otimes H$ (where $H$ denotes a rank-$2$ Heisenberg algebra).

The Kazhdan--Lusztig correspondence itself is the topic of \cref{sec:KLcorr}, see \cref{conj:bte}.  This immediately results in conjectural Loewy diagrams for the indecomposable projective $\singvoa$-modules, see \cref{fig:LoewyE}.  To lift these to $\affvoa$, we use the theory of vertex algebra extensions to construct an exact monoidal functor of a suitable subcategory of $\singvoa \otimes H$-modules to the weight category of $\affvoa$.  This is detailed in \cref{sec:KLrecon} with the conjectured Loewy diagrams for the indecomposable projective $\affvoa$-modules reported in \cref{fig:LoewyAff}.  This gives the first predictions for the structures of these projective covers and will be invaluable guides for their subsequent construction and analysis.

We also test the tensor part of the conjectured equivalence by using the tensor product decompositions for $\qgrp$, given in \cref{tensorsec}, to predict fusion rules for $\affvoa$ in \cref{sec:KLfusion} and checking that they descend to the Grothendieck fusion rules recently calculated in \cite{KRW}.  This is a highly non-trivial check of both our \cref{conj:bte} and the standard Verlinde conjecture for logarithmic models \cite{CreLog13,RidVer14}.

To conclude, we again use vertex algebra extension theory to study the representation theory of $\octvoa$.  This Feigin--Tipunin algebra was dubbed the octuplet algebra by Semikhatov \cite{SemVir11,SemNot13} as it is generated by eight primary fields of conformal weight four.  As an $\slthree$ generalisation of the triplet algebras, it is expected to be $C_2$-cofinite.  We verify in \cref{sec:KLoctuplet}, again assuming \cref{conj:bte}, that there are finitely many irreducibles (up to isomorphism) and that all are ordinary.  Predictions for their fusion rules are deduced and Loewy diagrams for the projective indecomposables are proposed in \cref{fig:LoewyOct}.

\subsection*{Acknowledgements}

The work of TC is supported by NSERC Grant Number RES0048511.
DR's research is supported by the Australian Research Council Discovery Projects DP160101520 and DP210101502, as well as an Australian Research Council Future Fellowship FT200100431.

\section{Preliminaries} \label{sec:bgrnd}

In this section, the unrolled restricted quantum group $\qgrp$ is defined and various results are recalled for later use in identifying the Loewy diagrams of its reducible Verma modules and their projective covers.

\subsection{The quantum group $\qgrp$}

Fix a choice of Cartan subalgebra $\csub \subset \slthree$.  The simple roots are denoted by $\alpha_i \in \csub^*$, $i=1,2$, and the highest root is $\alpha_3 = \alpha_1 + \alpha_2$.  The Killing form $\killing{\blank}{\blank}$ is normalised so that $\killing{\alpha_i}{\alpha_j} = A_{i,j}$, $i=1,2$, where $A = \begin{psmallmatrix} 2&-1\\-1&2 \end{psmallmatrix}$ is the Cartan matrix.  The dual form on $\csub^*$ will also be denoted by $\killing{\blank}{\blank}$.  Fix also a lattice $L \subset \csub^*$ such that $Q \subseteq L \subseteq P$, where $P$ and $Q$ are the weight and root lattices of $\slthree$, respectively.

Here and below, $\ii$ denotes $\sqrt{-1} \equiv \ee^{\pi \ii/2}$.

\begin{definition} \label{Uisl3}
	The \emph{unrolled restricted quantum group} $\qgrp$ is the complex Hopf algebra with generators $X_{\pm j}$, $H_j$ and $K_{\gamma}$, for $j=1,2$ and $\gamma \in L$, and relations
	\begin{subequations}
		\begin{equation} \label{eq:Uisl3rels}
			\begin{aligned}
				K_0 &= 1, & K_{\gamma_1}K_{\gamma_2} &= K_{\gamma_1+\gamma_2}, & K_{\gamma}X_{\pm j}K_{-\gamma} &= \ii^{\pm \langle \gamma, \alpha_j \rangle} X_{\pm j}, \\
				[X_{j},X_{-j'}] &= \delta_{j,j'} \frac{K_{\alpha_{j}}-K_{-\alpha_{j}}}{2\ii}, & X_{\pm j}^2 &= 0, & (X_{\pm1} X_{\pm2})^2 &= (X_{\pm2} X_{\pm1})^2, \\
				[H_{j},X_{\pm j'}] &= \pm A_{j,j'}X_{\pm j'}, & [H_{j},H_{j'}] &= 0, & [H_j,K_{\gamma}] &=0.
			\end{aligned}
		\end{equation}
		(The notation $K_j=K_{\alpha_j}$ and $K_{-\gamma}=K_{\gamma}^{-1}$ will also be used.)  The coproduct, counit and antipode are given by
		\begin{equation} \label{eq:Uisl3hopf}
			\begin{aligned}
				\Delta(K_{\gamma}) &= K_{\gamma} \otimes K_{\gamma}, & \epsilon(K_{\gamma}) &= 1, &  S(K_{\gamma}) &= K_{-\gamma},\\
				\Delta(X_j) &= X_j \otimes K_j + 1 \otimes X_j, & \epsilon(X_j) &= 0, & S(X_j) &= -X_jK_j^{-1},\\
				\Delta(X_{-j}) &= X_{-j} \otimes 1 + K_j^{-1} \otimes X_{-j}, & \epsilon(X_{-j}) &= 0, & S(X_{-j}) &= -K_jX_{-j}\\
				\Delta(H_j) &= 1 \otimes H_j + H_j \otimes 1, & \epsilon(H_j) &= 0, & S(H_j) &= -H_j,
			\end{aligned}
		\end{equation}
	\end{subequations}
	respectively.
\end{definition}

Normally, one would augment the relations \eqref{eq:Uisl3rels} with Serre relations.  However, these become trivial when the quantum group parameter $q$ is specialised to $\ii$.

The subalgebra of $\qgrp$ generated by the ``simple root vectors'' $X_{\pm1}$ and $X_{\pm2}$ will be denoted by $\qgrpnilpm$.  As in \cite[Section~6.2.3]{KS}, the root vectors associated to $\pm\alpha_3$ may be chosen to be
\begin{equation}
	X_3 = -X_1 X_2 - \ii X_2 X_1 \qquad \text{and} \qquad X_{-3} = -X_{-2} X_{-1} + \ii X_{-1} X_{-2}.
\end{equation}
It is easy to check from \eqref{eq:Uisl3rels} that $X_{\pm3}^2 = 0$.  The following result is easy to establish.

\begin{proposition}\label{PBW}
	The set $\set{X_{\pm1}^{n_1} X^{n_3}_{\pm 3} X^{n_2}_{\pm 2}}{n_1, n_2, n_3 = 0, 1}$ is a vector space basis for $\qgrpnilpm$.
\end{proposition}

\subsection{Finite-dimensional weight $\qgrp$-modules}

To discuss the representation theory of $\qgrp$, it will be convenient to identify the generators $H_1$ and $H_2$ with the simple coroots of $\slthree$, so that $\spn \{H_1,H_2\} = \csub$.

\begin{definition}\label{wt}
	Given a $\qgrp$-module $V$, denote by $V(\lambda)$ the simultaneous eigenspace of $H_1$ and $H_2$ with weight $\lambda \in \csub^*$:
	\begin{equation}\label{wtspace}
		V(\lambda) = \set{v \in V}{H_j v = \lambda(H_j) v,\ j=1,2}.
	\end{equation}
	A $\qgrp$-module $V$ is called a \emph{weight module} if it is the direct sum of its weight spaces $V(\lambda)$, $\lambda \in \csub^*$, and every $K_{\gamma}$, $\gamma \in L$, acts on each $V(\lambda)$ as multiplication by $\ii^{\killing{\gamma}{\lambda}}$.  A non-zero element of $V(\lambda)$ is called a \emph{weight vector} of weight $\lambda$.  The category of all finite-dimensional weight $\qgrp$-modules will be denoted by $\qgrpcat$.
\end{definition}

Note that the category $\qgrpcat$ is insensitive to the choice of lattice $L$ made for $\qgrp$ (see \cref{Uisl3}), because the action of the $K_{\gamma}$, $\gamma \in L$, in $\qgrpcat$ is completely determined by the action of $H_1$ and $H_2$.

\begin{definition}
	The \emph{character} of a module $V$ in $\qgrpcat$ is the following formal series in $z$:
	\begin{equation} \label{eq:defchar}
		\ch{V} = \sum_{\lambda \in \csub^*} \dim V(\lambda) \, z^{\lambda}.
	\end{equation}
\end{definition}

\begin{definition}
	A weight vector in a $\qgrp$-module is said to be \emph{maximal} if it is annihilated by $X_1$ and $X_2$.  Any $\qgrp$-module generated by a single maximal vector is called a \emph{highest-weight module} and the weight of the generating maximal vector is called its \emph{highest weight}.
\end{definition}

As usual, one may define the Verma $\qgrp$-module $\verma{\lambda}$, $\lambda \in \csub^*$, as the (unique up to isomorphism) universal highest-weight module generated by a maximal vector of weight $\lambda$.  It has a unique irreducible quotient which we denote by $\irred{\lambda}$.  Moreover, $\verma{\lambda}$ is free as a $\qgrpnilm$-module so \cref{PBW} gives the basis
\begin{equation} \label{eq:VermaBasis}
	\set{X_{-1}^{n_1} X_{-3}^{n_3} X_{-2}^{n_2} m_{\lambda}}{n_1, n_2, n_3 = 0,1},
\end{equation}
where $m_{\lambda}$ is any vector of weight $\lambda$.  The character of a Verma module is thus given by
\begin{equation} \label{eq:VermaChar}
	\ch{\verma{\lambda}}
	= \sum_{n_1=0}^1 \sum_{n_2=0}^1 \sum_{n_3=0}^1 z^{\lambda - n_1 \alpha_1 - n_2 \alpha_2 - n_3 \alpha_3}
	= z^{\lambda} \prod_{i=1}^3 (1+z^{-\alpha_i}).
\end{equation}
It moreover follows that both $\verma{\lambda}$ and $\irred{\lambda}$ belong to the finite-dimensional weight category $\qgrpcat$, for all $\lambda \in \csub^*$.  As noted in \cite[Section~4.2]{R}, $\qgrpcat$ has enough projectives.  In particular, $\irred{\lambda}$ and $\verma{\lambda}$ have a common projective cover in $\qgrpcat$, which we denote by $\proj{\lambda}$.

The question of when $\verma{\lambda}$ is irreducible was decided in \cite{R} in terms of the notion of typical weights.  Let $\rho$ denote the Weyl vector of $\slthree$.

\begin{definition} \label{typ}
	For each $\lambda \in \csub^*$, let
	\begin{equation} \label{eq:lk}
		\lambda_j = \killing{\lambda+\rho}{\alpha_j}, \quad j=1,2,3,
	\end{equation}
	and note that $\lambda_3 = \lambda_1 + \lambda_2$.  The index $\lambda_j$ is said to be \emph{atypical} whenever $\lambda_j \in 2\ZZ+1$ and \emph{typical} otherwise.  If all of the $\lambda_j$, $j=1,2,3$, are typical, then the weight $\lambda \in \csub^*$ is said to be \emph{typical}.  For $n=0,1,2$, a weight $\lambda$ with $n$ atypical indices $\lambda_j$ is said to be \emph{atypical of degree $n$}.
\end{definition}

For convenience, we illustrate the atypical weights in \cref{fig:atypwts}.
\begin{figure}
	\centering
	\begin{tikzpicture}[scale=0.6]
		\begin{scope}[shift={(90:1)}] \foreach \n in {-3,-1,1,3} \node[wt] at (30:\n) {}; \end{scope}
		\begin{scope}[shift={(90:3)}] \foreach \n in {-3,-1,1} \node[wt] at (30:\n) {}; \end{scope}
		\begin{scope}[shift={(90:-1)}] \foreach \n in {-3,-1,1,3} \node[wt] at (30:\n) {}; \end{scope}
		\begin{scope}[shift={(90:-3)}] \foreach \n in {-1,1,3} \node[wt] at (30:\n) {}; \end{scope}
		\draw[very thick] (-150:4.5) -- (30:4.5);
		\begin{scope}[shift={(90:2)}] \draw[very thick] (-150:4.5) -- (30:2.5); \end{scope}
		\begin{scope}[shift={(90:4)}] \draw[very thick] (-150:4.5) -- (30:0.5); \end{scope}
		\begin{scope}[shift={(90:-2)}] \draw[very thick] (-150:2.5) -- (30:4.5); \end{scope}
		\begin{scope}[shift={(90:-4)}] \draw[very thick] (-150:0.5) -- (30:4.5); \end{scope}
		\draw[very thick] (-90:4.5) -- (90:4.5);
		\begin{scope}[shift={(30:2)}] \draw[very thick] (-90:4.5) -- (90:2.5); \end{scope}
		\begin{scope}[shift={(30:4)}] \draw[very thick] (-90:4.5) -- (90:0.5); \end{scope}
		\begin{scope}[shift={(30:-2)}] \draw[very thick] (-90:2.5) -- (90:4.5); \end{scope}
		\begin{scope}[shift={(30:-4)}] \draw[very thick] (-90:0.5) -- (90:4.5); \end{scope}
		\begin{scope}[shift={(90:1)}] \draw[very thick] (150:3.5) -- (-30:4.5); \end{scope}
		\begin{scope}[shift={(90:3)}] \draw[very thick] (150:1.5) -- (-30:4.5); \end{scope}
		\begin{scope}[shift={(90:-1)}] \draw[very thick] (150:4.5) -- (-30:3.5); \end{scope}
		\begin{scope}[shift={(90:-3)}] \draw[very thick] (150:4.5) -- (-30:1.5); \end{scope}
		\foreach \n in {-4,-3,...,4} \node[atyp] at (30:\n) {};
		\begin{scope}[shift={(90:1)}] \foreach \n in {-4,-2,0,2} \node[atyp] at (30:\n) {}; \end{scope}
		\begin{scope}[shift={(90:2)}] \foreach \n in {-4,-3,...,2} \node[atyp] at (30:\n) {}; \end{scope}
		\begin{scope}[shift={(90:3)}] \foreach \n in {-4,-2,0} \node[atyp] at (30:\n) {}; \end{scope}
		\begin{scope}[shift={(90:4)}] \foreach \n in {-4,-3,...,0} \node[atyp] at (30:\n) {}; \end{scope}
		\begin{scope}[shift={(90:-1)}] \foreach \n in {-2,0,2,4} \node[atyp] at (30:\n) {}; \end{scope}
		\begin{scope}[shift={(90:-2)}] \foreach \n in {-2,-1,...,4} \node[atyp] at (30:\n) {}; \end{scope}
		\begin{scope}[shift={(90:-3)}] \foreach \n in {0,2,4} \node[atyp] at (30:\n) {}; \end{scope}
		\begin{scope}[shift={(90:-4)}] \foreach \n in {0,1,...,4} \node[atyp] at (30:\n) {}; \end{scope}
	\end{tikzpicture}
\caption{The atypical weights in $\csub^*$, pictured with $0$ at the centre.  Grey dots indicate typical weights in $P$, black lines indicate atypical weights of degree $1$ and black dots are atypical weights of degree $2$.  The Weyl vector $\rho$ is the grey dot closest to $0$ along the ray $30^{\circ}$ to the right of the upwards vertical.} \label{fig:atypwts}
\end{figure}
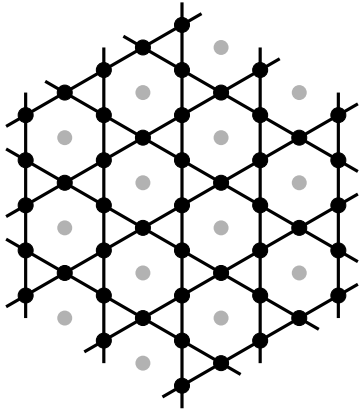

\begin{proposition}[{\cite[Propositions~4.4 and 4.6]{R}}] \label{irredM}
	$M_{\lambda}$ is irreducible if and only if $\lambda$ is typical.
\end{proposition}

We shall see in \cref{Verma,Projective} that the subquotient structure of $\verma{\lambda}$ and $\proj{\lambda}$ is determined by which of the $\lambda_j$, $j=1,2,3$, are atypical.

It is easy to see that two irreducible modules in $\qgrpcat$ are isomorphic if and only if they have the same character.  One might not expect this to remain true when ``irreducible'' is replaced by ``projective''.  For example, it is false for general Drinfeld--Jimbo quantum groups (at roots of unity).  However, unrolling allows us to consider additive module weights rather than just multiplicative ones (meaning $H_j$-eigenvalues rather than $K_{\gamma}$-eigenvalues).  This has the following consequence.

\begin{proposition}[{\cite[Proposition~4.13]{R}}] \label{projcharacter}
	If two projective modules in $\qgrpcat$ have the same character, then they are isomorphic.
\end{proposition}

Note that the category $\qgrpcat$ admits a character-preserving exact contravariant functor $V \mapsto \dual{V}$ analogous to the usual duality functor on the category of finite-dimensional modules over a Lie algebra (see \cite[Section~4.2]{R} for details).

\begin{definition} \label{dual}
	Let $\dual{\hphantom{\blank}}$ denote the dual defined on $\qgrpcat$ by the involutive antiautomorphism $S \circ \Omega$, where $S$ is the antipode of $\qgrp$ and $\Omega$ is the automorphism of $\qgrp$ defined by
	\begin{equation} \label{eqw}
		\Omega(X_{\pm j}) = X_{\mp j}, \quad \Omega(K_{\gamma}) = K_{-\gamma}, \quad \Omega(H_j) = -H_j, \qquad j=1,2,\ \gamma \in L.
	\end{equation}
\end{definition}

It is clear that irreducible modules in $\qgrpcat$ are self-dual with respect to this \lcnamecref{dual}.

\begin{theorem}[{\cite[Theorem~4.14]{R}}] \label{selfdual}
	With respect to $\dual{\hphantom{\blank}}$, $\proj{\lambda}$ is self-dual: $\dproj{\lambda} \cong \proj{\lambda}$.
\end{theorem}

This result heavily constrains the Loewy diagram of $P_{\lambda}$ and will be used extensively in \cref{Projective}.

\subsection{Filtrations and BGG reciprocity} \label{subsectionfilt}

Recall that a \emph{filtration}, or \emph{series}, of a module $V$ is a strictly increasing family of submodules ordered by inclusion:
\begin{equation} \label{series}
	0 = V_0 \subset V_1 \subset \dots \subset V_{n-1} \subset V_n = V.
\end{equation}
The integer $n$ is called the \emph{length} of the series.  Such a series is called a \emph{composition series} of $V$ if the quotient of each non-zero submodule in the filtration by its predecessor is irreducible: for each $j=1,\dots,n$, $V_j / V_{j-1} \cong \irred{\lambda_j}$ for some $\lambda_j \in \csub^*$.  Similarly, a series is called a \emph{standard filtration} (or \emph{Verma flag}) of $V$ if the successive quotients are Verma modules: for each $j=1,\dots,n$, $V_j / V_{j-1} \cong \verma{\lambda_j}$ for some $\lambda_j \in \csub^*$.

Every module in $\qgrpcat$ admits a composition series.  While this series is not unique, they all have the same length.  Moreover, the set of isomorphism classes of the irreducible quotients (the \emph{composition factors}) is uniquely fixed by the isomorphism class of the module.  The set of isomorphism classes of the quotients (the \emph{Verma factors}) of a standard filtration is likewise unique, though not every module in $\qgrpcat$ admits such a series.

Given a module $V \in \qgrpcat$ which admits a standard filtration, denote by $(V:\verma{\lambda})$ the multiplicity with which $\verma{\lambda}$ appears in the set of Verma factors of any standard filtration of $V$.  Similarly, denote by $[V:\irred{\lambda}]$ the multiplicity with which $\irred{\lambda}$ appears in the set of irreducible quotients of any composition series of $V$.

\begin{proposition}[{\cite[Proposition~4.12]{R}}] \label{bgg}
	BGG reciprocity holds for the category $\qgrpcat$. In other words, $(\proj{\lambda}:\verma{\mu})=[\verma{\mu}:\irred{\lambda}]$ for all $\lambda, \mu \in \csub^*$.
\end{proposition}

We shall also be interested in a different type of filtration/series in which the successive quotients are required to be completely reducible.  A composition series satisfies this property, but there are usually other examples.

\begin{definition}
	A \emph{Loewy series} of $V \in \qgrpcat$ is one with minimal length amongst all series in which the successive quotients are completely reducible.
\end{definition}

There are two standard examples of Loewy series.  First, the \emph{radical series}
\begin{equation}
	0 = \rad^n V \subset \rad^{n-1} V \subset \dots \subset \rad^1 V \subset \rad^0 V = V
\end{equation}
of $V \in \qgrpcat$ is defined inductively by setting $\rad^0 V = V$ and $\rad^j V = \rad (\rad^{j-1} V)$, for $j\ge1$.  Here, the \emph{radical} of $W \in \qgrpcat$, denoted by $\rad W$, is the smallest submodule such that the quotient is completely reducible.  Equivalently, $\rad W$ is the intersection of the maximal proper submodules of $W$.  The successive semisimple quotients are denoted by $\rad_j V = \rad^{j-1} V \big/ \rad^j V$.

The second example is the \emph{socle series}
\begin{equation}
	0 = \soc^0 V \subset \soc^1 V \subset \dots \subset \soc^{n-1} V \subset \soc^n V = V
\end{equation}
of $V \in \qgrpcat$, this being the series defined inductively by taking $\soc^0 V = 0$ and letting $\soc^j V$ be the unique submodule of $V$ satisfying $\soc(V \big/ \soc^{j-1} V) = \soc_j V$, where $\soc_j V = \soc^j V \big/ \soc^{j-1} V$.  Here, the \emph{socle} of $W \in \qgrpcat$, denoted by $\soc W$, is the largest completely reducible submodule.  Equivalently, it is the sum of the irreducible submodules of $W$.

Note that the radical and socle series of $V \in \qgrpcat$ are uniquely specified.  The dual nature of the definitions suggests that they might actually coincide, however this need not be the case.  Indeed, the radical series imposes maximality on quotients, starting from $V$ and working down to $0$, while the socle series imposes maximality on submodules, starting from $0$ and working up to $V$.  The upshot is that any Loewy series \eqref{series} must therefore lie between these two extremes:
\begin{equation}
	\rad^{n-k} V \subseteq V_k \subseteq \soc^k V, \qquad k=0,1,\dots,n.
\end{equation}
Consequently, $V$ admits a unique Loewy series if and only if the radical and socle series coincide.

In general, it is difficult to know if Loewy series are unique.  However, there are easily derived criteria if the length $n$ of the series is sufficiently small.  To see this, note that:
\begin{itemize}
	\item If $\soc_1 V = \soc V$ is irreducible, then $\rad^{n-1} V = \soc^1 V$.
	\item If the \emph{head} $\rad_1 V = V/\rad V$ of $V$ is irreducible, then $\rad^1 V = \soc^{n-1} V$.
\end{itemize}
It follows, when $n=2$, that all Loewy series coincide if either the socle or head of $V$ is irreducible.  Similarly, when $n=3$, all Loewy series coincide if both the socle and head are irreducible.

As our category $\qgrpcat$ admits a duality functor, one is naturally led to ask for a relationship between the Loewy series of a module $V$ and its dual $\dual{V}$.  For the radical and socle series, this is the content of the \emph{Landrock lemma} \cite[Lemma~8.4i)]{L} which states that
\begin{equation}
	\soc^j \dual{V} \cong \outdual{\frac{V}{\rad^j V}} \qquad \text{and} \qquad \rad^j \dual{V} \cong \outdual{\frac{V}{\soc^j V}}.
\end{equation}
The successive semisimple quotients are then related as follows:
\begin{align} \label{eq:socq=radq}
	\outdual{\soc_j \dual{V}}
	&= \outdual{\frac{\soc^j \dual{V}}{\soc^{j-1} \dual{V}}}
	\cong \outdual{\outdual{\frac{V}{\rad^j V}} \middle/ \outdual{\frac{V}{\rad^{j-1} V}}} \notag \\
	&\cong \ker \left( \frac{V}{\rad^j V} \onto \frac{V}{\rad^{j-1} V} \right)
	\cong \frac{\rad^{j-1} V}{\rad^j V}
	= \rad_j V.
\end{align}
Since the irreducible $\qgrp$-modules in $\qgrpcat$ are self-dual, we therefore have $\rad_j V \cong \soc_j \dual{V}$ for all $j=1,\dots,n$.

\begin{definition} \label{Loewyd}
	Given a Loewy series \eqref{series} of $V \in \qgrpcat$, the associated \emph{Loewy diagram} is a picture of the composition factors of $V$ arranged in horizontal layers that reflect the successive quotients of the series.  In particular, the $j$-th layer from the bottom consists of the irreducible direct summands of the completely reducible module $V_j/V_{j-1}$.
\end{definition}

We shall generally annotate a Loewy diagram with arrows pointing from a composition factor in the $j$-th layer to one in the $(j-1)$-th layer whenever these factors correspond to a non-split subquotient of $V_j / V_{j-2}$.  This additional information helps to describe the submodule structure of $V$, though it need not fix it completely.

Since the irreducibles $\irred{\lambda}$ are self-dual, \eqref{eq:socq=radq} implies that the Loewy diagram for the radical (socle) series of $\dual{V}$ (\emph{sans} arrows) is obtained from that of the socle (radical) series of $V \in \qgrpcat$ by turning it upside-down.  In particular, the radical and socle series of the projectives $\proj{\lambda}$ are related in this way, by \cref{selfdual}.

\section{Representation theory of $\qgrp$} \label{qgrpmod}

The Verma modules of $\qgrp$ are only $8$-dimensional, implying that their submodule structures may be readily computed.  This is the topic of the present \lcnamecref{qgrpmod}, along with the more intricate determination of the structures of the indecomposable projective modules and the decomposition of all tensor products of irreducibles in $\qgrpcat$.

\subsection{Verma modules and irreducible quotients} \label{Verma}

Recall from \cref{irredM} that the irreducibility of the Verma module $M_{\lambda}$ is determined by the typicality of the scalars $\lambda_j=\langle \lambda+\rho ,\alpha_j \rangle$, $j=1,2,3$.  Here, the structures of the reducible Verma modules are determined and the characters of their irreducible quotients computed.  Note that at most two of the $\lambda_j$ can be atypical (the third index must then lie in $2\ZZ$).

\begin{proposition} \label{irreduciblecharacter}
	\leavevmode
	\begin{enumerate}
		\item \label{it:typ}
		If $\lambda$ is typical, so $\lambda_1, \lambda_2, \lambda_3 \notin 2\ZZ+1$, then $\dim \irred{\lambda} = 8$ and
		\begin{equation}
			\ch{\irred{\lambda}} = z^{\lambda} + z^{\lambda-\alpha_1} + z^{\lambda-\alpha_2} + 2z^{\lambda-\alpha_3} + z^{\lambda-\alpha_1-\alpha_3} + z^{\lambda-\alpha_2-\alpha_3} + z^{\lambda-2\alpha_3}.
		\end{equation}
		\item \label{it:atyp1}
		If $\lambda$ is atypical of degree $1$, then $\dim \irred{\lambda} = 4$ and
		\begin{equation}
			\ch{\irred{\lambda}} = z^{\lambda} + z^{\lambda-\alpha_j} + z^{\lambda-\alpha_k} + z^{\lambda-\alpha_j-\alpha_k},
		\end{equation}
		where $j,k \in \{1,2,3\}$ are chosen so that $\lambda_j$ and $\lambda_k$ are typical.
		\item \label{it:atyp2}
		If $\lambda$ is atypical of degree $2$ and $\lambda_3$ is atypical, then $\dim \irred{\lambda} = 3$ and
		\begin{equation}
			\ch{\irred{\lambda}} = z^{\lambda} + z^{\lambda-\alpha_j} + z^{\lambda-\alpha_3},
		\end{equation}
		where $j \in \{1,2\}$ is chosen so that $\lambda_j$ is typical.
		\item \label{it:atyp2'}
		Otherwise, $\lambda$ is atypical of degree $2$ and $\lambda_3$ is typical.  In this case, $\dim \irred{\lambda} = 1$ and
		\begin{equation}
			\ch{\irred{\lambda}} = z^{\lambda}.
		\end{equation}
	\end{enumerate}
\end{proposition}
\begin{proof}
	Case \ref{it:typ} follows immediately from \cref{irredM} and the Verma character \eqref{eq:VermaChar}.  We therefore start with \ref{it:atyp2'}, which assumes that $\lambda_1, \lambda_2 \in 2\ZZ+1$.  If $v$ is a vector of weight $\lambda$ and $j,j'=1,2$, then
	\begin{equation} \label{eq:X+kX-km}
		\left[ X_j, X_{-j'} \right] v = \delta_{j,j'} \frac{K_j - K_j^{-1}}{2\ii} v = \delta_{j,j'} \frac{\ii^{\killing{\lambda}{\alpha_j}} - \ii^{-\killing{\lambda}{\alpha_j}}}{2\ii} v = -\frac{1}{2} \delta_{j,j'} \left(\ii^{\lambda_j} + \ii^{-\lambda_j} \right) v.
	\end{equation}
	Taking $v = m_{\lambda}$, a generating maximal vector in $\verma{\lambda}$, it follows that $X_j X_{-j'} m_{\lambda}$ is always $0$, for $j \ne j'$, and $X_j X_{-j} m_{\lambda} = 0$ if and only if $\lambda_j \in 2\ZZ+1$.  Thus, $X_{-1} m_{\lambda}$ and $X_{-2} m_{\lambda}$ belong to the maximal proper submodule $\maxi{\lambda}$ of $\verma{\lambda}$, when $\lambda$ is typical, hence so does $X_{-3} m_{\lambda}$.  $\irred{\lambda}$ is therefore one-dimensional, spanned by the equivalence class of $m_{\lambda}$, as claimed.

	In case \ref{it:atyp2}, take (without loss of generality) $\lambda_1 \in 2\ZZ$ and $\lambda_2, \lambda_3 \in 2\ZZ+1$.  Then, $X_{-2} m_{\lambda} \in \verma{\lambda}$ is maximal, as above, and so it and its descendants lie in $\maxi{\lambda}$.  We claim that $X_{-1} X_{-3} m_{\lambda} \in \maxi{\lambda}$ as well.  Since $X_2 X_{-1} X_{-3} m_{\lambda}$ vanishes for weight reasons, this follows from
	\begin{align}
		X_1 X_{-1} X_{-3} m_{\lambda}
		&= -X_1 X_{-1} X_{-2} X_{-1} m_{\lambda}
		= -[X_1, X_{-1}] X_{-2} X_{-1} m_{\lambda} - X_{-1} X_{-2} [X_1, X_{-1}] m_{\lambda} \notag \\
		&= \frac{1}{2} \left( \ii^{\lambda_1-1} + \ii^{-\lambda_1+1} \right) X_{-2} X_{-1} m_{\lambda} + \frac{1}{2} \left( \ii^{\lambda_1} + \ii^{-\lambda_1} \right) X_{-1} X_{-2} m_{\lambda} \notag \\
		&= \frac{1}{2} \left( \ii^{\lambda_1} + \ii^{-\lambda_1} \right) X_{-1} X_{-2} m_{\lambda},
	\end{align}
	which is non-zero since $\lambda_1 \in 2\ZZ$.  Similar calculations establish that $X_2$ sends $X_{-3} m_{\lambda}$ to a non-zero multiple of $X_{-1} m_{\lambda}$, which is in turn sent to a non-zero multiple of $m_{\lambda}$ by $X_1$.  These states are therefore not in $\maxi{\lambda}$, hence their images span $\irred{\lambda}$ and $\dim \irred{\lambda} = 3$.

	It remains to prove \ref{it:atyp1}.  Suppose first that $\lambda_2 \in 2\ZZ+1$, so that $\lambda_1, \lambda_3 \in \CC \setminus \ZZ$.  The previous case shows that $X_{-2} m_{\lambda}$ and its descendants are in $\maxi{\lambda}$.  This time, $X_{-1} X_{-3} m_{\lambda}$ is not because
	\begin{equation}
		X_1 X_{-1} X_{-3} m_{\lambda} = \frac{1}{2} \left( \ii^{\lambda_1-1} + \ii^{-(\lambda_1-1)} \right) X_{-2} X_{-1} m_{\lambda} \quad \pmod{\maxi{\lambda}},
	\end{equation}
	which is non-zero as $\lambda_1 \notin 2\ZZ$.  Moreover, $X_{-2} X_{-1} m_{\lambda} \notin \maxi{\lambda}$ as $X_2$ sends it to a non-zero multiple of $X_{-1} m_{\lambda}$ and $X_1$ sends the latter to $m_{\lambda}$.  $\dim \irred{\lambda}$ is thus $4$.

	As $\lambda_1 \in 2\ZZ+1$ obviously follows in the same way, the last case to consider is $\lambda_3 \in 2\ZZ+1$ and $\lambda_1, \lambda_2 \in \CC \setminus \ZZ$.  As above, it easily follows that $X_{-1} m_{\lambda}, X_{-2} m_{\lambda} \notin \maxi{\lambda}$.  However there is a vector of weight $\lambda - \alpha_3$, unique up to scalar multiples, that is maximal hence in $\maxi{\lambda}$:
	\begin{equation}
		v = \left( \ii^{\lambda_1} + \ii^{-\lambda_1} \right) X_{-1} X_{-2} m_{\lambda} - \left( \ii^{\lambda_1+1} + \ii^{-\lambda_1-1} \right) X_{-2} X_{-1} m_{\lambda}.
	\end{equation}
	Checking $X_1 v = 0$ is straightforward and an easy computation gives
	\begin{equation}
		X_2 v = -\left( \ii^{\lambda_3} + \ii^{-\lambda_3} \right) X_{-1} m_{\lambda},
	\end{equation}
	which vanishes because $\lambda_3 \in 2\ZZ+1$.  Now one checks that $v$ generates a four-dimensional submodule of $\verma{\lambda}$ (which is therefore $\maxi{\lambda}$).  It follows that $\dim \irred{\lambda} = 4$, completing the proof.
\end{proof}

It is very convenient that the dimension of $\irred{\lambda}$ completely determines the four different cases in \cref{irreduciblecharacter}.  This dimension also completely determines the Loewy diagrams of the Verma modules of $\qgrp$.

\begin{theorem} \label{LoewyM}
	The Loewy series of any given Verma module $\verma{\lambda}$ of $\qgrp$ all coincide.  When $\verma{\lambda}$ is reducible, its (unique) Loewy diagram is given in \cref{fig:LoewyM}.
\end{theorem}

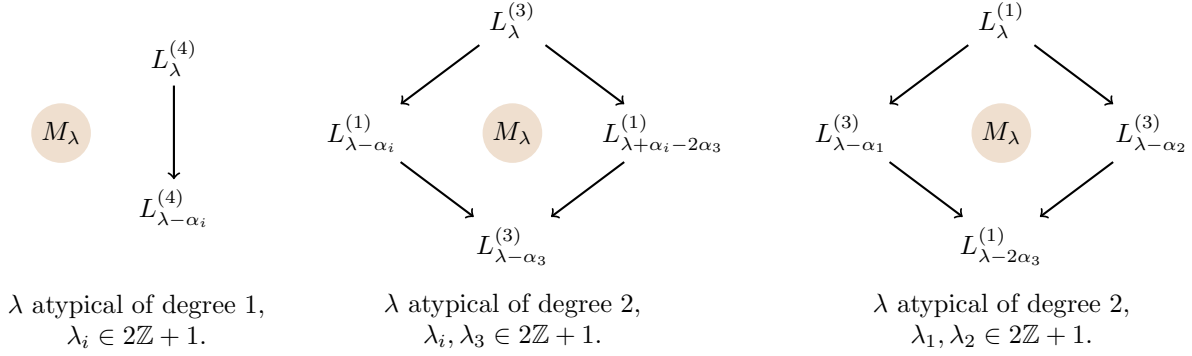
\begin{figure}
	\centering
	\begin{tikzpicture}[->, thick, scale=1]
		\begin{scope}[shift={(0,0)}]
			\node (top) at (1,1) {$\irr{4}{\lambda}$};
			\node (bottom) at (1,-1) {$\irr{4}{\lambda-\alpha_i}$};
			\draw (top) -- (bottom);
			\node[nom] at (-0.5,0) {$\verma{\lambda}$};
			\node[align=center] at (0.5,-2.5) {$\lambda$ atypical of degree $1$, \\ $\lambda_i \in 2\ZZ+1$.};
		\end{scope}
		\begin{scope}[shift={(5.5,0)}]
			\node (top) at (0,1.5) {$\irr{3}{\lambda}$};
			\node (middleI) at (-2,0) {$\irr{1}{\lambda-\alpha_i}$};
			\node (middleII) at (2,0) {$\irr{1}{\lambda+\alpha_i-2\alpha_3}$};
			\node (bottom) at (0,-1.5) {$\irr{3}{\lambda-\alpha_3}$};
			\draw (top) -- (middleI);
			\draw (top) -- (middleII);
			\draw (middleI) -- (bottom);
			\draw (middleII) -- (bottom);
			\node[nom] at (0,0) {$\verma{\lambda}$};
			\node[align=center] at (0,-2.5) {$\lambda$ atypical of degree $2$, \\ $\lambda_i, \lambda_3 \in 2\ZZ+1$.};
		\end{scope}
		\begin{scope}[shift={(12,0)}]
			\node (top) at (0,1.5) {$\irr{1}{\lambda}$};
			\node (middleI) at (-2,0) {$\irr{3}{\lambda-\alpha_1}$};
			\node (middleII) at (2,0) {$\irr{3}{\lambda-\alpha_2}$};
			\node (bottom) at (0,-1.5) {$\irr{1}{\lambda-2\alpha_3}$};
			\draw (top) -- (middleI);
			\draw (top) -- (middleII);
			\draw (middleI) -- (bottom);
			\draw (middleII) -- (bottom);
			\node[nom] at (0,0) {$\verma{\lambda}$};
			\node[align=center] at (0,-2.5) {$\lambda$ atypical of degree $2$, \\ $\lambda_1, \lambda_2 \in 2\ZZ+1$.};
		\end{scope}
	\end{tikzpicture}
\caption{The Loewy diagrams of the reducible Verma modules $\verma{\lambda}$ of $\qgrp$.  Here, the superscript on a composition factor indicates its dimension.} \label{fig:LoewyM}
\end{figure}

\begin{proof}
	Suppose first that $\dim \irred{\lambda} = 4$, so that there is a unique $i=1,2,3$ satisfying $\lambda_i \in 2\ZZ+1$ and thus $\lambda_j \in \CC \setminus \ZZ$ for all $j \ne i$.  The proof of \cref{irreduciblecharacter} established that the maximal proper submodule $\maxi{\lambda} = \rad \verma{\lambda}$ is generated by a maximal vector of weight $\mu = \lambda - \alpha_i$.  Because $\mu_i = \lambda_i - 2 \in 2\ZZ+1$ and $\mu_j = \lambda_j \pm 1 \in \CC \setminus \ZZ$, for $j \ne i$, $\dim \irred{\mu}$ is also $4$.  This implies that $\rad \verma{\lambda} \cong \irred{\mu}$, fixing the Loewy diagram of the radical series.  Because this series has length $2$ and the head is irreducible (true for all Verma modules), all Loewy series coincide and so the Loewy diagram is unique.

	Next, take $\dim \irred{\lambda} = 1$, so $\lambda_1, \lambda_2 \in 2\ZZ+1$.  In this case, the proof of \cref{irreduciblecharacter} established that $\rad \verma{\lambda}$ is generated by the maximal vectors $X_{-i} m_{\lambda}$ of weights $\lambda - \alpha_i$, $i=1,2$.  Obviously, the submodule $V^i$ generated by $X_{-i} m_{\lambda}$ does not contain the other.  Moreover, $\dim \irred{\lambda-\alpha_i} = 3$, so the dimensions of the composition factors found thus far sum to $7$.  There is therefore one remaining and its dimension must be $1$.  It is easy to check that it is $\irred{\lambda-2\alpha_3}$, generated by the maximal vector $(X_{-1} X_{-2})^2 m_{\lambda} = (X_{-2} X_{-1})^2 m_{\lambda}$, which is clearly contained in $V^1 \cap V^2$.  This forces $\rad^2 \verma{\lambda} \cong \irred{\lambda-2\alpha_3}$ and $\rad^1 \verma{\lambda} \cong V^1 + V^2$, hence $\rad_2 \verma{\lambda} \cong \irred{\lambda-\alpha_1} \oplus \irred{\lambda-\alpha_2}$.  The length of the radical series is thus $3$ and it has an irreducible head.  The socle is obviously the irreducible $\irred{\lambda-2\alpha_3}$, so again there is a unique Loewy diagram.

	It remains to consider $\dim \irred{\lambda} = 3$.  Without loss of generality, take $\lambda_1 \in 2\ZZ$ and $\lambda_2 \in 2\ZZ+1$.  Then, the proof of \cref{irreduciblecharacter} shows that $\rad \verma{\lambda}$ is the $5$-dimensional submodule generated by the maximal vector $X_{-2} m_{\lambda}$ and the non-maximal vector $X_{-1} X_{-3} m_{\lambda}$.  Under the action of $X_1$, the latter is mapped to a non-zero multiple of $X_{-1} X_{-2} m_{\lambda}$ which is easily checked to be maximal.  We have thus identified three composition factors of $\rad \verma{\lambda}$: $\irred{\lambda-\alpha_2}$, $\irred{\lambda-\alpha_1-\alpha_3}$ and $\irred{\lambda-\alpha_3}$.  Their dimensions are $1$, $1$ and $3$ by \cref{irreduciblecharacter}, so there are no more factors.  This proves that $\rad^2 \verma{\lambda} \cong \irred{\lambda - \alpha_3}$ and $\rad_2 \verma{\lambda} \cong \irred{\lambda - \alpha_2} \oplus \irred{\lambda - \alpha_1 - \alpha_3}$.  The length of the radical series is thus $3$ and $\soc \verma{\lambda} \cong \irred{\lambda-\alpha_3}$, so $\verma{\lambda}$ once more has a unique Loewy diagram.
\end{proof}

A useful consequence of \cref{LoewyM} is the classification of indecomposable extensions of irreducibles in $\qgrpcat$.  Self-extensions are ruled out as they are necessarily non-weight, hence not in $\qgrpcat$.  Such an extension is therefore a highest-weight module, hence a Verma quotient, or the dual of a highest-weight module.  Either way, the extension is unique up to isomorphism.

\begin{corollary} \label{cor:irredext}
	A non-split extension between $\irred{\lambda}$ and $\irred{\mu}$ exists, and is then unique up to isomorphism, if and only if one of the following conditions is satisfied:
	\begin{itemize}
		\item $\lambda$ is atypical of degree $1$ and $\mu = \lambda \pm \alpha_i$, where $i=1,2,3$ satisfies $\lambda_i \in 2\ZZ+1$.
		\item $\lambda$ is atypical of degree $2$ with $\lambda_1, \lambda_2 \in 2\ZZ+1$ and $\mu = \lambda \pm \alpha_1, \lambda \pm \alpha_2, \lambda-\alpha_1+2\alpha_3, \lambda-\alpha_2+2\alpha_3$.
		\item $\lambda$ is atypical of degree $2$ with $\lambda_3 \in 2\ZZ+1$ and $\mu = \lambda \pm \alpha_i, \lambda+\alpha_i-2\alpha_3$, where $i=1,2$ satisfies $\lambda_i \in 2\ZZ+1$.
	\end{itemize}
\end{corollary}

The blocks of $\qgrp$ are thus as follows.  If $\lambda$ is typical, then $\irred{\lambda}$ generates a semisimple block in which it is the only irreducible.  There are three one-parameter families of blocks corresponding to $\lambda$ atypical of degree $1$, one for each positive root $\alpha_i$.  The parameter may be taken as either of the typical $\lambda_i \in \CC \setminus \ZZ$.  Finally, there is a unique block for $\lambda$ atypical of degree $2$.  This is the principal block as it includes $\irred{0}$.

\subsection{Projective modules} \label{Projective}

We turn to the Loewy diagrams of the projective indecomposables $\proj{\lambda}$.  These are almost completely determined by combining the Loewy diagrams of the Verma modules with their standard (and costandard) filtrations.  The following \lcnamecref{Loewylemma} will help to complete the picture.

\begin{lemma} \label{Loewylemma}
	Given $\lambda_i, \lambda_3 \in 2\ZZ+1$, for some $i=1,2$, there does not exist a $\qgrp$-module whose Loewy diagram is the form
	\begin{equation} \label{noloewy}
		\begin{tikzpicture}[->, thick, scale=0.6, baseline=(l.base)]
			\node (t) at (0,2) {$\irr{3}{\lambda}$};
			\node (l) at (-2,0) {$\irr{1}{\lambda+\alpha_i}$};
			\node (r) at (2,0) {$\irr{1}{\lambda+\alpha_i-2\alpha_3}$};
			\node (b) at (0,-2) {$\irr{3}{\lambda}$};
			\draw (t) -- (l);
			\draw (t) -- (r);
			\draw (l) -- (b);
			\draw (r) -- (b);
		\end{tikzpicture}
		\ .
	\end{equation}
\end{lemma}
\begin{proof}
	Suppose that $V$ is such a module, with $i=2$ say (the argument for $i=1$ is identical).  Then,
	\begin{equation}
		\ch{V} = z^{\lambda+\alpha_2} + 2 \left( z^{\lambda} + z^{\lambda-\alpha_1}+z^{\lambda-\alpha_3} \right) + z^{\lambda+\alpha_2-2\alpha_3},
	\end{equation}
	by \cref{irreduciblecharacter}.  Let $v \in V$ be cyclic of weight $\lambda$.  Then, the arrow from $\irred{\lambda}$ to $\irred{\lambda+\alpha_2}$ in \eqref{noloewy} means that $X_2 v \ne 0$ and that from $\irred{\lambda+\alpha_2}$ to $\irred{\lambda}$ means that $w = X_{-2} X_2 v$ is linearly independent of $v$ (it generates the socle of $V$).  However, $X_{-2} v = 0$ since $\lambda - \alpha_2$ is not a weight of $V$, hence
	\begin{equation}
		w = [X_{-2},X_2] v = -\frac{K_2-K_2^{-1}}{2\ii} v
	\end{equation}
	is proportional to $v$, a contradiction.
\end{proof}

\begin{theorem} \label{thm:LoewyP}
	If $\lambda$ is typical, then $\proj{\lambda} =  \verma{\lambda} = \irred{\lambda}$.  Otherwise, the projective $\proj{\lambda}$ has a unique Loewy diagram which is presented, along with its character as a sum of Verma characters, in \cref{fig:LoewyP}.
\end{theorem}

\begin{figure}
	\centering
	\begin{tikzpicture}[->, thick, scale=0.72]
		\node (t) at (0,2) {$\irr{4}{\lambda}$};
		\node (l) at (2,0) {$\irr{4}{\lambda-\alpha_i}$};
		\node (r) at (-2,0) {$\irr{4}{\lambda+\alpha_i}$};
		\node (b) at (0,-2) {$\irr{4}{\lambda}$};
		\draw[verma] (t) -- (l);
		\draw[dverma] (t) -- (r);
		\draw[dverma] (l) -- (b);
		\draw[verma] (r) -- (b);
			\node[nom] at (-2.5,2) {$\pro{16}{\lambda}$};
			\node[align=center] at (0,-3.25) {$\lambda$ atypical of degree $1$, $\lambda_i \in 2\ZZ+1$, \\ $\ch{\proj{\lambda}} = \ch{\verma{\lambda+\alpha_i}} + \ch{\verma{\lambda}}$.};
	\end{tikzpicture}
	\\ \bigskip
	\begin{tikzpicture}[->, thick, scale=0.72]
			\node (1) at (0,4) [] {$\irr{3}{\lambda}$};
			\node (21) at (-4,2) [] {$\irr{1}{\lambda+\alpha_i}$};
			\node (22) at (0,2) [] {$\irr{1}{\lambda-\alpha_i}$};
			\node (23) at (4,2) [] {$\irr{1}{\lambda+\alpha_i-2\alpha_3}$};
			\node (31) at (-6,0) [] {$\irr{3}{\lambda+\alpha_3}$};
			\node (32) at (-2,0) [] {$\irr{3}{\lambda+2\alpha_i-\alpha_3}$};
			\node (33) at (2,0) [] {$\irr{3}{\lambda}$};
			\node (34) at (6,0) [] {$\irr{3}{\lambda-\alpha_3}$};
			\node (41) at (-4,-2) [] {$\irr{1}{\lambda+\alpha_i}$};
			\node (42) at (0,-2) [] {$\irr{1}{\lambda-\alpha_i}$};
			\node (43) at (4,-2) [] {$\irr{1}{\lambda+\alpha_i-2\alpha_3}$};
			\node (5) at (0,-4) [] {$\irr{3}{\lambda}$};
			\draw[verma] (1) -- (22);
			\draw[verma] (1) -- (23);
			\draw[verma] (22) -- (34);
			\draw[verma] (23) -- (34);
			\draw[verma] (21) -- (32);
			\draw[verma] (21) -- (33);
			\draw[verma] (32) -- (43);
			\draw[verma] (33) -- (43);
			\draw[verma] (31) -- (41);
			\draw[verma] (31) -- (42);
			\draw[verma] (41) -- (5);
			\draw[verma] (42) -- (5);
			\draw[dverma] (1) -- (21);
			\draw[dverma] (21) -- (31);
			\draw[dverma] (22) -- (31);
			\draw[dverma] (23) -- (32);
			\draw[dverma] (23) -- (33);
			\draw[dverma] (32) -- (41);
			\draw[dverma] (33) -- (41);
			\draw[dverma] (34) -- (42);
			\draw[dverma] (34) -- (43);
			\draw[dverma] (43) -- (5);
			\draw[other] (22) -- (33);
			\draw[other] (33) -- (42);
			\node[nom] at (-6,4) {$\pro{24}{\lambda}$};
			\node[align=center] at (0,-5.25) {$\lambda$ atypical of degree $2$, $\lambda_i, \lambda_3 \in 2\ZZ+1$, \\ $\ch{\proj{\lambda}} = \ch{\verma{\lambda+\alpha_3}} + \ch{\verma{\lambda+\alpha_i}} + \ch{\verma{\lambda}}$.};
		\end{tikzpicture}
	\\ \bigskip
	\begin{tikzpicture}[->, thick, xscale=1.1, yscale=0.72]
			\node (1) at (0,4) [] {$\irr{1}{\lambda}$};
			\node (21) at (-5,2) [] {$\irr{3}{\lambda+\alpha_2+\alpha_3}$};
			\node (22) at (-3,2) [] {$\irr{3}{\lambda+\alpha_1+\alpha_3}$};
			\node (23) at (-1,2) [] {$\irr{3}{\lambda+\alpha_2}$};
			\node (24) at (1,2) [] {$\irr{3}{\lambda+\alpha_1}$};
			\node (25) at (3,2) [] {$\irr{3}{\lambda-\alpha_1}$};
			\node (26) at (5,2) [] {$\irr{3}{\lambda-\alpha_2}$};
			\node (31) at (-6,0) [] {$\irr{1}{\lambda+2\alpha_3}$};
			\node (32) at (-4,0) [] {$\irr{1}{\lambda+2\alpha_2}$};
			\node (33) at (-2,0) [] {$\irr{1}{\lambda+2\alpha_1}$};
			\node (34) at (0,0) [] {${\irr{1}{\lambda}}{}^{\oplus4}$};
			\node (35) at (2,0) [] {$\irr{1}{\lambda-2\alpha_1}$};
			\node (36) at (4,0) [] {$\irr{1}{\lambda-2\alpha_2}$};
			\node (37) at (6,0) [] {$\irr{1}{\lambda-2\alpha_3}$};
			\node (41) at (-5,-2) [] {$\irr{3}{\lambda+\alpha_2+\alpha_3}$};
			\node (42) at (-3,-2) [] {$\irr{3}{\lambda+\alpha_1+\alpha_3}$};
			\node (43) at (-1,-2) [] {$\irr{3}{\lambda+\alpha_2}$};
			\node (44) at (1,-2) [] {$\irr{3}{\lambda+\alpha_1}$};
			\node (45) at (3,-2) [] {$\irr{3}{\lambda-\alpha_1}$};
			\node (46) at (5,-2) [] {$\irr{3}{\lambda-\alpha_2}$};
			\node (5) at (0,-4) [] {$\irr{1}{\lambda}$};
			\draw[verma] (1) -- (25);
			\draw[verma] (1) -- (26);
			\draw[verma] (25) -- (37);
			\draw[verma] (26) -- (37);
			\draw[verma] (21) -- (32);
			\draw[verma] (21) -- (34);
			\draw[verma] (32) -- (43);
			\draw[verma] (34) -- (43);
			\draw[verma] (22) -- (33);
			\draw[verma] (22) -- (34);
			\draw[verma] (33) -- (44);
			\draw[verma] (34) -- (44);
			\draw[verma] (23) -- (34);
			\draw[verma] (23) -- (35);
			\draw[verma] (34) -- (45);
			\draw[verma] (35) -- (45);
			\draw[verma] (24) -- (34);
			\draw[verma] (24) -- (36);
			\draw[verma] (34) -- (46);
			\draw[verma] (36) -- (46);
			\draw[verma] (31) -- (41);
			\draw[verma] (31) -- (42);
			\draw[verma] (41) -- (5);
			\draw[verma] (42) -- (5);
			\draw[dverma] (1) -- (21);
			\draw[dverma] (1) -- (22);
			\draw[dverma] (21) -- (31);
			\draw[dverma] (22) -- (31);
			\draw[dverma] (23) -- (32);
			\draw[dverma] (32) -- (41);
			\draw[dverma] (34) -- (41);
			\draw[dverma] (24) -- (33);
			\draw[dverma] (33) -- (42);
			\draw[dverma] (34) -- (42);
			\draw[dverma] (25) -- (34);
			\draw[dverma] (25) -- (35);
			\draw[dverma] (35) -- (43);
			\draw[dverma] (26) -- (34);
			\draw[dverma] (26) -- (36);
			\draw[dverma] (36) -- (44);
			\draw[dverma] (37) -- (45);
			\draw[dverma] (37) -- (46);
			\draw[dverma] (45) -- (5);
			\draw[dverma] (46) -- (5);
			\draw[other] (1) -- (23);
			\draw[other] (1) -- (24);
			\draw[other] (43) -- (5);
			\draw[other] (44) -- (5);
			\node[nom] at (-6,4) {$\pro{48}{\lambda}$};
			\node[align=center] at (0,-5.25) {$\lambda$ atypical of degree $2$, $\lambda_1, \lambda_2 \in 2\ZZ+1$, \\ $\ch{\proj{\lambda}} = \ch{\verma{\lambda+2\alpha_3}} + \ch{\verma{\lambda+2\alpha_3-\alpha_1}} + \ch{\verma{\lambda+2\alpha_3-\alpha_2}} + \ch{\verma{\lambda+\alpha_1}} + \ch{\verma{\lambda+\alpha_2}} + \ch{\verma{\lambda}}$.};
		\end{tikzpicture}
	\caption{Loewy diagrams and characters of the projectives $\pro{d}{\lambda}$ with $\lambda$ atypical.  Here, $d$ is the dimension.  Red arrows are associated with standard filtrations, blue with costandard filtrations and all other arrows are green.  The bottom diagram is not quite complete, because we make no attempt to resolve the four composition factors isomorphic to $\irr{1}{\lambda}$ in the middle row.} \label{fig:LoewyP}
\end{figure}

\begin{proof}
	The fact that $\irred{\lambda}$ is projective, for $\lambda$ typical, follows immediately from \cref{cor:irredext}.  Otherwise, the multiplicities of the Verma factors may be obtained from \cref{bgg,LoewyM}.

	Suppose first that $\dim \irred{\lambda} = 4$ and $i=1,2,3$ is the unique index satisfying $\lambda_i \in 2\ZZ+1$.  Then, $\irred{\lambda}$ can only appear as the head or the socle of another Verma of the same type.  These two possibilities give $(\proj{\lambda} : \verma{\mu}) = [\verma{\mu} : \irred{\lambda}] = \delta^{\mu,\lambda} + \delta^{\mu,\lambda+\alpha_i}$.  In the Grothendieck group of $\qgrpcat$, one therefore has
	\begin{equation}
		[\proj{\lambda}] = [\verma{\lambda}] + [\verma{\lambda+\alpha_i}].
	\end{equation}
	Obviously, only $\verma{\lambda}$ has the same head as $\proj{\lambda}$.  Moreover, \cref{LoewyM} shows that only $\verma{\lambda+\alpha_i}$ has the same socle as $\proj{\lambda}$:
	\begin{equation} \label{eq:irredsoc}
		\soc \proj{\lambda} = \soc_1 \proj{\lambda} \cong \rad_1 \proj{\lambda} = \frac{\proj{\lambda}}{\rad \proj{\lambda}} \cong \irred{\lambda},
	\end{equation}
	$\proj{\lambda}$ thus admits a single standard filtration: $0 \subset \verma{\lambda+\alpha_i} \subset \proj{\lambda}$.

	It follows that $\irred{\lambda-\alpha_i}$ appears in $\rad_2 \proj{\lambda}$ and $\irred{\lambda+\alpha_i}$ appears in $\soc_2 \proj{\lambda}$.  Since projectives are self-dual (\cref{selfdual}), \eqref{eq:socq=radq} leads to $\rad_2 \proj{\lambda} = \soc_2 \proj{\lambda} \cong \irred{\lambda-\alpha_i} \oplus \irred{\lambda+\alpha_i}$.  The Loewy length of $\proj{\lambda}$ is thus $3$ and the Loewy diagram (with some arrows missing) is
	
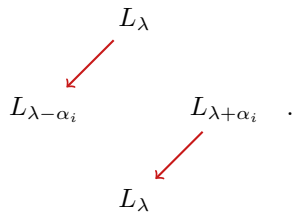
\begin{center}
		\begin{tikzpicture}[->, thick, scale=0.6, baseline=(l.base)]
			\node (t) at (0,2) {$\irred{\lambda}$};
			\node (l) at (-2,0) {$\irred{\lambda-\alpha_i}$};
			\node (r) at (2,0) {$\irred{\lambda+\alpha_i}$};
			\node (b) at (0,-2) {$\irred{\lambda}$};
			\draw[verma] (t) -- (l);
			\draw[verma] (r) -- (b);
		\end{tikzpicture}
		\ .
	\end{center}

	This diagram obviously needs two more arrows.  While it is clear where they go, it is instructive for the other cases to see them arise from a \emph{costandard filtration} of $\proj{\lambda}$, that is one whose successive quotients are \emph{dual} Verma modules.  This series is easily deduced by applying the duality functor to the short exact sequence
	\begin{equation}
		0 \longrightarrow \verma{\lambda+\alpha_i} \longrightarrow \proj{\lambda} \longrightarrow \verma{\lambda} \longrightarrow 0,
	\end{equation}
	the result being $0 \subset \dverma{\lambda} \subset \proj{\lambda}$ with $\proj{\lambda} \big/ \dverma{\lambda} \cong \dverma{\lambda+\alpha_i}$.  Since Verma modules have unique Loewy diagrams, the same is true for their duals and the latter's diagrams are obtained by turning those of \cref{fig:LoewyM} upside-down.  Drawing in blue the arrows corresponding to these dual Verma modules on the partial Loewy diagram above (to distinguish them from the Verma arrows drawn in red), we arrive at the Loewy diagram of \cref{fig:LoewyP}.  As the socle and head are irreducible, it follows that this is the unique Loewy series of $\proj{\lambda}$.

	Suppose next that $\dim \irred{\lambda} = 3$, so that $\lambda_i, \lambda_3 \in 2\ZZ+1$ for some $i=1,2$.  Consulting \cref{fig:LoewyM}, BGG reciprocity implies this time that
	\begin{equation}
		[\proj{\lambda}] = [\verma{\lambda}] + [\verma{\lambda+\alpha_i}] + [\verma{\lambda+\alpha_3}].
	\end{equation}
	Again, the only Verma factor whose head matches that of $\proj{\lambda}$ is $\verma{\lambda}$ and the only Verma whose socle matches that of $\proj{\lambda}$ is $\verma{\lambda+\alpha_3}$.  We conclude that in this case, there is also a unique standard series for $\proj{\lambda}$ and it has the form
	\begin{equation} \label{eq:standardseriesdim=3}
		0 \subset \verma{\lambda+\alpha_3} \subset N \subset \proj{\lambda}, \qquad \text{with}\ \frac{N}{\verma{\lambda+\alpha_3}} \cong \verma{\lambda+\alpha_i}\ \text{and}\ \frac{\proj{\lambda}}{N} \cong \verma{\lambda}.
	\end{equation}

	Combining this with \eqref{eq:socq=radq} and \cref{LoewyM} identifies some of the composition factors in the second Loewy layers of the radical and socle series:
	\begin{equation}
		\irred{\lambda-\alpha_i} \oplus \irred{\lambda+\alpha_i-2\alpha_3} \subseteq \rad_2 \proj{\lambda} \cong \soc_2 \proj{\lambda} \supseteq \irred{\lambda+\alpha_i} \oplus \irred{\lambda-\alpha_i}.
	\end{equation}
	It follows that $\irred{\lambda+\alpha_i}$ appears in $\rad_2 \proj{\lambda}$ and $\irred{\lambda+\alpha_i-2\alpha_3}$ appears in $\soc_2 \proj{\lambda}$.  These extra composition factors are the head and socle of $\verma{\lambda+\alpha_i}$, hence the length of the radical and socle series of $\proj{\lambda}$ is $5$.  In fact, these series coincide and the Loewy diagram (with some arrows missing) is as follows:
	
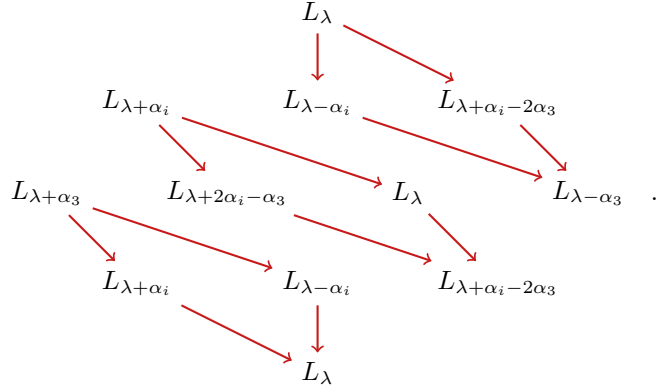
\begin{center}
		\begin{tikzpicture}[->, thick, scale=0.6, baseline=(33.base)]
			\node (1) at (0,4) [] {$\irred{\lambda}$};
			\node (21) at (-4,2) [] {$\irred{\lambda+\alpha_i}$};
			\node (22) at (0,2) [] {$\irred{\lambda-\alpha_i}$};
			\node (23) at (4,2) [] {$\irred{\lambda+\alpha_i-2\alpha_3}$};
			\node (31) at (-6,0) [] {$\irred{\lambda+\alpha_3}$};
			\node (32) at (-2,0) [] {$\irred{\lambda+2\alpha_i-\alpha_3}$};
			\node (33) at (2,0) [] {$\irred{\lambda}$};
			\node (34) at (6,0) [] {$\irred{\lambda-\alpha_3}$};
			\node (41) at (-4,-2) [] {$\irred{\lambda+\alpha_i}$};
			\node (42) at (0,-2) [] {$\irred{\lambda-\alpha_i}$};
			\node (43) at (4,-2) [] {$\irred{\lambda+\alpha_i-2\alpha_3}$};
			\node (5) at (0,-4) [] {$\irred{\lambda}$};
			\draw[verma] (1) -- (22);
			\draw[verma] (1) -- (23);
			\draw[verma] (22) -- (34);
			\draw[verma] (23) -- (34);
			\draw[verma] (21) -- (32);
			\draw[verma] (21) -- (33);
			\draw[verma] (32) -- (43);
			\draw[verma] (33) -- (43);
			\draw[verma] (31) -- (41);
			\draw[verma] (31) -- (42);
			\draw[verma] (41) -- (5);
			\draw[verma] (42) -- (5);
		\end{tikzpicture}
		\ .
	\end{center}

	To complement these red Verma arrows, one next computes the corresponding costandard filtration of $\proj{\lambda}$ by dualising \eqref{eq:standardseriesdim=3}:
	\begin{equation}
		\dverma{\lambda} \into \proj{\lambda} \overset{\pi}{\onto} \dual{N} \qquad \text{and} \qquad
		\dverma{\lambda+\alpha_i} \into \dual{N} \overset{\pi'}{\onto} \dverma{\lambda+\alpha_3}.
	\end{equation}
	The first says that $\dverma{\lambda} = \ker \pi$ is a submodule of $\proj{\lambda}$ and the second says that $\pi^{-1}(\dverma{\lambda+\alpha_i})$ is too.  Of course, $\pi^{-1}(\dverma{\lambda+\alpha_i}) \big/ \dverma{\lambda} = \pi^{-1}(\dverma{\lambda+\alpha_i}) \big/ \ker \pi \cong \dverma{\lambda+\alpha_i}$.  Similarly, $\pi^{-1}(\dverma{\lambda+\alpha_i}) = \ker(\pi' \circ \pi)$ and so $\proj{\lambda} \big/ \pi^{-1}(\dverma{\lambda+\alpha_i}) \cong \dverma{\lambda+\alpha_3}$.  The required costandard series is thus
	\begin{equation}
		0 \subset \dverma{\lambda} \subset \pi^{-1}(\dverma{\lambda+\alpha_i}) \subset \proj{\lambda}
	\end{equation}
	and the dual Verma quotients are $\dverma{\lambda}$, $\dverma{\lambda+\alpha_i}$ and $\dverma{\lambda+\alpha_3}$.  Drawing these dual Verma arrows on the Loewy diagram of $\proj{\lambda}$ in blue (and not repeating if the arrow already exists) gives
	
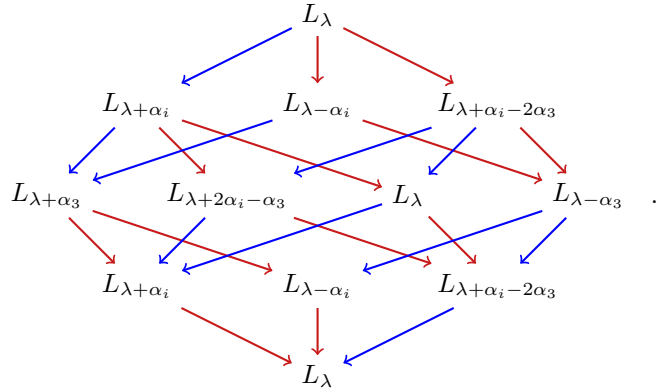
\begin{center}
		\begin{tikzpicture}[->, thick, scale=0.6, baseline=(33.base)]
			\node (1) at (0,4) [] {$\irred{\lambda}$};
			\node (21) at (-4,2) [] {$\irred{\lambda+\alpha_i}$};
			\node (22) at (0,2) [] {$\irred{\lambda-\alpha_i}$};
			\node (23) at (4,2) [] {$\irred{\lambda+\alpha_i-2\alpha_3}$};
			\node (31) at (-6,0) [] {$\irred{\lambda+\alpha_3}$};
			\node (32) at (-2,0) [] {$\irred{\lambda+2\alpha_i-\alpha_3}$};
			\node (33) at (2,0) [] {$\irred{\lambda}$};
			\node (34) at (6,0) [] {$\irred{\lambda-\alpha_3}$};
			\node (41) at (-4,-2) [] {$\irred{\lambda+\alpha_i}$};
			\node (42) at (0,-2) [] {$\irred{\lambda-\alpha_i}$};
			\node (43) at (4,-2) [] {$\irred{\lambda+\alpha_i-2\alpha_3}$};
			\node (5) at (0,-4) [] {$\irred{\lambda}$};
			\draw[verma] (1) -- (22);
			\draw[verma] (1) -- (23);
			\draw[verma] (22) -- (34);
			\draw[verma] (23) -- (34);
			\draw[verma] (21) -- (32);
			\draw[verma] (21) -- (33);
			\draw[verma] (32) -- (43);
			\draw[verma] (33) -- (43);
			\draw[verma] (31) -- (41);
			\draw[verma] (31) -- (42);
			\draw[verma] (41) -- (5);
			\draw[verma] (42) -- (5);
			\draw[dverma] (1) -- (21);
			\draw[dverma] (21) -- (31);
			\draw[dverma] (22) -- (31);
			\draw[dverma] (23) -- (32);
			\draw[dverma] (23) -- (33);
			\draw[dverma] (32) -- (41);
			\draw[dverma] (33) -- (41);
			\draw[dverma] (34) -- (42);
			\draw[dverma] (34) -- (43);
			\draw[dverma] (43) -- (5);
		\end{tikzpicture}
		\ .
	\end{center}

	It remains to check for any additional arrows beyond those already established.  There are only a few possibilities and most are ruled out by \cref{cor:irredext}.  The only possible additions are an arrow from the composition factor $\irred{\lambda-\alpha_i}$ in the second-top layer to $\irred{\lambda}$ in the middle layer and from the latter factor to $\irred{\lambda-\alpha_i}$ in the second-bottom layer.  In both cases, the absence of the arrow would imply that $\proj{\lambda}$ has a quotient or submodule whose Loewy diagram is
	\begin{center}
		\begin{tikzpicture}[->, thick, scale=0.6, baseline=(23.base)]
			\node (1) at (0,4) [] {$\irred{\lambda}$};
			\node (21) at (-2,2) [] {$\irred{\lambda+\alpha_i}$};
			\node (22) at (2,2) [] {$\irred{\lambda+\alpha_i-2\alpha_3}$};
			\node (3) at (0,0) [] {$\irred{\lambda}$};
			\draw[->,verma] (1) -- (22);
			\draw[->,verma] (21) -- (3);
			\draw[->,dverma] (1) -- (21);
			\draw[->,dverma] (22) -- (3);
		\end{tikzpicture}
		\ .
	\end{center}
	But, such a $\qgrp$-module does not exist, by \cref{Loewylemma}.  Both arrows must therefore be present and the Loewy diagram is thus as in \cref{fig:LoewyP}.

	Finally, we consider the case $\dim \irred{\lambda} = 1$, hence $\lambda_1, \lambda_2 \in 2\ZZ+1$.  Now, BGG reciprocity gives six Verma modules for the projective:
	\begin{equation}
		[\proj{\lambda}] = [\verma{\lambda}] + [\verma{\lambda+\alpha_1}] + [\verma{\lambda+\alpha_2}] + [\verma{\lambda+2\alpha_3-\alpha_1}] + [\verma{\lambda+2\alpha_3-\alpha_2}] + [\verma{\lambda+2\alpha_3}].
	\end{equation}
	Following the same procedure as in the previous case, we find that the radical and socle series again coincide and have length $5$.  The Loewy diagram (with arrows corresponding to any of the standard filtrations) is
	\begin{center}
		\begin{tikzpicture}[->, thick, scale=0.75, baseline=(34.base)]
			\node (1) at (0,4) [] {$\irred{\lambda}$};
			\node (21) at (-6.25,2) [] {$\irred{\lambda-\alpha_1+2\alpha_3}$};
			\node (22) at (-3.75,2) [] {$\irred{\lambda-\alpha_2+2\alpha_3}$};
			\node (23) at (-1.25,2) [] {$\irred{\lambda+\alpha_2}$};
			\node (24) at (1.25,2) [] {$\irred{\lambda+\alpha_1}$};
			\node (25) at (3.75,2) [] {$\irred{\lambda-\alpha_1}$};
			\node (26) at (6.25,2) [] {$\irred{\lambda-\alpha_2}$};
			\node (31) at (-9,0) [] {$\irred{\lambda+2\alpha_3}$};
			\node (32) at (-7,0) [] {$\irred{\lambda+2\alpha_2}$};
			\node (33) at (-5,0) [] {$\irred{\lambda}$};
			\node (34) at (-3,0) [] {$\irred{\lambda+2\alpha_1}$};
			\node (35) at (-1,0) [] {$\irred{\lambda}$};
			\node (36) at (1,0) [] {$\irred{\lambda}$};
			\node (37) at (3,0) [] {$\irred{\lambda-2\alpha_1}$};
			\node (38) at (5,0) [] {$\irred{\lambda}$};
			\node (39) at (7,0) [] {$\irred{\lambda-2\alpha_2}$};
			\node (30) at (9,0) [] {$\irred{\lambda-2\alpha_3}$};
			\node (41) at (-6.25,-2) [] {$\irred{\lambda-\alpha_1+2\alpha_3}$};
			\node (42) at (-3.75,-2) [] {$\irred{\lambda-\alpha_2+2\alpha_3}$};
			\node (43) at (-1.25,-2) [] {$\irred{\lambda+\alpha_2}$};
			\node (44) at (1.25,-2) [] {$\irred{\lambda+\alpha_1}$};
			\node (45) at (3.75,-2) [] {$\irred{\lambda-\alpha_1}$};
			\node (46) at (6.25,-2) [] {$\irred{\lambda-\alpha_2}$};
			\node (5) at (0,-4) [] {$\irred{\lambda}$};
			\draw[verma] (1) -- (26);
			\draw[verma] (1) -- (25);
			\draw[verma] (25) -- (30);
			\draw[verma] (26) -- (30);
			\draw[verma] (24) -- (38);
			\draw[verma] (24) -- (39);
			\draw[verma] (38) -- (46);
			\draw[verma] (39) -- (46);
			\draw[verma] (23) -- (36);
			\draw[verma] (23) -- (37);
			\draw[verma] (36) -- (45);
			\draw[verma] (37) -- (45);
			\draw[verma] (22) -- (34);
			\draw[verma] (22) -- (35);
			\draw[verma] (34) -- (44);
			\draw[verma] (35) -- (44);
			\draw[verma] (21) -- (32);
			\draw[verma] (21) -- (33);
			\draw[verma] (32) -- (43);
			\draw[verma] (33) -- (43);
			\draw[verma] (31) -- (41);
			\draw[verma] (31) -- (42);
			\draw[verma] (41) -- (5);
			\draw[verma] (42) -- (5);
		\end{tikzpicture}
		\ .
	\end{center}
	Unfortunately, it is not straightforward to add the arrows corresponding to (any of) the costandard filtrations because $\irred{\lambda}$ occurs in the middle row with multiplicity $4$ --- roughly speaking, it is difficult to identify which copy (more precisely, which linear combination of copies) of $\irred{\lambda}$ corresponds to the head or tail of the arrows in a costandard filtration.  Rather than deal with this, we shall be content with a partial Loewy diagram in which these $4$ copies of $\irred{\lambda}$ are collapsed to $\irred{\lambda}^{\oplus4}$.  With this, adding the costandard arrows is easy.  Finally, there are obvious arrows to add between the first and second, and fourth and fifth, rows.  \cref{cor:irredext} then rules out any more.
\end{proof}

\subsection{Tensor product decompositions} \label{tensorsec}

Recall that there are four types of irreducible $\qgrp$-modules in $\qgrpcat$, distinguished by their dimensions.  Including this dimension as a superscript (as in \cref{fig:LoewyM,fig:LoewyP}), these types are characterised by the following conditions on their highest weights:
\begin{equation} \label{eq:thefourtypes}
	\begin{matrix}
		\irr{1}{\lambda}: && \lambda_1, \lambda_2 \in 2\ZZ+1,\ \lambda_3 \in 2\ZZ; \\
		\irr{3}{\lambda}: && \lambda_i, \lambda_3 \in 2\ZZ+1,\ \lambda_j \in 2\ZZ && \text{($\{i,j\}=\{1,2\}$)}; \\
		\irr{4}{\lambda}: && \lambda_i \in 2\ZZ+1,\ \lambda_j, \lambda_k \notin \ZZ && \text{($\{i,j,k\}=\{1,2,3\}$)}; \\
		\irr{8}{\lambda}: && \lambda_1, \lambda_2, \lambda_3 \notin 2\ZZ+1.
	\end{matrix}
\end{equation}
Recall also that $\lambda_i$ is said to be atypical if $\lambda_i \in 2\ZZ+1$ and typical otherwise.  In this \lcnamecref{tensorsec}, the tensor product of (almost) every pair of irreducible $\qgrp$-modules in $\qgrpcat$ is explicitly decomposed (up to isomorphism) as a direct sum of indecomposables.

It is worth mentioning that these computations are simplified by the fact that $\qgrpcat$ is braided.  This follows immediately from the fact that $\qgrp$ admits a quasitriangular structure \cite[Theorem~41]{GP}, see also \cite[Equations~(4.8--9)]{R}.

The easiest tensor products to identify are those of the form $\irr{1}{\lambda} \otimes \irred{\mu}$.  These have dimension $\dim \irred{\mu}$ and possess a maximal vector of weight $\lambda+\mu$.  Since
\begin{equation} \label{eq:indexofsum}
	(\lambda+\mu)_i = \lambda_i + \mu_i - \killing{\rho}{\alpha_i}, \quad i=1,2,3,
\end{equation}
and $\lambda_i - \killing{\rho}{\alpha_i} \in 2\ZZ$ for each $i$, it follows that $\lambda+\mu$ satisfies the same set of conditions in \eqref{eq:thefourtypes} as $\mu$.  This leads to the following conclusion.

\begin{proposition} \label{prop:TP1xL}
	Given $\lambda, \mu \in \csub^*$ satisfying $\lambda_1, \lambda_2 \in 2\ZZ+1$, one has
	\begin{equation} \label{eq:TP1xL}
		\irr{1}{\lambda} \otimes \irred{\mu} \cong \irred{\lambda+\mu}.
	\end{equation}
\end{proposition}

Tensor products involving a typical irreducible $\irr{8}{\mu}$ are likewise relatively easy to identify because of the following facts:
\begin{itemize}
	\item $\qgrpcat$ is pivotal \cite[Lemma~17]{GPV}, so the projectives form a tensor ideal.
	\item Every projective is self-dual (\cref{selfdual}), thus injective, hence their extensions split.
	\item Projectives are uniquely determined (up to isomorphism) by their characters (\cref{projcharacter}).
\end{itemize}
The tensor product of any module of $\qgrpcat$ with an $\irr{8}{\mu}$ thus decomposes as a direct sum of projectives which may be identified from its character.

To illustrate how one identifies such a tensor product, consider $\irr{3}{\lambda} \otimes \irr{8}{\mu}$, so that the $\mu_i$, $i=1,2,3$, are typical (note that $\irr{1}{\lambda} \otimes \irr{8}{\mu}$ was already identified in \cref{prop:TP1xL}).  Without loss of generality, assume that $\lambda_1 \in 2\ZZ+1$ and $\lambda_2 \in 2\ZZ$.  Then, the character of the tensor product satisfies
\begin{align}
	\ch{\irr{3}{\lambda} \otimes \irr{8}{\mu}} &= \ch{\irr{3}{\lambda}} \ch{\irr{8}{\mu}} = (z^{\lambda} + z^{\lambda-\alpha_2} + z^{\lambda-\alpha_3}) \ch{\verma{\mu}} \notag \\
	&= \ch{\verma{\lambda+\mu}} + \ch{\verma{\lambda+\mu-\alpha_2}} + \ch{\verma{\lambda+\mu-\alpha_3}}.
\end{align}
It therefore remains to identify which projective has this character.  This splits into several cases.

When $\lambda+\mu$, $\lambda+\mu-\alpha_2$ and $\lambda+\mu-\alpha_3$ are all typical, the projective is $\pro{8}{\lambda+\mu} \oplus \pro{8}{\lambda+\mu-\alpha_2} \oplus \pro{8}{\lambda+\mu-\alpha_3}$, where the dimension of each indecomposable summand is likewise indicated in the superscript.  By \eqref{eq:indexofsum}, the typicality of all three weights is equivalent to $\mu_i \notin \ZZ$, $i=1,2,3$.

Suppose now that $\mu_1 \in 2\ZZ$, but $\mu_2, \mu_3 \notin \ZZ$.  Then, $\lambda+\mu$ is typical while $\lambda+\mu-\alpha_2$ and $\lambda+\mu-\alpha_3$ are atypical of degree $1$: $(\lambda+\mu-\alpha_2)_1 = \lambda_1 + \mu_1 \in 2\ZZ+1$ and $(\lambda+\mu-\alpha_3)_1 = \lambda_1 + \mu_1 - 2 \in 2\ZZ+1$.  It follows that $\pro{8}{\lambda+\mu}$ splits off as a direct summand of the tensor product and the remaining two Verma modules combine to form a $16$-dimensional indecomposable projective.  Comparing characters with those in \cref{fig:LoewyP}, the latter projective is identified as $\pro{16}{\lambda+\mu-\alpha_3}$.

The story is similar when $\mu_2$ or $\mu_3$ is the only integral label (which then lies in $2\ZZ$).  The last remaining possibility is $\mu_1, \mu_2, \mu_3 \in 2\ZZ$.  Then, $\lambda+\mu$, $\lambda+\mu-\alpha_2$ and $\lambda+\mu-\alpha_3$ are atypical of degree $2$, corresponding to irreducibles of dimension $3$, $1$ and $3$, respectively.  The only projective consistent with this is $\pro{24}{\lambda+\mu-\alpha_3}$, again by \cref{fig:LoewyP}.

\begin{proposition} \label{prop:TP3xM}
	Given $\lambda \in \csub^*$ satisfying $\lambda_i, \lambda_3 \in 2\ZZ+1$ and $\lambda_j \in 2\ZZ$, where $\{i,j\} = \{1,2\}$, and $\mu \in \csub^*$ typical, $\irr{3}{\lambda} \otimes \irr{8}{\mu}$ is isomorphic to
	\begin{align} \label{eq:TP3xM}
		\mu_1, \mu_2, \mu_3 \in 2\ZZ &: &  &\pro{24}{\lambda+\mu-\alpha_3}, \notag \\
		\mu_i \in 2\ZZ,\ \mu_j, \mu_3 \notin \ZZ &: & &\pro{8}{\lambda+\mu} \oplus \pro{16}{\lambda+\mu-\alpha_3}, \notag \\
		\mu_j \in 2\ZZ,\ \mu_i, \mu_3 \notin \ZZ &: & &\pro{16}{\lambda+\mu-\alpha_j} \oplus \pro{8}{\lambda+\mu-\alpha_3}, \\
		\mu_3 \in 2\ZZ,\ \mu_1, \mu_2 \notin \ZZ &: & &\pro{8}{\lambda+\mu-\alpha_j} \oplus \pro{16}{\lambda+\mu-\alpha_3}, \notag \\
		\mu_1, \mu_2, \mu_3 \notin \ZZ &: & &\pro{8}{\lambda+\mu} \oplus \pro{8}{\lambda+\mu-\alpha_j} \oplus \pro{8}{\lambda+\mu-\alpha_3}. \notag
	\end{align}
\end{proposition}

Explicit decompositions for $\irr{4}{\lambda} \otimes \irr{8}{\mu}$ and $\irr{8}{\lambda} \otimes \irr{8}{\mu}$ may be readily found using the same methods.  We omit the details.

\begin{proposition} \label{prop:TP4xM}
	Given $\lambda \in \csub^*$ satisfying $\lambda_i \in 2\ZZ+1$ and $\lambda_j, \lambda_3 \notin \ZZ$, where $\{i,j\} = \{1,2\}$, and $\mu \in \csub^*$ typical, $\irr{4}{\lambda} \otimes \irr{8}{\mu}$ is isomorphic to
	\begin{align} \label{eq:TP4xM}
		\mu_i, \lambda_j + \mu_j \in 2\ZZ,\ \lambda_3 + \mu_3 \in 2\ZZ+1 &: & &\pro{24}{\lambda+\mu-\alpha_3} \oplus \pro{8}{\lambda+\mu-\alpha_j-\alpha_3}, \notag \\
		\mu_i, \lambda_3 + \mu_3 \in 2\ZZ,\ \lambda_j + \mu_j \in 2\ZZ+1 &: & &\pro{8}{\lambda+\mu} \oplus \pro{24}{\lambda+\mu-\alpha_j-\alpha_3}, \notag \\
		\mu_i \in 2\ZZ,\ \lambda_j + \mu_j, \lambda_3 + \mu_3 \notin \ZZ &: & &\pro{8}{\lambda+\mu} \oplus \pro{16}{\lambda+\mu-\alpha_3} \oplus \pro{8}{\lambda+\mu-\alpha_j-\alpha_3}, \notag \\
		\lambda_j + \mu_j \in 2\ZZ,\ \mu_i, \lambda_3 + \mu_3 \notin \ZZ &: & &\pro{16}{\lambda+\mu-\alpha_j} \oplus \pro{8}{\lambda+\mu-\alpha_3} \oplus \pro{8}{\lambda+\mu-\alpha_j-\alpha_3}, \notag \\
		\lambda_j + \mu_j \in 2\ZZ+1,\ \mu_i, \lambda_3 + \mu_3 \notin \ZZ &: & &\pro{8}{\lambda+\mu} \oplus \pro{8}{\lambda+\mu-\alpha_3} \oplus \pro{16}{\lambda+\mu-\alpha_j-\alpha_3}, \\
		\lambda_3 + \mu_3 \in 2\ZZ,\ \mu_i, \lambda_j + \mu_j \notin \ZZ &: & &\pro{8}{\lambda+\mu} \oplus \pro{8}{\lambda+\mu-\alpha_j} \oplus \pro{16}{\lambda+\mu-\alpha_j-\alpha_3}, \notag \\
		\lambda_3 + \mu_3 \in 2\ZZ+1,\ \mu_i, \lambda_j + \mu_j \notin \ZZ &: & &\pro{8}{\lambda+\mu-\alpha_j} \oplus \pro{16}{\lambda+\mu-\alpha_3} \oplus \pro{8}{\lambda+\mu-\alpha_j-\alpha_3}, \notag \\
		\mu_1, \lambda_2 + \mu_2, \lambda_3 + \mu_3 \notin \ZZ &; & &\pro{8}{\lambda+\mu} \oplus \pro{8}{\lambda+\mu-\alpha_j} \oplus \pro{8}{\lambda+\mu-\alpha_3} \oplus \pro{8}{\lambda+\mu-\alpha_j-\alpha_3}. \notag
	\end{align}
\end{proposition}

\begin{proposition} \label{prop:TP4XM'}
	Given $\lambda \in \csub^*$ satisfying $\lambda_1, \lambda_2 \notin \ZZ$ and $\lambda_3 \in 2\ZZ+1$, and $\mu \in \csub^*$ typical, $\irr{4}{\lambda} \otimes \irr{8}{\mu}$ is isomorphic to
	\begin{align} \label{eq:TP4xM'}
		\lambda_i + \mu_i, \mu_3 \in 2\ZZ,\ \lambda_j + \mu_j \in 2\ZZ+1 &: & & \pro{8}{\lambda+\mu-\alpha_j} \oplus \pro{24}{\lambda+\mu-\alpha_3}, \notag \\
		\mu_3 \in 2\ZZ,\ , \lambda_1 + \mu_1, \lambda_2 + \mu_2 \notin \ZZ &: & & \pro{8}{\lambda+\mu-\alpha_1} \oplus \pro{8}{\lambda+\mu-\alpha_2} \oplus \pro{16}{\lambda+\mu-\alpha_3}, \notag \\
		\lambda_i + \mu_i \in 2\ZZ,\ \lambda_j + \mu_j, \mu_3 \notin \ZZ &: & & \pro{16}{\lambda+\mu-\alpha_i} \oplus \pro{8}{\lambda+\mu-\alpha_j} \oplus \pro{8}{\lambda+\mu-\alpha_3}, \\
		\lambda_i + \mu_i \in 2\ZZ+1,\ \lambda_j + \mu_j, \mu_3 \notin \ZZ &: & & \pro{8}{\lambda+\mu} \oplus \pro{8}{\lambda+\mu-\alpha_i} \oplus \pro{16}{\lambda+\mu-\alpha_3}, \notag \\
		\lambda_1 + \mu_1, \lambda_2 + \mu_2, \mu_3 \notin \ZZ &; & & \pro{8}{\lambda+\mu} \oplus \pro{8}{\lambda+\mu-\alpha_1} \oplus \pro{8}{\lambda+\mu-\alpha_2} \oplus \pro{8}{\lambda+\mu-\alpha_3}, \notag
	\end{align}
	where $\{i,j\}=\{1,2\}$.
\end{proposition}

\begin{proposition} \label{prop:TP8xM}
	Given typical $\lambda, \mu \in \csub^*$, $\irr{8}{\lambda} \otimes \irr{8}{\mu}$ is isomorphic to
	\begin{align} \label{eq:TP8xM}
		\lambda_1 + \mu_1, \lambda_2 + \mu_2, \lambda_3 + \mu_3 \in 2\ZZ &: & &2 \pro{8}{\lambda+\mu-\alpha_3} \oplus \pro{48}{\lambda+\mu-2\alpha_3}, \notag \\
		\lambda_i + \mu_i \in 2\ZZ,\ \lambda_j + \mu_j, \lambda_3 + \mu_3 \in 2\ZZ+1 &: & &\pro{8}{\lambda+\mu-\alpha_j} \oplus \pro{8}{\lambda+\mu-\alpha_i-\alpha_3} \oplus \pro{24}{\lambda+\mu-\alpha_3} \oplus \pro{24}{\lambda+\mu-2\alpha_3}, \notag \\
		\lambda_1 + \mu_1, \lambda_2 + \mu_2 \in 2\ZZ+1,\ \lambda_3 + \mu_3 \in 2\ZZ &: & &\pro{8}{\lambda+\mu} \oplus \pro{24}{\lambda+\mu-\alpha_1-\alpha_3} \oplus \pro{24}{\lambda+\mu-\alpha_2-\alpha_3} \oplus \pro{8}{\lambda+\mu-2\alpha_3}, \notag \\
		\lambda_i + \mu_i \in 2\ZZ,\ \lambda_j + \mu_j, \lambda_3 + \mu_3 \notin \ZZ &: & &\pro{16}{\lambda+\mu-\alpha_i} \oplus \pro{8}{\lambda+\mu-\alpha_j} \oplus 2 \pro{8}{\lambda+\mu-\alpha_3} \notag \\
		&&&\phantom{\pro{16}{\lambda+\mu-\alpha_i}} \oplus \pro{8}{\lambda+\mu-\alpha_i-\alpha_3} \oplus \pro{16}{\lambda+\mu-2\alpha_3}, \notag \\
		\lambda_i + \mu_i \in 2\ZZ+1,\ \lambda_j + \mu_j, \lambda_3 + \mu_3 \notin \ZZ &: & &\pro{8}{\lambda+\mu} \oplus \pro{8}{\lambda+\mu-\alpha_1} \oplus \pro{16}{\lambda+\mu-\alpha_3} \notag \\
		&&&\phantom{\pro{8}{\lambda+\mu}} \oplus \pro{16}{\lambda+\mu-\alpha_i-\alpha_3} \oplus \pro{8}{\lambda+\mu-\alpha_j-\alpha_3} \oplus \pro{8}{\lambda+\mu-2\alpha_3}, \\
		\lambda_3 + \mu_3 \in 2\ZZ,\ \lambda_1 + \mu_1, \lambda_2 + \mu_2 \notin \ZZ &: & &\pro{8}{\lambda+\mu} \oplus 2 \pro{8}{\lambda+\mu-\alpha_3} \oplus \pro{16}{\lambda+\mu-\alpha_1-\alpha_3} \notag \\
		&&&\phantom{\pro{8}{\lambda+\mu}} \oplus \pro{16}{\lambda+\mu-\alpha_2-\alpha_3} \oplus \pro{8}{\lambda+\mu-2\alpha_3}, \notag \\
		\lambda_3 + \mu_3 \in 2\ZZ+1,\ \lambda_1 + \mu_1, \lambda_2 + \mu_2 \notin \ZZ &: & &\pro{8}{\lambda+\mu-\alpha_1} \oplus \pro{8}{\lambda+\mu-\alpha_2} \oplus \pro{16}{\lambda+\mu-\alpha_3} \notag \\
		&&&\phantom{\pro{8}{\lambda+\mu-\alpha_1}} \oplus \pro{8}{\lambda+\mu-\alpha_1-\alpha_3} \oplus \pro{8}{\lambda+\mu-\alpha_2-\alpha_3} \oplus \pro{16}{\lambda+\mu-2\alpha_3}, \notag \\
		\lambda_1 + \mu_1, \lambda_2 + \mu_2, \lambda_3 + \mu_3 \notin \ZZ &: & &\pro{8}{\lambda+\mu} \oplus \pro{8}{\lambda+\mu-\alpha_1} \oplus \pro{8}{\lambda+\mu-\alpha_2} \oplus 2 \pro{8}{\lambda+\mu-\alpha_3} \notag \\
		&&&\phantom{\pro{8}{\lambda+\mu}} \oplus \pro{8}{\lambda+\mu-\alpha_1-\alpha_3} \oplus \pro{8}{\lambda+\mu-\alpha_2-\alpha_3} \oplus \pro{8}{\lambda+\mu-2\alpha_3}, \notag
	\end{align}
	where $\{i,j\}=\{1,2\}$ and $2V$ means $V^{\oplus 2}$.
\end{proposition}

It only remains to decompose the tensor products of the $3$- and $4$-dimensional irreducibles.  Here, identifying projective summands is not always sufficient and one must sometimes resort to brute force computation.  However, there are still many cases for which projectivity is very helpful.

As an illustration, consider the decomposition of $\irr{4}{\lambda} \otimes \irr{4}{\mu}$, where $\lambda_1, \mu_1 \in 2\ZZ+1$ and $\lambda_2, \mu_2 \notin \ZZ$.  This tensor product obviously has a maximal vector $v$ of weight $\lambda+\mu$ that satisfies $(\lambda+\mu)_1 \in 2\ZZ+1$.  Assume that $\lambda+\mu$ is atypical of degree $1$, so $(\lambda+\mu)_2, (\lambda+\mu)_3 \notin \ZZ$.  Then, $v$ generates a submodule of dimension at least $4$ in which $X_{-2} v$ is not maximal.

Now, the weight $\lambda+\mu-\alpha_2$ has multiplicity $2$ in the tensor product while $\lambda+\mu$ has multiplicity $1$.  $X_2$ must therefore annihilate some vector $w$ of weight $\lambda+\mu-\alpha_2$.  As $\lambda+\mu+\alpha_1-\alpha_2$ is not a weight of $\irr{4}{\lambda} \otimes \irr{4}{\mu}$, $w$ is maximal.  But, $\lambda+\mu-\alpha_2$ is typical, hence $w$ generates a copy of the irreducible projective $\pro{8}{\lambda+\mu-\alpha_2}$ which is therefore a direct summand of the tensor product.

Of the $8$ weights (counted with multiplicity) not accounted for by the copy of $\pro{8}{\lambda+\mu-\alpha_2}$, at least $4$ belong to the submodule, isomorphic to $\irr{4}{\lambda+\mu}$, generated by $v$.  That leaves $4$, the highest of which is easily checked to be $\lambda+\mu-\alpha_2-\alpha_3$.  However, this is also atypical of degree $1$, hence the corresponding composition factor is $\irr{4}{\lambda+\mu-\alpha_2-\alpha_3}$ which accounts for all the remaining weights.

The tensor product therefore has three composition factors: $\irr{4}{\lambda+\mu}$, $\pro{8}{\lambda+\mu-\alpha_2}$ and $\irr{4}{\lambda+\mu-\alpha_2-\alpha_3}$.  The only unknown left is whether the $4$-dimensional irreducibles are direct summands or parts of an $8$-dimensional indecomposable.  However, the latter is impossible, by \cref{cor:irredext}.

Similar arguments result in the following decompositions.

\begin{proposition} \label{prop:TP34x34}
	\leavevmode
	\begin{itemize}
		\item Given $\lambda, \mu, \lambda+\mu \in \csub^*$ atypical of degree $1$, $\irr{4}{\lambda} \otimes \irr{4}{\mu}$ is isomorphic to
	\begin{align}
		\lambda_i, \mu_j \in 2\ZZ+1 &: & &\pro{8}{\lambda+\mu} \oplus \pro{8}{\lambda+\mu-\alpha_k} & \textup{(} \{i,j,k\} &= \{1,2,3\} \textup{)}, \notag \\
		\lambda_i, \mu_i \in 2\ZZ+1 &: & &\irr{4}{\lambda+\mu} \oplus \pro{8}{\lambda+\mu-\alpha_j} \oplus \irr{4}{\lambda+\mu-\alpha_j-\alpha_3} & \textup{(} \{i,j\} &= \{1,2\} \textup{)}, \label{eq:TP4x4} \\
		\lambda_3, \mu_3 \in 2\ZZ+1 &: & &\pro{8}{\lambda+\mu} \oplus \irr{4}{\lambda+\mu-\alpha_1} \oplus \irr{4}{\lambda+\mu-\alpha_2}. \notag
	\end{align}
	\item Given $\lambda \in \csub^*$ atypical of degree $2$ with $\lambda_3 \in 2\ZZ+1$ and $\mu$ atypical of degree $1$, $\irr{3}{\lambda} \otimes \irr{4}{\mu}$ is isomorphic to
	\begin{equation}
		\begin{aligned} \label{eq:TP3x4}
			\lambda_i, \mu_j \in 2\ZZ+1 &: &&& &\pro{8}{\lambda+\mu} \oplus \irr{4}{\lambda+\mu-\alpha_k} &&& \textup{(} i \in \{1,2\},\ \{i,j,k\} &= \{1,2,3\} \textup{)}, \\
			\lambda_i, \mu_i \in 2\ZZ+1 &: &&& &\irr{4}{\lambda+\mu} \oplus \pro{8}{\lambda+\mu-\alpha_j} &&& \textup{(} \{i,j\} &= \{1,2\} \textup{)}.
		\end{aligned}
	\end{equation}
	\item Given $\lambda, \mu \in \csub^*$ atypical of degree $2$ with $\lambda_i, \lambda_3, \mu_j, \mu_3 \in 2\ZZ+1$, where $\{i,j\} = \{1,2\}$, one has
	\begin{equation} \label{eq:TP3x3}
		\irr{3}{\lambda} \otimes \irr{3}{\mu} \cong \pro{8}{\lambda+\mu} \oplus \irr{1}{\lambda+\mu-\alpha_3}.
	\end{equation}
	\end{itemize}
\end{proposition}

The remaining tensor products are $\irr{4}{\lambda} \otimes \irr{4}{\mu}$, with $\lambda+\mu$ atypical of degree $2$, and $\irr{3}{\lambda} \otimes \irr{3}{\mu}$, with $\lambda_i, \lambda_3, \mu_i, \mu_3 \in 2\ZZ+1$.  It seems that there is no easy way to identify these decompositions except by brute force computation.  We outline the results below.

\begin{proposition} \label{prop:TP4x4}
	\leavevmode
	\begin{itemize}
		\item Given $\lambda,\mu \in \mathfrak{h}^*$ atypical of degree $1$ and $\lambda+\mu$ atypical of degree of $2$, $\irr{4}{\lambda} \otimes \irr{4}{\mu}$ is isomorphic to
		\begin{align}
			\lambda_i, \mu_i, (\lambda+\mu)_i, (\lambda+\mu)_j \in 2\ZZ+1 &: & &K_{\lambda+\mu-\alpha_j}^{(16)} & \textup{(} i,j &\in \{1,2\} \textup{)}, \notag \\
			\lambda_i, \mu_i, (\lambda+\mu)_i, (\lambda+\mu)_3 \in 2\ZZ+1 &: & &R_{\lambda+\mu-\alpha_j-\alpha_3}^{(8)} \oplus \pro{8}{\lambda+\mu-\alpha_j} & \textup{(} i,j &\in \{1,2\} \textup{)}, \\
			\lambda_3, \mu_3, (\lambda+\mu)_1, (\lambda+\mu)_2 \in 2\ZZ+1 &: & &S_{\lambda+\mu}^{(16)} 
			, \notag
		\end{align}
		where $K_{\lambda+\mu-\alpha_j}^{(16)}$ is an indecomposable quotient of $\pro{24}{\lambda+\mu-\alpha_j} \oplus \pro{24}{\lambda+\mu-\alpha_j-\alpha_3}$, $R_{\lambda+\mu-\alpha_j-\alpha_3}^{(8)}$ is an indecomposable quotient of $\pro{48}{\lambda+\mu-\alpha_j-\alpha_3}$ and $S_{\lambda+\mu}^{(16)}$ is an indecomposable quotient of $\pro{24}{\lambda+\mu-\alpha_1} \oplus \pro{24}{\lambda+\mu-\alpha_2}$. 
		Their (unique) Loewy diagrams are pictured in \cref{fig:LoewyKRSN}.
		\item Given $\lambda,\mu \in \mathfrak{h}^*$ atypical of degree $2$ with $\lambda_i, \lambda_3, \mu_i, \mu_3 \in 2\ZZ+1$, for some $i \in \{1,2\}$, one has
		\begin{equation} \label{eq:TP3x3'}
			\irr{3}{\lambda} \otimes \irr{3}{\mu} \cong \quo{9}{\lambda+\mu-\alpha_j},
		\end{equation}
		where $\quo{9}{\lambda+\mu-\alpha_j}$ is an indecomposable quotient of $\pro{24}{\lambda+\mu-\alpha_j}$, $\{i,j\} = \{1,2\}$.  Its (unique) Loewy diagram is again pictured in \cref{fig:LoewyKRSN}.
	\end{itemize}
\end{proposition}

\begin{figure}
	\centering
	\begin{tikzpicture}[->, thick, scale=0.72]
		\node (12) at (0,4) [] {$\irr{3}{\lambda}$};
		\node (13) at (3,4) [] {$\irr{3}{\lambda-\alpha_3}$};
		\node (21) at (-3,2) [] {$\irr{1}{\lambda+\alpha_i}$};
		\node (22) at (0,2) [] {$\irr{1}{\lambda-\alpha_i}$};
		\node (23) at (3,2) [] {$\irr{1}{\lambda+\alpha_i-2\alpha_3}$};
		\node (24) at (6,2) [] {$\irr{1}{\lambda-\alpha_i-2\alpha_3}$};
		\node (33) at (3,0) [] {$\irr{3}{\lambda-\alpha_3}$};
		\node (32) at (0,0) [] {$\irr{3}{\lambda}$};
		\draw (12) -- (22);
		\draw (12) -- (23);
		\draw (12) -- (21);
		\draw (13) -- (22);
		\draw (13) -- (23);
		\draw (13) -- (24);
		\draw (21) -- (32);
		\draw (22) -- (32);
		\draw (23) -- (32);
		\draw (24) -- (33);
		\draw (22) -- (33);
		\draw (23) -- (33);
		\node[nom] at (-4,4) {$K^{(16)}_{\lambda}$};
		\begin{scope}[shift={(13,0)}]
			\node (1) at (0,4) [] {$\irr{1}{\lambda}$};
			\node (21) at (-2,2) [] {$\irr{3}{\lambda+\alpha_j+\alpha_3}$};
			\node (22) at (2,2) [] {$\irr{3}{\lambda-\alpha_j}$};
			\node (3) at (0,0) [] {$\irr{1}{\lambda}$};
			\draw (1) -- (21);
			\draw (1) -- (22);
			\draw (1) -- (21);
			\draw (21) -- (3);
			\draw (22) -- (3);
			\node[nom] at (-3,4) {$R^{(8)}_{\lambda}$};
		\end{scope}
		\begin{scope}[shift={(1.5,-6)}]
			\node (11) at (-1.5,4) [] {$\irr{3}{\lambda-\alpha_1}$};
			\node (12) at (1.5,4) [] {$\irr{3}{\lambda-\alpha_2}$};
			\node (21) at (-4.5,2) [] {$\irr{1}{\lambda-2\alpha_1}$};
			\node (22) at (-1.5,2) [] {$\irr{1}{\lambda}$};
			\node (22') at (1.5,2) {$\irr{1}{\lambda-2\alpha_3}$};
			\node (23) at (4.5,2) [] {$\irr{1}{\lambda-2\alpha_2}$};
			\node (31) at (-1.5,0) [] {$\irr{3}{\lambda-\alpha_1}$};
			\node (32) at (1.5,0) [] {$\irr{3}{\lambda-\alpha_2}$};
			\draw (11) -- (21);
			\draw (11) -- (22);
			\draw (11) -- (22');
			\draw (12) -- (22');
			\draw (12) -- (22);
			\draw (12) -- (23);
			\draw (21) -- (31);
			\draw (22) -- (31);
			\draw (22) -- (32);
			\draw (22') -- (31);
			\draw (22') -- (32);
			\draw (23) -- (32);
			\node[nom] at (-5.5,4) {$S^{(15)}_{\lambda}$};
		\end{scope}
		%
		%
		\begin{scope}[shift={(13,-6)}]
			\node (1) at (0,4) [] {$\irr{3}{\lambda}$};
			\node (21) at (-3,2) [] {$\irr{1}{\lambda+\alpha_i}$};
			\node (22) at (0,2) [] {$\irr{1}{\lambda-\alpha_i}$};
			\node (23) at (3,2) [] {$\irr{1}{\lambda+\alpha_i-2\alpha_3}$};
			\node (3) at (0,0) [] {$\irr{3}{\lambda}$};
			\draw (1) -- (22);
			\draw (1) -- (23);
			\draw (1) -- (21);
			\draw (21) -- (3);
			\draw (22) -- (3);
			\draw (23) -- (3);
			\node[nom] at (-4,4) {$\quo{9}{\lambda}$};
		\end{scope}
	\end{tikzpicture}
\caption{The Loewy diagrams of the indecomposable $\qgrp$-modules appearing in \cref{prop:TP4x4}.} \label{fig:LoewyKRSN}
\end{figure}

\begin{proof}
	We outline the argument resulting in \eqref{eq:TP3x3'}, the other tensor products being similar though slightly more intricate.  So let $v \in \irr{3}{\lambda} \otimes \irr{3}{\mu}$ have weight $\lambda+\mu$.  Then, $v$ is maximal and $(\lambda+\mu)_1, (\lambda+\mu)_2 \in 2\ZZ+1$, hence there is a composition factor $\irr{1}{\lambda+\mu}$.  Using the explicit coproduct of \eqref{eq:Uisl3hopf} and the structure of the irreducibles worked out in the proof of \cref{irreduciblecharacter}, one computes that $X_{-1} v = 0$ but $X_{-2} v \ne 0$.  Indeed, $X_{-2} v$ is maximal and so there is a composition factor $\irr{3}{\lambda+\mu-\alpha_2}$.  	From the structure of this factor, it follows that $X_{-2} X_{-1} X_{-2} v \ne 0$.  However, $(X_{-1} X_{-2})^2 v = (X_{-2} X_{-1})^2 v = 0$ and so $v$ generates a $4$-dimensional submodule $V$ of $\irr{3}{\lambda} \otimes \irr{3}{\mu}$.

	Next, note that $\lambda+\mu-\alpha_2$ has multiplicity $2$ in $\irr{3}{\lambda} \otimes \irr{3}{\mu}$.  Explicit computation shows that $X_2$ does not act as $0$ on the corresponding weight space.  In other words, there exists $w$ of this weight with $X_2 w = v$.  Since $X_1 w = 0$ by weight considerations, $[w]$ is maximal in $(\irr{3}{\lambda} \otimes \irr{3}{\mu}) \big/ V$ and so gives a second composition factor $\irr{3}{\lambda+\mu-\alpha_2}$.  It follows that $X_{-1} [w], X_{-2} X_{-1} [w] \ne 0$.  However, explicit calculation shows that $X_{-2} [w]$ and $X_{-1} X_{-2} X_{-1} [w]$ are also non-zero, indicating that $[w]$ does not generate an irreducible submodule of $(\irr{3}{\lambda} \otimes \irr{3}{\mu}) \big/ V$.  Indeed, these calculations imply that $(\irr{3}{\lambda} \otimes \irr{3}{\mu}) \big/ V$, and hence $\irr{3}{\lambda} \otimes \irr{3}{\mu}$, has $\irr{1}{\lambda+\mu-2\alpha_2}$ and $\irr{1}{\lambda+\mu-2\alpha_3}$ as composition factors.

	By dimensional considerations, these complete the set of composition factors of $\irr{3}{\lambda} \otimes \irr{3}{\mu}$.  To complete the description of the structure of this module, one checks that $X_{-1} X_{-2} w$ and $X_1 X_{-1} X_{-2} X_{-1} w$ are both non-zero multiples of $X_{-2} X_{-1} X_{-2} v$.  With this, it is easy to compute the radical and socle series of $\irr{3}{\lambda} \otimes \irr{3}{\mu}$ (which again coincide) and draw the Loewy diagram.  Finally, the fact that $\irr{3}{\lambda} \otimes \irr{3}{\mu}$ may be realised as a quotient of $\pro{24}{\lambda+\mu-\alpha_2}$ follows from the latter's role as the projective cover of the head of the former.
\end{proof}

\section{Categories related to $\affvoa$} \label{sec:KL}

In this \lcnamecref{sec:KL}, a Kazhdan--Lusztig correspondence is conjectured between the category of finite-dimensional weight $\qgrp$-modules and a module category of a vertex operator algebra.  More precisely, the conjecture is that this correspondence takes the form of an equivalence of braided tensor categories.  The vertex operator algebra appearing in this correspondence is not $\affvoa$, but its parafermionic coset.  This is of course the commutant of the rank-$2$ Heisenberg subalgebra of $\affvoa$ generated by the fields associated with the Cartan subalgebra of $\slthree$.

The reason why one might expect a Kazhdan--Lusztig correspondence involving this parafermionic vertex operator algebra is that it has recently been shown \cite{A,AdaPar20} to be isomorphic to $\singvoa$, a rank-$2$ analogue of the singlet algebras $\singlet$, $p \in \ZZ_{\ge2}$ \cite{Kausch:1990vg}.  This isomorphism extends the known coincidences between the $L_k(\sltwo)$-parafermions, with $k=-\frac{1}{2}$ and $-\frac{4}{3}$, and the $\singlet$, with $p=2,3$ respectively \cite{Ada-sl2,RidFus10} (as well as between $k=-\frac{2}{3}$ and the bosonic orbifold of an $N=1$ ``supersinglet'' algebra \cite{A,AugMod17}).

$\singvoa$ is one of a family, parametrised by $p \in \ZZ_{\ge2}$ and an ADE Dynkin diagram, of vertex algebras introduced by Feigin and Tipunin \cite{FT}, see also \cite{CM2}.  Aside from the $A_1$ case (the singlet algebras), the representation theories of these vertex operators algebras have not yet received much attention.  However, that of the singlet algebras is very well understood, see \cite{AM-singlet,AM-log,TsuExt13,RidMod13,CMY-singlet}.  In this case, a conjectural braided tensor equivalence has been proposed \cite{CGP,CM} and studied \cite{CMR,CGR,GN,CLR}.  It connects a category of modules over $\singlet$ with the category of finite-dimensional weight modules of the unrolled restricted quantum group $\UHbar{\mfsl{2}}$ with $q=\ee^{\pi \ii / p}$.

It is thus natural to conjecture \cite{CR} that this equivalence generalises to higher ranks and, in particular, to a braided tensor equivalence between the category of finite-dimensional weight $\qgrp$-modules and an appropriate category of $\singvoa$-modules.  Here, this conjecture is made explicit by comparing with the representation theory of the simple affine vertex operator algebra $\affvoa$.  Consequences for other related vertex operator algebra categories are also explored.

\subsection{Representation theory of $\affvoa$} \label{sec:KLaff}

The simple affine vertex operator algebra $\affvoa$ has central charge $-8$.  A list of its irreducible modules includes the following:
\begin{itemize}
	\item Four highest-weight modules $\airr{0}$, $\airr{-\frac{3}{2}\omega_1}$, $\airr{-\frac{3}{2}\omega_2}$ and $\airr{-\frac{1}{2}\rho}$.  These are highest-weight with respect to the standard Borel of $\aslthree$; the subscript indicates the projection of the highest weight onto $\csub$.  Their ground states have conformal weight $0$, $-\frac{1}{2}$, $-\frac{1}{2}$ and $-\frac{1}{2}$, respectively.
	\item A single family of semirelaxed highest-weight modules $\asemi{\mu}$, for all $[\mu] \in (-\frac{3}{2} \omega_1 + \CC \alpha_1) / \ZZ \alpha_1$ except for $[\mu] = [-\frac{3}{2} \omega_1]$ and $[\mu] = [-\frac{1}{2} \rho]$.  The root vectors of $\slthree$ corresponding to $\pm \alpha_1$, $-\alpha_2$ and $-\alpha_3$ act injectively while those for $\alpha_2$ and $\alpha_3$ act locally nilpotently.  The $\slthree$-weights of the ground states constitute the set $\mu + \ZZ \alpha_1 - \NN \alpha_2 - \NN \alpha_3$; these always have conformal weight $-\frac{1}{2}$.
	\item A single family of relaxed highest-weight modules $\arel{\mu}$, for all $[\mu] \in \csub^* / Q$ except
	\begin{equation} \label{eq:Rtyp}
		[\mu] \in [-\tfrac{3}{2} \omega_1 + \CC \alpha_1] \cup [-\tfrac{1}{2} \rho + \CC \alpha_2] \cup [-\tfrac{3}{2} \omega_1 + \CC \alpha_3].
	\end{equation}
	All root vectors of $\slthree$ act injectively and the set of $\slthree$-weights of the ground states is $\mu + Q$; these also always have conformal weight $-\frac{1}{2}$.
\end{itemize}
The highest-weight modules were originally classified in \cite{PerVer07}, while the relaxed highest-weight modules were constructed, but not proved to be exhaustive, in \cite{A}.  A complete classification was first obtained in \cite{AraWei16} using Gelfand-Tsetlin methods, see also \cite{KR1} for an alternative classification based on Mathieu's coherent families \cite{MatCla00}.

A complete set of irreducible positive-energy weight $\affvoa$-modules, with finite-dimensional weight spaces, is obtained from the irreducibles listed above by twisting by the automorphisms of $\slthree$ that preserve the chosen Cartan subalgebra.  These automorphisms form a group isomorphic to the dihedral group $D_6$.  Similarly, a complete set of irreducible weight $\affvoa$-modules, with finite-dimensional weight spaces, is obtained by also twisting by the spectral flow automorphisms $\sfaut{\omega}$, $\omega \in P$, of $\aslthree$.  Recall that the latter are defined by
\begin{equation}
	\sfmod{\omega}{e^\alpha_n} = e^\alpha_{n - \killing{\alpha}{\omega}}, \quad
	\sfmod{\omega}{h_n} = h_n - \killing{\omega}{h} \delta_{n,0} K, \quad
	\sfmod{\omega}{K} = K.
\end{equation}
Occasionally, a given $D_6$- or spectral flow twist will relate two of the irreducibles listed above, for example $\sfmod{\omega_i}{\airr{0}} \cong \airr{-3\omega_i/2}$ for $i=1,2$.  However, most twists result in new irreducibles.  A complete list of the exceptions may be found in \cite[Equation~(5.9)]{KRW}.

Denote the category of irreducible positive-energy weight $\affvoa$-modules, with finite-dimensional weight spaces and \emph{real} weights, by $\affcat$.  The list \cite[Equation~(5.9)]{KRW} of exceptions shows that the irreducibles in $\affcat$ may all be realised as spectral flows of the following irreducibles:
\begin{equation} \label{eq:irruptotwists}
	\airr{0}, \quad \airr{-\frac{1}{2}\rho}, \quad \conjmod{\airr{-\frac{1}{2}\rho}}, \quad \asemi{\mu}, \quad \weylmod{2}{\asemi{\mu}}, \quad \weylaut{1} \weylmod{2}{\asemi{\mu}}, \quad \arel{\mu}.
\end{equation}
Here, $\weylaut{i} \in D_6$, $i=1,2,3$, is the Weyl reflection of $\slthree$ corresponding to $\alpha_i$ and $\conjaut \in D_6$ is the conjugation automorphism (which acts on the Cartan subalgebra of $\slthree$ as multiplication by $-1$).

The relaxed family $\arel{\mu}$ also includes reducible $\affvoa$-modules.  Writing $\omega_3 = \omega_2 - \omega_1$, their composition factors were determined in \cite[Equations~(5.11--12)]{KRW} and are as follows:
\begin{subequations} \label{eq:Rdegens}
	\begin{gather}
		\begin{aligned}
			\mu &\in -\tfrac{3}{2} \omega_1 + \CC \alpha_1, & [\mu] &\ne [-\tfrac{3}{2} \omega_1], [-\tfrac{1}{2} \rho]:
			& &\asemi{\mu},\ \sfmod{-\omega_2}{\asemi{\mu-\frac{1}{2}\alpha_1}}, \\
			\mu &\in -\tfrac{1}{2} \rho + \CC \alpha_2, & [\mu] &\ne [-\tfrac{1}{2} \rho], [-\tfrac{3}{2} \omega_2]:
			& &\weylaut{1} \weylmod{2}{\asemi{\weylaut{2} \cdot \weylmod{1}{\mu}}},\ \sfaut{\omega_1} \weylaut{1} \weylmod{2}{\asemi{\weylaut{2} \weylaut{1} \cdot \mu + \frac{3}{2}\omega_2}}, \\
			\mu &\in -\tfrac{3}{2} \omega_1 + \CC \alpha_3, & [\mu] &\ne [-\tfrac{3}{2} \omega_1], [-\tfrac{3}{2} \omega_2]:
			& &\weylmod{2}{\asemi{\weylmod{2}{\mu}}},\ \sfaut{\omega_3} \weylmod{2}{\asemi{\weylaut{2} \cdot \mu + \frac{3}{2}\omega_2}},
		\end{aligned}
		\label{eq:Rdegen1} \\
		\begin{aligned}
			\mu &= [-\tfrac{3}{2} \omega_1]:
			& &\sfmod{\omega_1}{\airr{0}},\ \sfmod{-\omega_1}{\airr{0}},\ \sfaut{\omega_2} \conjmod{\airr{-\frac{1}{2}\rho}},\ \sfmod{-\omega_2}{\airr{-\frac{1}{2}\rho}}, \\
			\mu &= [-\tfrac{1}{2} \rho]:
			& &\airr{-\frac{1}{2}\rho},\ \conjmod{\airr{-\frac{1}{2}\rho}},\ \sfmod{\omega_3}{\airr{0}},\ \sfmod{-\omega_3}{\airr{0}}, \\
			\mu &= [-\tfrac{3}{2} \omega_2]:
			& &\sfmod{\omega_2}{\airr{0}},\ \sfmod{-\omega_2}{\airr{0}},\ \sfaut{\omega_1} \conjmod{\airr{-\frac{1}{2}\rho}},\ \sfmod{-\omega_1}{\airr{-\frac{1}{2}\rho}}.
		\end{aligned}
		\label{eq:Rdegen2}
	\end{gather}
\end{subequations}
Here, every composition factor has been expressed as a spectral flow of an irreducible in \eqref{eq:irruptotwists}.

Grothendieck fusion rules for all irreducibles in $\affcat$ were computed in \cite{KRW} using a conjectural Verlinde formula \cite{CreLog13,RidVer14}.  As Grothendieck fusion respects $D_6$-twists and spectral flow,
\begin{equation} \label{eq:respect}
	\begin{aligned}
		\weylmod{}{M} \Grfuse \weylmod{}{N} &= \weylmod{}{M \Grfuse N}, &&& \weylaut{} &\in D_6, \\
		\sfmod{\omega}{M} \Grfuse \sfmod{\omega'}{N} &= \sfmod{\omega+\omega'}{M \Grfuse N}, &&& \omega, \omega' &\in P,
	\end{aligned}
\end{equation}
it suffices to list the following rules \cite[Equations~(5.54, 56, 58, 61, 62, 64, 65 and 68)]{KRW}:
\begin{subequations} \label{eq:grfusionrules}
	\begin{align}
		\airr{-\frac{1}{2}\rho} \Grfuse \airr{-\frac{1}{2}\rho} &= \airr{0} + \sfmod{2\omega_1}{\airr{0}} + \sfmod{2\omega_2}{\airr{0}} + 2\, \sfaut{\rho} \conjmod{\airr{-\frac{1}{2}\rho}}, \label{eq:GFR33} \\
		\airr{-\frac{1}{2}\rho} \Grfuse \conjmod{\airr{-\frac{1}{2}\rho}} &= \airr{0} + \arel{0}, \label{eq:GFR33bar} \\
		\airr{-\frac{1}{2}\rho} \Grfuse \asemi{\mu} &= \sfmod{\omega_1}{\asemi{\mu+\frac{1}{2}\alpha_1}} + \sfmod{\omega_2}{\arel{\mu+\frac{1}{2}\alpha_2}}, \label{eq:GFR34} \\
		\airr{-\frac{1}{2}\rho} \Grfuse \arel{\mu} &= \arel{\mu-\frac{1}{2}\rho} + \sfmod{\omega_1}{\arel{\mu+\frac{1}{2}\alpha_1}} + \sfmod{\omega_2}{\arel{\mu+\frac{1}{2}\alpha_2}}, \label{eq:GFR38} \\
		\asemi{\lambda} \Grfuse \asemi{\mu} &= \sfmod{\omega_1}{\asemi{\lambda+\mu+\frac{3}{2}\omega_1}} + \sfmod{\omega_2}{\arel{\lambda+\mu+\frac{3}{2}\omega_2}} + \sfmod{\omega_3}{\asemi{\lambda+\mu+\frac{3}{2}\omega_3}}, \label{eq:GFR44} \\
		\asemi{\lambda} \Grfuse \weylmod{2}{\asemi{\mu}} &= \arel{\lambda+\weylmod{2}{\mu}} + \sfmod{\omega_1}{\arel{\lambda+\weylmod{2}{\mu}+\frac{3}{2}\omega_1}}, \label{eq:GFR4w4} \\
		\asemi{\lambda} \Grfuse \weylaut{1} \weylmod{2}{\asemi{\mu}} &= \arel{\lambda+\weylaut{1}\weylmod{2}{\mu}} + \sfmod{\omega_3}{\arel{\lambda+\weylaut{1}\weylmod{2}{\mu}+\frac{3}{2}\omega_3}}, \label{eq:GFR4ww4} \\
		\asemi{\lambda} \Grfuse \arel{\mu} &= \arel{\lambda+\mu} + \textstyle\sum_{i=1}^3 \sfmod{\omega_i}{\arel{\lambda+\mu+\frac{3}{2}\omega_i}}, \label{eq:GFR48} \\
		\arel{\lambda} \Grfuse \arel{\mu} &= 2 \, \arel{\lambda+\mu} + \textstyle\sum_{i=1}^3 \left( \sfmod{\omega_i}{\arel{\lambda+\mu+\frac{3}{2}\omega_i}} + \sfmod{-\omega_i}{\arel{\lambda+\mu+\frac{3}{2}\omega_i}} \right). \label{eq:GFR88}
	\end{align}
\end{subequations}

The restriction made above, that modules in $\affcat$ have real weights, is purely technical in nature.  However, it is worth noting that the decompositions \eqref{eq:grfusionrules} were also deduced in \cite{KRW} under the assumption that all weights are real.

\subsection{A parafermionic coset} \label{sec:KLcoset}

The vertex operator algebra $\affvoa$ has a rank-$2$ Heisenberg subalgebra $H$ generated by the fields associated with the Cartan subalgebra $\csub$.  As noted above, the commutant of this subalgebra in $\affvoa$ is known \cite{A,AdaPar20} to be isomorphic to the Feigin--Tipunin algebra $\singvoa$.  The consequent conformal embedding $\singvoa \otimes H \into \affvoa$ then leads to decompositions of $\affvoa$-modules as $\singvoa \otimes H$-modules.

To analyse these decompositions, the following \lcnamecref{conj:vtc} is useful.  It is understood to be in force for all that follows.
\begin{conjecture} \label{conj:vtc}
	The category $\affcat$ of weight $\affvoa$-modules with finite-dimensional weight spaces and real weights admits the structure of a vertex tensor category.  Furthermore, the tensor (fusion) product $\singfuse$ of $\affcat$ descends to the Grothendieck group and the resulting Grothendieck fusion rules are as in \eqref{eq:respect} and \eqref{eq:grfusionrules}.
\end{conjecture}
Assuming this \lcnamecref{conj:vtc}, the results of \cite{CKLR} for general Heisenberg commutants imply that the submodule structure of a given module in $\affcat$ is shared by each of the $\singvoa$-modules appearing in its decomposition.  Moreover, they also imply that $\affvoa$ decomposes into simple currents:
\begin{equation} \label{eq:vacdecomp}
	\affvoa \cong \bigoplus_{\lambda \in Q} \singirr{1}{\lambda} \otimes \fock{\lambda}, \qquad
	\singirr{1}{\lambda} \singfuse \singirr{1}{\mu} \cong \singirr{1}{\lambda+\mu}.
\end{equation}
Here, $\fock{\lambda}$ denotes the Fock module of $H$ of highest weight $\lambda$ and $\singirr{1}{\lambda}$, $\lambda \in Q$, is an irreducible $\singvoa$-module.

The natural $\affvoa$-modules with which to start this decomposition game are the relaxed highest-weight $\affvoa$-modules $\arel{\mu} \in \affcat$, because these (along with their spectral flows) are the \emph{standard modules} in the sense of \cite{CreLog13,RidVer14}.  In particular, every other irreducible weight module appears as a composition factor of a standard one, see \eqref{eq:Rdegens}.  As the $\slthree$-weights of $\arel{\mu}$ are precisely the elements of $[\mu] = \mu+Q$, it decomposes into modules over $\singvoa \otimes H$ as follows:
\begin{equation} \label{eq:Rdecomp}
	\arel{\mu} \cong \:\bigoplus_{\mathclap{\lambda \in \mu+Q}}\: \singirr{8}{\lambda} \otimes \fock{\lambda}.
\end{equation}
Here, the $\singirr{8}{\lambda}$ are $\singvoa$-modules satisfying $\singirr{8}{\lambda} \singfuse \singirr{1}{\mu} \cong \singirr{8}{\lambda+\mu}$ \cite{CKLR}.

The character of $\singirr{8}{\lambda}$ is now easily deduced from \cite[Equation~(5.14)]{KRW}:
\begin{equation} \label{eq:cosetRchar}
	\ch{\arel{\mu}} = \frac{1}{\eta(q)^4} \:\sum_{\mathclap{\lambda \in \mu+Q}}\: z^{\lambda}
	= \:\sum_{\mathclap{\lambda \in \mu+Q}}\: \frac{q^{\frac{1}{3}\norm{\lambda}^2}}{\eta(q)^2} \frac{z^{\lambda} q^{-\frac{1}{3}\norm{\lambda}^2}}{\eta(q)^2}
	\quad \Rightarrow \quad
	\ch{\singirr{8}{\lambda}} = \frac{q^{\frac{1}{3}\norm{\lambda}^2}}{\eta(q)^2}.
\end{equation}
This suggests that $\singirr{8}{\lambda}$ may be constructed as another rank-$2$ Fock module, albeit one whose central charge is $-10$ (this free field approach to $\singvoa$ is that originally used in \cite{FT}).  Applying \cite[Equation~(5.4)]{KRW} now gives the character decomposition of a general standard module:
\begin{equation} \label{eq:cosetstchar}
	\ch{\sfmod{\omega}{\arel{\mu}}}
	= \:\sum_{\mathclap{\lambda \in \mu+Q}}\: \frac{q^{\frac{1}{3}\norm{\lambda}^2}}{\eta(q)^2} \frac{z^{\lambda-\frac{3}{2}\omega} q^{-\frac{1}{3}\norm{\lambda-\frac{3}{2}\omega}^2}}{\eta(q)^2}
	= \:\sum_{\mathclap{\lambda \in \mu+Q}}\: \ch{\singirr{8}{\lambda}} \ch{\fock{\lambda-\frac{3}{2}\omega}}.
\end{equation}
When $\arel{\mu}$ is irreducible (typical), this shows that
\begin{equation}
	\sfmod{\omega}{\arel{\mu}} \cong \:\bigoplus_{\mathclap{\lambda \in \mu+Q}}\: \singirr{8}{\lambda} \otimes \fock{\lambda-\frac{3}{2}\omega},
\end{equation}
hence that spectral flow only changes the Fock module labels.  The same must then be true for the atypical $\arel{\mu}$ as well.

\subsection{A logarithmic Kazhdan--Lusztig correspondence} \label{sec:KLcorr}

Recall that $\arel{\mu}$ is typical unless $[\mu] \in \csub^* / Q$ belongs to the set in \eqref{eq:Rtyp}.  The $\singirr{8}{\lambda}$ appearing in \eqref{eq:Rdecomp} are then irreducible (typical) unless
\begin{equation} \label{eq:atypcoset}
	\lambda \in \bigcup_{\substack{i,j=1 \\ i \ne j}}^3 \left( (\ZZ + \tfrac{1}{2}) \alpha_i + \RR \alpha_j \right)
\end{equation}
(recalling the restriction to real weights in the definition of $\affcat$).  These typicality conditions are illustrated in \cref{fig:atypsing}.  Crucially, the conditions for $\singirr{8}{\lambda}$ match those given in \cref{fig:atypwts} for the Verma modules $\verma{\lambda'}$ of $\qgrp$, if $\lambda$ and $\lambda'$ are related by the affine transformation $\lambda' = \varphi(\lambda) + \rho$ of $\csub^*$.  Here, $\varphi$ is the linear map satisfying
\begin{equation} \label{eq:affine}
	\varphi(\alpha_1) = 2 \omega_1, \quad \varphi(\alpha_2) = 2 \omega_3 \quad \text{and} \quad \varphi(\alpha_3) = 2 \omega_2.
\end{equation}
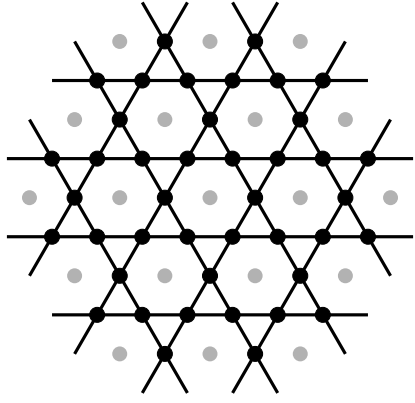
\begin{figure}
	\centering
	\begin{tikzpicture}[scale=1.2]
		\foreach \n in {-2,-1,0,1,2} \node[wt] at (0:\n) {};
		\begin{scope}[shift={(60:1)}] \foreach \n in {-2,-1,0,1} \node[wt] at (0:\n) {}; \end{scope}
		\begin{scope}[shift={(60:2)}] \foreach \n in {-2,-1,0} \node[wt] at (0:\n) {}; \end{scope}
		\begin{scope}[shift={(-60:1)}] \foreach \n in {-2,-1,0,1} \node[wt] at (0:\n) {}; \end{scope}
		\begin{scope}[shift={(-60:2)}] \foreach \n in {-2,-1,0} \node[wt] at (0:\n) {}; \end{scope}
		\begin{scope}[shift={(60:0.5)}] \draw[very thick] (180:2.5) -- (0:2); \end{scope}
		\begin{scope}[shift={(60:1.5)}] \draw[very thick] (180:2.5) -- (0:1); \end{scope}
		\begin{scope}[shift={(-60:0.5)}] \draw[very thick] (180:2.5) -- (0:2); \end{scope}
		\begin{scope}[shift={(-60:1.5)}] \draw[very thick] (180:2.5) -- (0:1); \end{scope}
		\begin{scope}[shift={(60:0.5)}] \draw[very thick] (120:2) -- (-60:2.5); \end{scope}
		\begin{scope}[shift={(60:1.5)}] \draw[very thick] (120:1) -- (-60:2.5); \end{scope}
		\begin{scope}[shift={(-120:0.5)}] \draw[very thick] (120:2.5) -- (-60:2); \end{scope}
		\begin{scope}[shift={(-120:1.5)}] \draw[very thick] (120:2.5) -- (-60:1); \end{scope}
		\begin{scope}[shift={(120:0.5)}] \draw[very thick] (-120:2.5) -- (60:2); \end{scope}
		\begin{scope}[shift={(120:1.5)}] \draw[very thick] (-120:2.5) -- (60:1); \end{scope}
		\begin{scope}[shift={(-60:0.5)}] \draw[very thick] (-120:2) -- (60:2.5); \end{scope}
		\begin{scope}[shift={(-60:1.5)}] \draw[very thick] (-120:1) -- (60:2.5); \end{scope}
		\foreach \n in {-1.5,-0.5,...,1.5} \node[atyp] at (0:\n) {};
		\begin{scope}[shift={(60:0.5)}] \foreach \n in {-2,-1.5,...,1.5} \node[atyp] at (0:\n) {}; \end{scope}
		\begin{scope}[shift={(60:1)}] \foreach \n in {-1.5,-0.5,0.5} \node[atyp] at (0:\n) {}; \end{scope}
		\begin{scope}[shift={(60:1.5)}] \foreach \n in {-2,-1.5,...,0.5} \node[atyp] at (0:\n) {}; \end{scope}
		\begin{scope}[shift={(60:2)}] \foreach \n in {-1.5,-0.5} \node[atyp] at (0:\n) {}; \end{scope}
		\begin{scope}[shift={(-60:0.5)}] \foreach \n in {-2,-1.5,...,1.5} \node[atyp] at (0:\n) {}; \end{scope}
		\begin{scope}[shift={(-60:1)}] \foreach \n in {-1.5,-0.5,0.5} \node[atyp] at (0:\n) {}; \end{scope}
		\begin{scope}[shift={(-60:1.5)}] \foreach \n in {-2,-1.5,...,0.5} \node[atyp] at (0:\n) {}; \end{scope}
		\begin{scope}[shift={(-60:2)}] \foreach \n in {-1.5,-0.5} \node[atyp] at (0:\n) {}; \end{scope}
	\end{tikzpicture}
\caption{The atypical weights $\lambda \in \csub^*$ of the $\singvoa$-modules $\singirr{8}{\lambda}$, pictured with $0$ at the centre.  Grey dots indicate typical weights in $Q$, black lines indicate atypical weights of degree $1$ and black dots are atypical weights of degree $2$.  The highest root $\alpha_3$ is the grey dot closest to $0$ along the ray $30^{\circ}$ to the right of the upwards vertical.} \label{fig:atypsing}
\end{figure}
This suggests the following (partial) Kazhdan--Lusztig correspondence.
\begin{conjecture} \label{conj:btetyp}
	The $\singvoa$-modules obtained by decomposing the $\affvoa$-modules in $\affcat$ form a vertex tensor category $\singcat$.  The identification of $\singirr{8}{\lambda}$ with $\verma{\varphi(\lambda)+\rho}$ partially defines a braided tensor equivalence between $\singcat$ and the full subcategory $\qgrpcat_{\RR}$ of $\qgrpcat$ containing the finite-dimensional weight $\qgrp$-modules with \emph{real} weights.
\end{conjecture}
Of course, if $\lambda$ is typical for $\singvoa$, then $\varphi(\lambda)+\rho$ is typical for $\qgrp$, hence $\verma{\varphi(\lambda)+\rho}$ is irreducible and the identification in \cref{conj:btetyp} becomes
\begin{equation}
	\singirr{8}{\lambda} \corrto \irr{8}{\varphi(\lambda)+\rho}.
\end{equation}
The complete (conjectural) Kazhdan--Lusztig correspondence is obtained by extending this identification to the remaining irreducible $\qgrp$-modules in $\qgrpcat_{\RR}$.

For this, note that degree-$1$ atypicality in $\qgrpcat$ corresponds to the $\asemi{\mu}$ and their $D_6$-twists in $\affcat$.  Decomposing defines (generically irreducible) $\singvoa$-modules $\singirr{4}{\lambda}$:
\begin{subequations} \label{eq:Sdecomp}
	\begin{align}
		\asemi{\mu} &\cong \:\bigoplus_{\mathclap{\lambda \in \mu + Q}}\: \singirr{4}{\lambda} \otimes \fock{\lambda}, &
		\mu &\in -\tfrac{3}{2} \omega_1 + \RR \alpha_1.
	\intertext{This defines the $\singirr{4}{\lambda}$ when $\lambda \in -\frac{3}{2} \omega_1 + \RR \alpha_1 + Q$.  However, twisting by $\weylaut{2}$ and $\weylaut{1} \weylaut{2}$ extends this to $\lambda \in -\frac{3}{2} \omega_1 + \RR \alpha_3 + Q$ and $\lambda \in -\frac{1}{2} \rho + \RR \alpha_2 + Q$, respectively:}
		\weylmod{2}{\asemi{\weylmod{2}{\mu}}} &\cong \:\bigoplus_{\mathclap{\lambda \in \mu + Q}}\: \singirr{4}{\lambda} \otimes \fock{\lambda}, &
		\mu &\in -\tfrac{3}{2} \omega_1 + \RR \alpha_3, \\
		\weylaut{1} \weylmod{2}{\asemi{\weylaut{2} \weylmod{1}{\mu}}} &\cong \:\bigoplus_{\mathclap{\lambda \in \mu + Q}}\: \singirr{4}{\lambda} \otimes \fock{\lambda}, &
		\mu &\in -\tfrac{1}{2} \rho + \RR \alpha_2.
	\end{align}
\end{subequations}
For degree-$2$ atypicality, the decompositions are instead
\begin{equation} \label{eq:Ldecomp}
	\airr{-\frac{1}{2} \rho} \cong \:\bigoplus_{\mathclap{\lambda \in -\frac{1}{2} \rho + Q}}\: \singirr{3}{\lambda} \otimes \fock{\lambda}, \qquad
	\conjmod{\airr{-\frac{1}{2} \rho}} \cong \:\bigoplus_{\mathclap{\lambda \in -\frac{1}{2} \rho + Q}}\: \singirr{\overline{3}}{\lambda} \otimes \fock{\lambda}.
\end{equation}
Recall also the decomposition \eqref{eq:vacdecomp} defining the $\singirr{1}{\lambda}$, $\lambda \in Q$.  In all cases, \cite{CKLR} gives
\begin{equation}
	\singirr{1}{\lambda} \singfuse \singirr{d}{\mu} \cong \singirr{d}{\lambda+\mu}, \quad d=1,3,4,8.
\end{equation}

The full Kazhdan--Lusztig correspondence is as follows.
\begin{conjecture} \label{conj:bte}
	There is a braided tensor equivalence between $\singcat$ and $\qgrpcat_{\RR}$ defined by
	\begin{subequations} \label{eq:KLequiv}
		\begin{gather}
			\singirr{8}{\lambda} \corrto \irr{8}{\varphi(\lambda)+\rho} \qquad \text{for generic $\lambda$}, \\
			\singirr{4}{\lambda} \corrto
			\begin{cases*}
				\irr{4}{\varphi(\lambda)+\rho} & for generic $\lambda$ in $-\tfrac{3}{2} \omega_1 + \RR \alpha_1 + Q$ or $-\frac{3}{2} \omega_1 + \RR \alpha_3 + Q$, \\
				\irr{4}{\varphi(\lambda)} & for generic $\lambda$ in $-\tfrac{1}{2} \rho + \RR \alpha_2 + Q$,
			\end{cases*}
			\\
			\singirr{3}{\lambda} \corrto \irr{3}{\varphi(\lambda)+\rho}, \quad
			\singirr{\overline{3}}{\lambda} \corrto \irr{3}{\varphi(\lambda)} \qquad \text{for $\lambda \in -\frac{1}{2} \rho + Q$}, \\
			\singirr{1}{\lambda} \corrto \irr{1}{\varphi(\lambda)} \qquad \text{for $\lambda \in Q$}.
		\end{gather}
	\end{subequations}
\end{conjecture}
This \lcnamecref{conj:bte} will be tested in what follows by comparing the tensor products of $\qgrpcat_{\RR}$, computed in \cref{tensorsec}, with the Grothendieck fusion rules of $\affcat$, computed in \cite{KRW} and listed in \eqref{eq:grfusionrules}.

This correspondence predicts Loewy diagrams for the projective covers of the irreducible $\singvoa$-modules in $\singcat$.  More precisely, the prediction is that the irreducible $\singirr{8}{\mu}$ are already projective and for the projective covers of the other irreducibles, Loewy diagrams are obtained by replacing each composition factor in \cref{fig:LoewyP} by its image in $\singcat$ under the equivalence \eqref{eq:KLequiv}, see \cref{fig:LoewyE}.  For this, it is convenient to invert \eqref{eq:KLequiv}:
\begin{subequations} \label{eq:KLequiv'}
	\begin{gather}
		\irr{8}{\mu} \corrto \singirr{8}{\varphi^{-1}(\mu-\rho)} \qquad \text{for generic $\mu$}, \\
		\irr{4}{\mu} \corrto
		\begin{cases*}
			\singirr{8}{\varphi^{-1}(\mu-\rho)} & for generic $\mu$ in $\RR \omega_1 + 2P$ or $\RR \omega_2 + 2P$, \\
			\singirr{8}{\varphi^{-1}(\mu)} & for generic $\mu$ in $\omega_1 + \RR \omega_3 + 2P = \omega_2 + \RR \omega_3 + 2P$,
		\end{cases*}
		\\
		\irr{3}{\mu} \corrto \singirr{3}{\varphi^{-1}(\mu-\rho)} \quad \text{for $\mu \in \omega_1 + 2P$}, \qquad
		\irr{3}{\mu} \corrto \singirr{\overline{3}}{\varphi^{-1}(\mu)} \quad \text{for $\mu \in \omega_2 + 2P$}, \\
		\irr{1}{\mu} \corrto \singirr{1}{\varphi^{-1}(\mu)} \quad \text{for $\mu \in 2P$}.
	\end{gather}
\end{subequations}

\begin{figure}
	\centering
	\begin{tikzpicture}[->, thick, scale=0.72]
		\node[nom] at (-3,2) {$\singpro{16}{\lambda}$};
		\node (t) at (0,2) {$\singirr{4}{\lambda}$};
		\node (l) at (2,0) {$\singirr{4}{\lambda-\frac{3}{2}\omega_2}$};
		\node (r) at (-2,0) {$\singirr{4}{\lambda+\frac{3}{2}\omega_2}$};
		\node (b) at (0,-2) {$\singirr{4}{\lambda}$};
		\draw[verma] (t) -- (l);
		\draw[dverma] (t) -- (r);
		\draw[dverma] (l) -- (b);
		\draw[verma] (r) -- (b);
		\node[align=center] at (0,-3.25) {$\lambda \in -\frac{3}{2} \omega_1 + \RR \alpha_1 + Q$, \\ atypical of degree $1$.};
		\begin{scope}[shift={(7,0)}]
			\node (t) at (0,2) {$\singirr{4}{\lambda}$};
			\node (l) at (2,0) {$\singirr{4}{\lambda-\frac{3}{2}\omega_3}$};
			\node (r) at (-2,0) {$\singirr{4}{\lambda+\frac{3}{2}\omega_3}$};
			\node (b) at (0,-2) {$\singirr{4}{\lambda}$};
			\draw[verma] (t) -- (l);
			\draw[dverma] (t) -- (r);
			\draw[dverma] (l) -- (b);
			\draw[verma] (r) -- (b);
			\node[align=center] at (0,-3.25) {$\lambda \in -\frac{3}{2} \omega_1 + \RR \alpha_3 + Q$, \\ atypical of degree $1$.};
		\end{scope}
		\begin{scope}[shift={(14,0)}]
			\node (t) at (0,2) {$\singirr{4}{\lambda}$};
			\node (l) at (2,0) {$\singirr{4}{\lambda-\frac{3}{2}\omega_1}$};
			\node (r) at (-2,0) {$\singirr{4}{\lambda+\frac{3}{2}\omega_1}$};
			\node (b) at (0,-2) {$\singirr{4}{\lambda}$};
			\draw[verma] (t) -- (l);
			\draw[dverma] (t) -- (r);
			\draw[dverma] (l) -- (b);
			\draw[verma] (r) -- (b);
			\node[align=center] at (0,-3.25) {$\lambda \in -\frac{1}{2} \rho + \RR \alpha_2 + Q$, \\ atypical of degree $1$.};
		\end{scope}
	\end{tikzpicture}
	\\ \bigskip
	\begin{tikzpicture}[->, thick, xscale=0.6, yscale=0.72]
			\node (1) at (0,4) [] {$\singirr{3}{\lambda}$};
			\node (21) at (-4,2) [] {$\singirr{1}{\lambda+\frac{3}{2}\rho}$};
			\node (22) at (0,2) [] {$\singirr{1}{\lambda-\frac{3}{2}\omega_3}$};
			\node (23) at (4,2) [] {$\singirr{1}{\lambda+\frac{3}{2}\omega_3}$};
			\node (31) at (-6,0) [] {$\singirr{\overline{3}}{\lambda+3\omega_1}$};
			\node (32) at (-2,0) [] {$\singirr{\overline{3}}{\lambda+3\omega_2}$};
			\node (33) at (2,0) [] {$\singirr{3}{\lambda}$};
			\node (34) at (6,0) [] {$\singirr{\overline{3}}{\lambda}$};
			\node (41) at (-4,-2) [] {$\singirr{1}{\lambda+\frac{3}{2}\rho}$};
			\node (42) at (0,-2) [] {$\singirr{1}{\lambda-\frac{3}{2}\omega_3}$};
			\node (43) at (4,-2) [] {$\singirr{1}{\lambda+\frac{3}{2}\omega_3}$};
			\node (5) at (0,-4) [] {$\singirr{3}{\lambda}$};
			\draw[verma] (1) -- (22);
			\draw[verma] (1) -- (23);
			\draw[verma] (22) -- (34);
			\draw[verma] (23) -- (34);
			\draw[verma] (21) -- (32);
			\draw[verma] (21) -- (33);
			\draw[verma] (32) -- (43);
			\draw[verma] (33) -- (43);
			\draw[verma] (31) -- (41);
			\draw[verma] (31) -- (42);
			\draw[verma] (41) -- (5);
			\draw[verma] (42) -- (5);
			\draw[dverma] (1) -- (21);
			\draw[dverma] (21) -- (31);
			\draw[dverma] (22) -- (31);
			\draw[dverma] (23) -- (32);
			\draw[dverma] (23) -- (33);
			\draw[dverma] (32) -- (41);
			\draw[dverma] (33) -- (41);
			\draw[dverma] (34) -- (42);
			\draw[dverma] (34) -- (43);
			\draw[dverma] (43) -- (5);
			\draw[other] (22) -- (33);
			\draw[other] (33) -- (42);
			\node[nom] at (-6,4) {$\singpro{24}{\lambda}$};
			\node[align=center] at (0,-5.25) {$\lambda \in -\frac{1}{2} \rho + Q$, \\ atypical of degree $2$.};
			\begin{scope}[shift={(14.5,0)}]
				\node (1) at (0,4) [] {$\singirr{\overline{3}}{\lambda}$};
				\node (21) at (-4,2) [] {$\singirr{1}{\lambda-\frac{3}{2}\omega_3}$};
				\node (22) at (0,2) [] {$\singirr{1}{\lambda+\frac{3}{2}\omega_3}$};
				\node (23) at (4,2) [] {$\singirr{1}{\lambda-\frac{3}{2}\rho}$};
				\node (31) at (-6,0) [] {$\singirr{3}{\lambda}$};
				\node (32) at (-2,0) [] {$\singirr{3}{\lambda-3\omega_2}$};
				\node (33) at (2,0) [] {$\singirr{\overline{3}}{\lambda}$};
				\node (34) at (6,0) [] {$\singirr{3}{\lambda-3\omega_1}$};
				\node (41) at (-4,-2) [] {$\singirr{1}{\lambda-\frac{3}{2}\omega_3}$};
				\node (42) at (0,-2) [] {$\singirr{1}{\lambda+\frac{3}{2}\omega_3}$};
				\node (43) at (4,-2) [] {$\singirr{1}{\lambda-\frac{3}{2}\rho}$};
				\node (5) at (0,-4) [] {$\singirr{\overline{3}}{\lambda}$};
				\draw[verma] (1) -- (22);
				\draw[verma] (1) -- (23);
				\draw[verma] (22) -- (34);
				\draw[verma] (23) -- (34);
				\draw[verma] (21) -- (32);
				\draw[verma] (21) -- (33);
				\draw[verma] (32) -- (43);
				\draw[verma] (33) -- (43);
				\draw[verma] (31) -- (41);
				\draw[verma] (31) -- (42);
				\draw[verma] (41) -- (5);
				\draw[verma] (42) -- (5);
				\draw[dverma] (1) -- (21);
				\draw[dverma] (21) -- (31);
				\draw[dverma] (22) -- (31);
				\draw[dverma] (23) -- (32);
				\draw[dverma] (23) -- (33);
				\draw[dverma] (32) -- (41);
				\draw[dverma] (33) -- (41);
				\draw[dverma] (34) -- (42);
				\draw[dverma] (34) -- (43);
				\draw[dverma] (43) -- (5);
				\draw[other] (22) -- (33);
				\draw[other] (33) -- (42);
				\node[nom] at (-6,4) {$\singpro{\overline{24}}{\lambda}$};
				\node[align=center] at (0,-5.25) {$\lambda \in -\frac{1}{2} \rho + Q$, \\ atypical of degree $2$.};
			\end{scope}
		\end{tikzpicture}
	\\ \bigskip
	\begin{tikzpicture}[->, thick, xscale=1.1, yscale=0.72]
			\node (1) at (0,4) [] {$\singirr{1}{\lambda}$};
			\node (21) at (-5,2) [] {$\singirr{\overline{3}}{\lambda+\frac{3}{2}\rho}$};
			\node (22) at (-3,2) [] {$\singirr{3}{\lambda-\frac{3}{2}\omega_3}$};
			\node (23) at (-1,2) [] {$\singirr{3}{\lambda+\frac{3}{2}\omega_3}$};
			\node (24) at (1,2) [] {$\singirr{\overline{3}}{\lambda-\frac{3}{2}\omega_3}$};
			\node (25) at (3,2) [] {$\singirr{\overline{3}}{\lambda+\frac{3}{2}\omega_3}$};
			\node (26) at (5,2) [] {$\singirr{3}{\lambda-\frac{3}{2}\rho}$};
			\node (31) at (-6,0) [] {$\singirr{1}{\lambda+3\omega_1}$};
			\node (32) at (-4,0) [] {$\singirr{1}{\lambda+3\omega_2}$};
			\node (33) at (-2,0) [] {$\singirr{1}{\lambda-3\omega_3}$};
			\node (34) at (0,0) [] {${\singirr{1}{\lambda}}{}^{\oplus4}$};
			\node (35) at (2,0) [] {$\singirr{1}{\lambda+3\omega_3}$};
			\node (36) at (4,0) [] {$\singirr{1}{\lambda-3\omega_2}$};
			\node (37) at (6,0) [] {$\singirr{1}{\lambda-3\omega_1}$};
			\node (41) at (-5,-2) [] {$\singirr{\overline{3}}{\lambda+\frac{3}{2}\rho}$};
			\node (42) at (-3,-2) [] {$\singirr{3}{\lambda-\frac{3}{2}\omega_3}$};
			\node (43) at (-1,-2) [] {$\singirr{3}{\lambda+\frac{3}{2}\omega_3}$};
			\node (44) at (1,-2) [] {$\singirr{\overline{3}}{\lambda-\frac{3}{2}\omega_3}$};
			\node (45) at (3,-2) [] {$\singirr{\overline{3}}{\lambda+\frac{3}{2}\omega_3}$};
			\node (46) at (5,-2) [] {$\singirr{3}{\lambda-\frac{3}{2}\rho}$};
			\node (5) at (0,-4) [] {$\singirr{1}{\lambda}$};
			\draw[verma] (1) -- (25);
			\draw[verma] (1) -- (26);
			\draw[verma] (25) -- (37);
			\draw[verma] (26) -- (37);
			\draw[verma] (21) -- (32);
			\draw[verma] (21) -- (34);
			\draw[verma] (32) -- (43);
			\draw[verma] (34) -- (43);
			\draw[verma] (22) -- (33);
			\draw[verma] (22) -- (34);
			\draw[verma] (33) -- (44);
			\draw[verma] (34) -- (44);
			\draw[verma] (23) -- (34);
			\draw[verma] (23) -- (35);
			\draw[verma] (34) -- (45);
			\draw[verma] (35) -- (45);
			\draw[verma] (24) -- (34);
			\draw[verma] (24) -- (36);
			\draw[verma] (34) -- (46);
			\draw[verma] (36) -- (46);
			\draw[verma] (31) -- (41);
			\draw[verma] (31) -- (42);
			\draw[verma] (41) -- (5);
			\draw[verma] (42) -- (5);
			\draw[dverma] (1) -- (21);
			\draw[dverma] (1) -- (22);
			\draw[dverma] (21) -- (31);
			\draw[dverma] (22) -- (31);
			\draw[dverma] (23) -- (32);
			\draw[dverma] (32) -- (41);
			\draw[dverma] (34) -- (41);
			\draw[dverma] (24) -- (33);
			\draw[dverma] (33) -- (42);
			\draw[dverma] (34) -- (42);
			\draw[dverma] (25) -- (34);
			\draw[dverma] (25) -- (35);
			\draw[dverma] (35) -- (43);
			\draw[dverma] (26) -- (34);
			\draw[dverma] (26) -- (36);
			\draw[dverma] (36) -- (44);
			\draw[dverma] (37) -- (45);
			\draw[dverma] (37) -- (46);
			\draw[dverma] (45) -- (5);
			\draw[dverma] (46) -- (5);
			\draw[other] (1) -- (23);
			\draw[other] (1) -- (24);
			\draw[other] (43) -- (5);
			\draw[other] (44) -- (5);
			\node[nom] at (-6,4) {$\singpro{48}{\lambda}$};
			\node[align=center] at (0,-5.25) {$\lambda \in Q$, \\ atypical of degree $2$.};
		\end{tikzpicture}
	\caption{Conjectural Loewy diagrams for the projective covers of the atypical irreducible $\singvoa$-modules $\singirr{4}{\lambda}$, $\singirr{3}{\lambda}$, $\singirr{\overline{3}}{\lambda}$ and $\singirr{1}{\lambda}$.} \label{fig:LoewyE}
\end{figure}
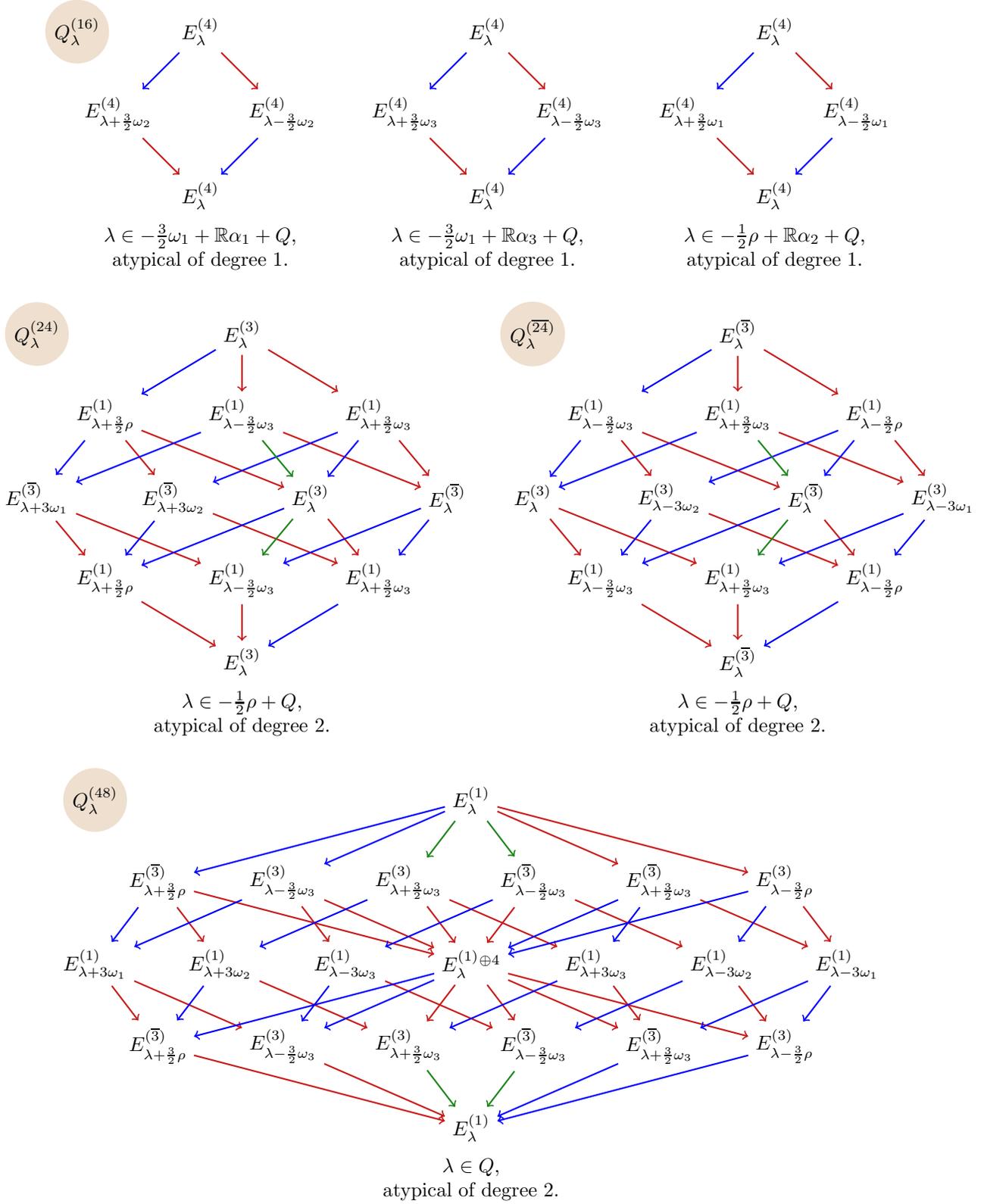

To illustrate this replacement process, let $\singpro{24}{\lambda}$ denote the projective cover of $\singirr{3}{\lambda}$, $\lambda \in -\frac{1}{2} \rho + Q$.  Under the equivalence \eqref{eq:KLequiv}, $\singirr{3}{\lambda} \corrto \irr{3}{\mu}$ so $\singpro{24}{\lambda} \corrto \pro{24}{\mu}$, where $\mu = \varphi(\lambda) + \rho \in \omega_1 + 2P$.  The composition factors of $\pro{24}{\mu}$ are given in \cref{fig:LoewyP}.  Ignoring repetitions, there are only $7$:
\begin{equation}
	\irr{3}{\mu}, \quad \irr{1}{\mu+\alpha_2}, \quad \irr{1}{\mu-\alpha_2}, \quad \irr{1}{\mu+\alpha_2-2\alpha_3}, \quad
	\irr{3}{\mu+\alpha_3}, \quad \irr{3}{\mu+2\alpha_2-\alpha_3} \quad \text{and} \quad \irr{3}{\mu-\alpha_3}.
\end{equation}
Noting that $\mu\pm\alpha_3, \mu+2\alpha_2-\alpha_3 \in \omega_2 + 2P$, the corresponding $\singvoa$-modules under \eqref{eq:KLequiv'} are
\begin{equation}
	\singirr{3}{\lambda}, \quad \singirr{1}{\lambda+\frac{3}{2}\rho}, \quad \singirr{1}{\lambda-\frac{3}{2}\omega_3}, \quad
	\singirr{1}{\lambda+\frac{3}{2}\omega_3}, \quad \singirr{\overline{3}}{\lambda+3\omega_1}, \quad
	\singirr{\overline{3}}{\lambda} \quad \text{and} \quad \singirr{\overline{3}}{\lambda+3\omega_2},
\end{equation}
respectively.  For this, it is convenient to record the following useful facts:
\begin{equation} \label{eq:defvarphiinv}
	\varphi^{-1}(\alpha_1) = -\tfrac{3}{2} \omega_3, \quad
	\varphi^{-1}(\alpha_2) = \tfrac{3}{2} \omega_2 \quad \text{and} \quad
	\varphi^{-1}(\alpha_3) = \tfrac{3}{2} \omega_1.
\end{equation}
Making these replacements in the Loewy diagram of $\pro{24}{\mu}$ now gives that of $\singpro{24}{\lambda}$.

Note that the typical, atypical of degree $1$ and atypical of degree $2$ irreducibles $\singirr{d}{\lambda}$ appear in $1$, $2$ and $3$ distinct Loewy layers of its projective cover, suggesting that the Virasoro zero-mode acts on the projectives with Jordan blocks of rank $1$, $2$ and $3$, respectively.

\subsection{Reconstructing $\affvoa$-modules} \label{sec:KLrecon}

Just as $\singvoa$-modules, tensored with Fock modules, may be obtained by restricting $\affvoa$-modules, so too may $\affvoa$-modules be obtained by \emph{inducing} $\singvoa$-modules, again after tensoring with appropriate Fock modules.  For the situation at hand, the setup is as follows:
\begin{itemize}
	\item Restriction defines a functor from the category $\affcat$ of weight $\affvoa$-modules with finite multiplicities and real weights to a direct limit completion \cite{CMY-lim} of the Deligne product of the category $\singcat$ of $\singvoa$-modules and the category $\fockcat$ of Fock modules over $H$.
	\item The Fock modules in $\fockcat$ will be restricted to having \emph{real} highest weights so that $\fockcat$ is a vertex tensor category \cite[Theorem~2.3]{CKLR}.  This is the reason for restricting $\affcat$ to have real weights.
	\item \cref{conj:btetyp} assumes that $\singcat$ is also a vertex tensor category.
	\item As $\fockcat$ is semisimple, the Deligne product of $\singcat$ and $\fockcat$ is then a vertex tensor category \cite[Theorem~5.5]{CKM2} as is its direct limit completion \cite{CMY-lim}.
\end{itemize}
As $\affvoa$ is a simple current extension in this direct limit completion \cite{CMY-lim,AR}, it now follows from \cite{CKL,CKM1} that $\affcat$ is the category of local modules for $\affvoa$, viewed as a commutative algebra object in the direct limit completion.  There is therefore an induction functor that maps any module $V$ in the completion that centralises the algebra object $\affvoa$ to an object $\induct{V}$ in $\affcat$.  Moreover, this functor is a vertex tensor functor \cite{CKM1}, meaning that it respects the fusion products of the completion and $\affcat$:
\begin{equation} \label{eq:indfusion}
	\induct{V \singfuse V'} \cong \induct{V} \afffuse \induct{V'}.
\end{equation}

Practically, one starts with a $\singvoa$-module such as $\singirr{8}{\lambda}$, tensors it with some Fock module $\fock{\mu}$ and induces to a $\affvoa$-module by fusing with the decomposition \eqref{eq:vacdecomp}:
\begin{equation}
	\induct{\singirr{8}{\lambda} \otimes \fock{\mu}} = \bigoplus_{\nu \in Q} \left( \singirr{8}{\lambda} \otimes \fock{\mu} \right) \singfuse \left( \singirr{1}{\nu} \otimes \fock{\nu} \right) = \bigoplus_{\nu \in Q} \singirr{8}{\lambda+\nu} \otimes \fock{\mu+\nu}.
\end{equation}
The resulting object in the direct limit completion defines a local $\affvoa$-module if and only if the conformal weights of the summands differ by integers.  By \eqref{eq:cosetRchar}, the conformal weights of $\singirr{8}{\lambda} \otimes \fock{\mu}$ are given ($\bmod{1}$) by
\begin{equation}
	\Delta(\lambda,\mu) = \tfrac{1}{3} \norm{\lambda}^2 - \tfrac{1}{2} - \tfrac{1}{3} \norm{\mu}^2 = \tfrac{1}{3} \killing{\lambda-\mu}{\lambda+\mu} - \tfrac{1}{2}.
\end{equation}
The difference between the conformal weights of summands whose $\nu$-labels differ by $\beta \in Q$ is then
\begin{equation}
	\Delta(\lambda+\nu+\beta,\mu+\nu+\beta) - \Delta(\lambda+\nu,\mu+\nu) = \tfrac{2}{3} \killing{\lambda-\mu}{\beta}.
\end{equation}

This proves that inducing $\singirr{8}{\lambda} \otimes \fock{\mu}$ gives an $\affvoa$-module if and only if $\lambda-\mu \in \frac{3}{2} P$.  In this case, \eqref{eq:cosetstchar} identifies the result:
\begin{subequations} \label{eq:induct}
	\begin{equation} \label{eq:identst}
		\induct{\singirr{8}{\lambda} \otimes \fock{\mu}} \cong \sfmod{\frac{2}{3} (\lambda-\mu)}{\arel{\lambda}}.
	\end{equation}
	Similarly, the decompositions \eqref{eq:vacdecomp}, \eqref{eq:Sdecomp} and \eqref{eq:Ldecomp} (and their spectral flows) lead to the following identifications, again assuming that $\lambda-\mu \in \frac{3}{2} P$:
	\begin{gather}
		\induct{\singirr{4}{\lambda} \otimes \fock{\mu}} \cong \label{eq:induce4}
		\begin{cases}
			\sfmod{\frac{2}{3} (\lambda-\mu)}{\asemi{\lambda'}}, & \lambda \in -\tfrac{3}{2} \omega_1 + \RR \alpha_1 + Q, \\
			\sfaut{\frac{2}{3} (\lambda-\mu)} \weylmod{2}{\asemi{\weylmod{2}{\lambda'}}}, & \lambda \in -\tfrac{3}{2} \omega_1 + \RR \alpha_3 + Q, \\
			\sfaut{\frac{2}{3} (\lambda-\mu)} \weylaut{1} \weylmod{2}{\asemi{\weylaut{2} \weylmod{1}{\lambda'}}}, & \lambda \in -\tfrac{1}{2} \rho + \RR \alpha_2 + Q,
		\end{cases}
	\intertext{(here, $\lambda' \in \lambda + Q$ is chosen so that $\lambda' \in -\frac{3}{2} \omega_1 + \RR \alpha_1$, $-\frac{3}{2} \omega_1 + \RR \alpha_3$ and $-\frac{3}{2} \omega_3 + \RR \alpha_2$, respectively)}
		\begin{aligned}
			\induct{\singirr{3}{\lambda} \otimes \fock{\mu}} &\cong \sfmod{\frac{2}{3} (\lambda-\mu)}{\airr{-\frac{1}{2} \rho}}, \\
			\induct{\singirr{\overline{3}}{\lambda} \otimes \fock{\mu}} &\cong \sfaut{\frac{2}{3} (\lambda-\mu)} \conjmod{\airr{-\frac{1}{2} \rho}},
		\end{aligned}
		\qquad \lambda \in -\tfrac{1}{2} \rho + Q, \label{eq:induce3} \\
		\induct{\singirr{1}{\lambda} \otimes \fock{\mu}} \cong \sfmod{\frac{2}{3} (\lambda-\mu)}{\airr{0}}, \qquad \lambda \in Q. \label{eq:induce1}
	\end{gather}
\end{subequations}

With the inductions \eqref{eq:induct} identified, it is now straightforward to predict Loewy diagrams for the projective covers of the irreducible $\affvoa$-modules: the irreducible relaxed modules $\arel{\lambda}$ are already projective and Loewy diagrams for the remaining projective covers are obtained by tensoring each composition factor in \cref{fig:LoewyE} by $\fock{\mu}$ and inducing.  The results are pictured in \cref{fig:LoewyAff}.
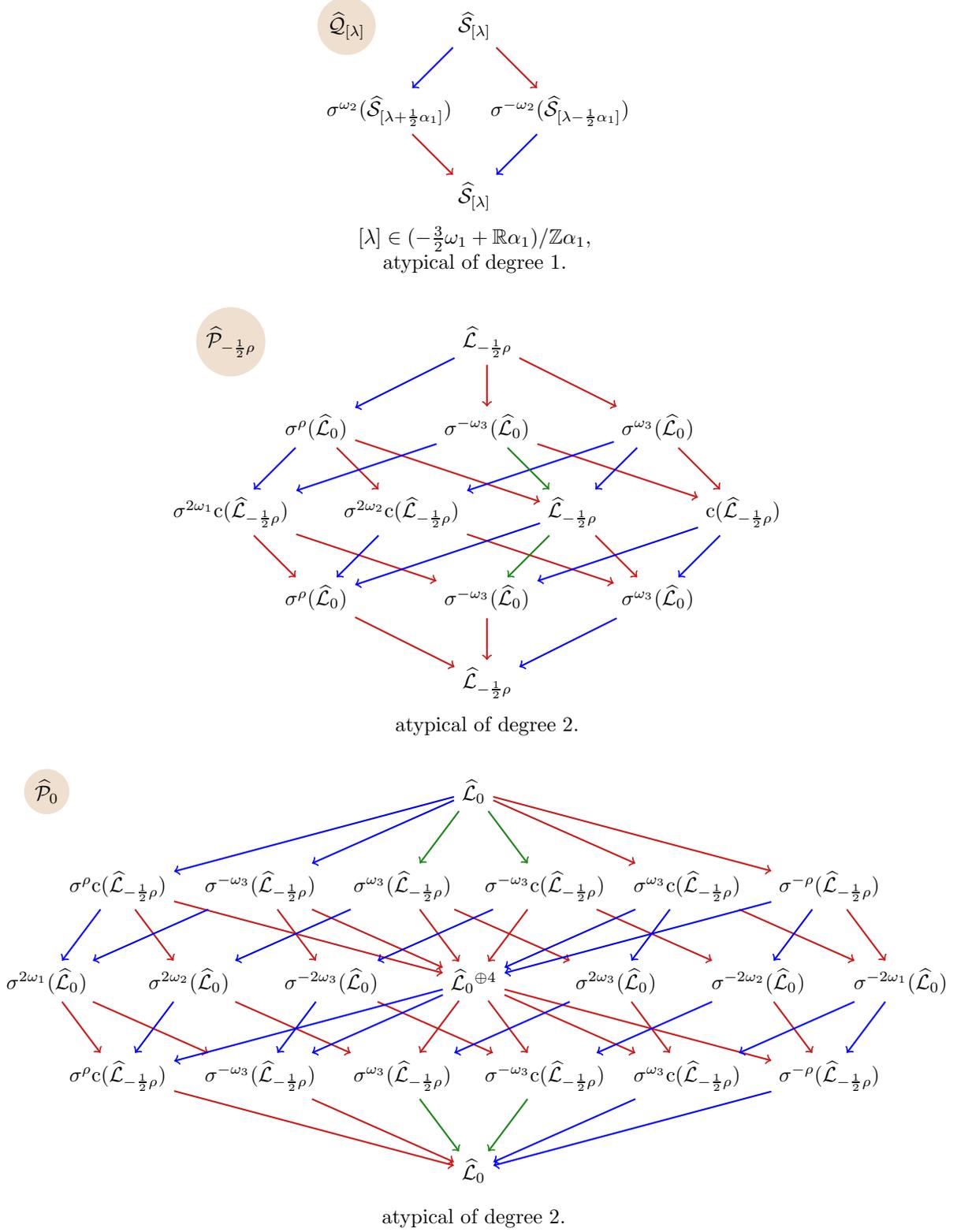
\begin{figure}
	\centering
	\begin{tikzpicture}[->, thick, scale=0.72]
		\node[nom] at (-3,2) {$\asemipro{\lambda}$};
		\node (t) at (0,2) {$\asemi{\lambda}$};
		\node (l) at (2,0) {$\sfmod{-\omega_2}{\asemi{\lambda-\frac{1}{2}\alpha_1}}$};
		\node (r) at (-2,0) {$\sfmod{\omega_2}{\asemi{\lambda+\frac{1}{2}\alpha_1}}$};
		\node (b) at (0,-2) {$\asemi{\lambda}$};
		\draw[verma] (t) -- (l);
		\draw[dverma] (t) -- (r);
		\draw[dverma] (l) -- (b);
		\draw[verma] (r) -- (b);
		\node[align=center] at (0,-3.25) {$[\lambda] \in (-\frac{3}{2} \omega_1 + \RR \alpha_1) / \ZZ \alpha_1$, \\ atypical of degree $1$.};
	\end{tikzpicture}
	\\ \bigskip
	\begin{tikzpicture}[->, thick, scale=0.72]
			\node (1) at (0,4) [] {$\airr{-\frac{1}{2}\rho}$};
			\node (21) at (-4,2) [] {$\sfmod{\rho}{\airr{0}}$};
			\node (22) at (0,2) [] {$\sfmod{-\omega_3}{\airr{0}}$};
			\node (23) at (4,2) [] {$\sfmod{\omega_3}{\airr{0}}$};
			\node (31) at (-6,0) [] {$\sfaut{2\omega_1} \conjmod{\airr{-\frac{1}{2}\rho}}$};
			\node (32) at (-2,0) [] {$\sfaut{2\omega_2} \conjmod{\airr{-\frac{1}{2}\rho}}$};
			\node (33) at (2,0) [] {$\airr{-\frac{1}{2}\rho}$};
			\node (34) at (6,0) [] {$\conjmod{\airr{-\frac{1}{2}\rho}}$};
			\node (41) at (-4,-2) [] {$\sfmod{\rho}{\airr{0}}$};
			\node (42) at (0,-2) [] {$\sfmod{-\omega_3}{\airr{0}}$};
			\node (43) at (4,-2) [] {$\sfmod{\omega_3}{\airr{0}}$};
			\node (5) at (0,-4) [] {$\airr{-\frac{1}{2}\rho}$};
			\draw[verma] (1) -- (22);
			\draw[verma] (1) -- (23);
			\draw[verma] (22) -- (34);
			\draw[verma] (23) -- (34);
			\draw[verma] (21) -- (32);
			\draw[verma] (21) -- (33);
			\draw[verma] (32) -- (43);
			\draw[verma] (33) -- (43);
			\draw[verma] (31) -- (41);
			\draw[verma] (31) -- (42);
			\draw[verma] (41) -- (5);
			\draw[verma] (42) -- (5);
			\draw[dverma] (1) -- (21);
			\draw[dverma] (21) -- (31);
			\draw[dverma] (22) -- (31);
			\draw[dverma] (23) -- (32);
			\draw[dverma] (23) -- (33);
			\draw[dverma] (32) -- (41);
			\draw[dverma] (33) -- (41);
			\draw[dverma] (34) -- (42);
			\draw[dverma] (34) -- (43);
			\draw[dverma] (43) -- (5);
			\draw[other] (22) -- (33);
			\draw[other] (33) -- (42);
			\node[nom] at (-6,4) {$\airrpro{-\frac{1}{2}\rho}$};
			\node[align=center] at (0,-5) {atypical of degree $2$.};
		\end{tikzpicture}
	\\ \bigskip
	\begin{tikzpicture}[->, thick, xscale=1.2, yscale=0.8]
			\node (1) at (0,4) [] {$\airr{0}$};
			\node (21) at (-5,2) [] {$\sfaut{\rho} \conjmod{\airr{-\frac{1}{2}\rho}}$};
			\node (22) at (-3,2) [] {$\sfmod{-\omega_3}{\airr{-\frac{1}{2}\rho}}$};
			\node (23) at (-1,2) [] {$\sfmod{\omega_3}{\airr{-\frac{1}{2}\rho}}$};
			\node (24) at (1,2) [] {$\sfaut{-\omega_3} \conjmod{\airr{-\frac{1}{2}\rho}}$};
			\node (25) at (3,2) [] {$\sfaut{\omega_3} \conjmod{\airr{-\frac{1}{2}\rho}}$};
			\node (26) at (5,2) [] {$\sfmod{-\rho}{\airr{-\frac{1}{2}\rho}}$};
			\node (31) at (-6,0) [] {$\sfmod{2\omega_1}{\airr{0}}$};
			\node (32) at (-4,0) [] {$\sfmod{2\omega_2}{\airr{0}}$};
			\node (33) at (-2,0) [] {$\sfmod{-2\omega_3}{\airr{0}}$};
			\node (34) at (0,0) [] {${\airr{0}}{}^{\oplus4}$};
			\node (35) at (2,0) [] {$\sfmod{2\omega_3}{\airr{0}}$};
			\node (36) at (4,0) [] {$\sfmod{-2\omega_2}{\airr{0}}$};
			\node (37) at (6,0) [] {$\sfmod{-2\omega_1}{\airr{0}}$};
			\node (41) at (-5,-2) [] {$\sfaut{\rho} \conjmod{\airr{-\frac{1}{2}\rho}}$};
			\node (42) at (-3,-2) [] {$\sfmod{-\omega_3}{\airr{-\frac{1}{2}\rho}}$};
			\node (43) at (-1,-2) [] {$\sfmod{\omega_3}{\airr{-\frac{1}{2}\rho}}$};
			\node (44) at (1,-2) [] {$\sfaut{-\omega_3} \conjmod{\airr{-\frac{1}{2}\rho}}$};
			\node (45) at (3,-2) [] {$\sfaut{\omega_3} \conjmod{\airr{-\frac{1}{2}\rho}}$};
			\node (46) at (5,-2) [] {$\sfmod{-\rho}{\airr{-\frac{1}{2}\rho}}$};
			\node (5) at (0,-4) [] {$\airr{0}$};
			\draw[verma] (1) -- (25);
			\draw[verma] (1) -- (26);
			\draw[verma] (25) -- (37);
			\draw[verma] (26) -- (37);
			\draw[verma] (21) -- (32);
			\draw[verma] (21) -- (34);
			\draw[verma] (32) -- (43);
			\draw[verma] (34) -- (43);
			\draw[verma] (22) -- (33);
			\draw[verma] (22) -- (34);
			\draw[verma] (33) -- (44);
			\draw[verma] (34) -- (44);
			\draw[verma] (23) -- (34);
			\draw[verma] (23) -- (35);
			\draw[verma] (34) -- (45);
			\draw[verma] (35) -- (45);
			\draw[verma] (24) -- (34);
			\draw[verma] (24) -- (36);
			\draw[verma] (34) -- (46);
			\draw[verma] (36) -- (46);
			\draw[verma] (31) -- (41);
			\draw[verma] (31) -- (42);
			\draw[verma] (41) -- (5);
			\draw[verma] (42) -- (5);
			\draw[dverma] (1) -- (21);
			\draw[dverma] (1) -- (22);
			\draw[dverma] (21) -- (31);
			\draw[dverma] (22) -- (31);
			\draw[dverma] (23) -- (32);
			\draw[dverma] (32) -- (41);
			\draw[dverma] (34) -- (41);
			\draw[dverma] (24) -- (33);
			\draw[dverma] (33) -- (42);
			\draw[dverma] (34) -- (42);
			\draw[dverma] (25) -- (34);
			\draw[dverma] (25) -- (35);
			\draw[dverma] (35) -- (43);
			\draw[dverma] (26) -- (34);
			\draw[dverma] (26) -- (36);
			\draw[dverma] (36) -- (44);
			\draw[dverma] (37) -- (45);
			\draw[dverma] (37) -- (46);
			\draw[dverma] (45) -- (5);
			\draw[dverma] (46) -- (5);
			\draw[other] (1) -- (23);
			\draw[other] (1) -- (24);
			\draw[other] (43) -- (5);
			\draw[other] (44) -- (5);
			\node[nom] at (-6,4) {$\airrpro{0}$};
			\node[align=center] at (0,-5) {atypical of degree $2$.};
		\end{tikzpicture}
	\caption{Conjectural Loewy diagrams for the projective covers of the atypical irreducible $\affvoa$-modules $\asemi{\lambda}$, $\airr{-\frac{1}{2}\rho}$ and $\airr{0}$.  Projective covers for their $D_6$ and spectral flow twists are obtained by applying the twist to each automorphism.} \label{fig:LoewyAff}
\end{figure}

To verify the (predicted) diagram of the projective cover $\asemipro{\lambda}$ of $\asemi{\lambda}$, $[\lambda] \in (-\frac{3}{2} \omega_1 + \RR \alpha_1) / \ZZ \alpha_1$, start with the (predicted) diagram in \cref{fig:LoewyE} for the projective cover $\singpro{16}{\lambda}$ of the irreducible $\singvoa$-module $\singirr{4}{\lambda}$.  The composition factors of $\singpro{16}{\lambda}$ induce, after tensoring with $\fock{\lambda}$, to give
\begin{equation}
	\asemi{\lambda'}, \quad \sfmod{\omega_2}{\asemi{(\lambda+\frac{3}{2}\omega_2)'}} \quad \text{and} \quad \sfmod{-\omega_2}{\asemi{(\lambda-\frac{3}{2}\omega_2)'}}
\end{equation}
(the first with multiplicity $2$), by \eqref{eq:induce4}.  Here, $\lambda'$ may be taken to be $\lambda$ because the latter already belongs to $-\frac{3}{2} \omega_1 + \RR \alpha_1$ by assumption.  However, $\lambda \pm \frac{3}{2}\omega_2$ does not.  To determine the corresponding primed weight, note that $\frac{3}{2}\omega_2 = \frac{1}{2} \alpha_1 + \alpha_2$ so we may take
\begin{equation}
	(\lambda \pm \tfrac{3}{2}\omega_2)' = \lambda \pm (\tfrac{3}{2}\omega_2 - \alpha_2) = \lambda \pm \tfrac{1}{2}\alpha_1 \in \lambda \pm \tfrac{3}{2}\omega_2 + Q,
\end{equation}
which does lie in $-\frac{3}{2} \omega_1 + \RR \alpha_1$.  The Loewy diagram in \cref{fig:LoewyAff} now follows.

Note again that the multiplicity with which the head of each projective appears suggests that the Virasoro zero-mode acts with rank-$2$ Jordan blocks on the $\asemipro{\lambda}$ and with rank-$3$ Jordan blocks on $\airrpro{0}$ and $\airrpro{-\frac{1}{2}\rho}$ (and their spectral flows).  This is consistent with the recent work \cite{AdaRel21} who have constructed indecomposable $\affvoa$-modules with rank-$2$ and $3$ blocks using inverse quantum hamiltonian reduction \cite{Ada-log,AdaRea20}.  It will be interesting to try to identify their constructions in the setting presented here.

\subsection{Fusion rules}\label{sec:KLfusion}

Another important check of this Kazhdan--Lusztig correspondence is that it is indeed a tensor equivalence: it maps tensor products in $\qgrpcat_{\RR}$ to fusion rules in $\singcat$.  As the latter are not known, the former will be used to deduce fusion rules in $\affcat$ (again using \cite{CKM1}).  These in turn will be compared with the Grothendieck fusion rules \eqref{eq:grfusionrules} computed for $\affcat$ in \cite{KRW}.  Note that a given Grothendieck fusion rule completely characterises the corresponding fusion rule if the fusion product in question happens to be completely reducible.

The first case to analyse is the generic fusion product of two relaxed $\affvoa$-modules (in which the result is completely reducible).  In the decomposition \eqref{eq:Rdecomp}, the $\singvoa$-modules are the irreducible $\singirr{8}{\lambda}$.  Their fusion product is computed using \cref{conj:bte,prop:TP8xM} as follows:
\begin{align}
	\singirr{8}{\lambda} \singfuse \singirr{8}{\mu}
	&\corrto \irr{8}{\varphi(\lambda)+\rho} \otimes \irr{8}{\varphi(\mu)+\rho}
	\cong \bigoplus_{i,j=1}^3 \irr{8}{\varphi(\lambda+\mu)+2\rho-\alpha_i-\alpha_j}
	= \bigoplus_{i,j=1}^3 \irr{8}{\varphi \left( \lambda+\mu + \varphi^{-1}(\alpha_3-\alpha_i-\alpha_j) \right) + \rho} \notag \\
	&\corrto 2 \, \singirr{8}{\lambda+\mu} \oplus \bigoplus_{i=1}^3 \Bigl( \singirr{8}{\lambda+\mu+\frac{3}{2}\omega_i} \oplus \singirr{8}{\lambda+\mu-\frac{3}{2}\omega_i} \Bigr).
\end{align}
Tensoring with $\fock{\lambda} \singfuse \fock{\mu} \cong \fock{\lambda+\mu}$ and inducing now gives
\begin{align}
	\arel{\lambda} \afffuse \arel{\mu}
	&\cong \induct{\singirr{8}{\lambda} \otimes \fock{\lambda}} \afffuse \induct{\singirr{8}{\mu} \otimes \fock{\mu}}
	\cong \induct{(\singirr{8}{\lambda} \singfuse \singirr{8}{\mu}) \otimes \fock{\lambda+\mu}} \notag \\
	&\cong 2 \, \induct{\singirr{8}{\lambda+\mu} \otimes \fock{\lambda+\mu}} \oplus \bigoplus_{i=1}^3 \Bigl( \induct{\singirr{8}{\lambda+\mu+\frac{3}{2}\omega_i} \otimes \fock{\lambda+\mu}} \oplus \induct{\singirr{8}{\lambda+\mu-\frac{3}{2}\omega_i} \otimes \fock{\lambda+\mu}} \Bigr) \notag \\
	&\cong 2 \, \arel{\lambda+\mu} \oplus \bigoplus_{i=1}^3 \Bigl( \sfmod{\omega_i}{\arel{\lambda+\mu+\frac{3}{2}\omega_i}} \oplus \sfmod{-\omega_i}{\arel{\lambda+\mu-\frac{3}{2}\omega_i}} \Bigr),
\end{align}
courtesy of \eqref{eq:indfusion} and \eqref{eq:identst}.  This correctly reproduces the generic fusion rule given in \eqref{eq:GFR88}.

Similar calculations recover the other generic fusion rules in \eqref{eq:grfusionrules} and thereby provide strong evidence for the proposed Kazhdan--Lusztig correspondence (\cref{conj:bte}).  Here, the generic fusion rule involving $\asemi{\lambda}$ and $\weylaut{1} \weylmod{2}{\asemi{\weylaut{2} \weylmod{1}{\mu}}}$, where $\lambda \in -\frac{3}{2} \omega_1 + \RR \alpha_1$ and $\mu \in -\frac{1}{2} \rho + \RR \alpha_2$, is computed to illustrate the different formulae in \eqref{eq:KLequiv}.

First, note that $\singirr{4}{\lambda} \singfuse \singirr{4}{\mu}$ corresponds to
\begin{equation}
	\irr{4}{\varphi(\lambda)+\rho} \otimes \irr{4}{\varphi(\mu)} \cong \irr{8}{\varphi(\lambda+\mu)+\rho} \oplus \irr{8}{\varphi(\lambda+\mu)+\rho-\alpha_1},
\end{equation}
by \cref{prop:TP34x34}, because $\bigl( \varphi(\lambda) + \rho \bigr)_2, \varphi(\mu)_3 \in 2\ZZ+1$.  The conjectural fusion rule in $\singcat$ is thus
\begin{equation}
	\singirr{4}{\lambda} \singfuse \singirr{4}{\mu} \cong \singirr{8}{\lambda+\mu} \oplus \singirr{8}{\lambda+\mu+\frac{3}{2}\omega_3},
\end{equation}
by \eqref{eq:defvarphiinv}.  Inducing then yields the generic fusion rule
\begin{equation}
	\asemi{\lambda} \afffuse \weylaut{1} \weylmod{2}{\asemi{\weylaut{2} \weylmod{1}{\mu}}} \cong \arel{\lambda+\mu} \oplus \sfmod{\omega_3}{\arel{\lambda+\mu+\frac{3}{2}\omega_3}}
\end{equation}
in $\affcat$, which is equivalent to \eqref{eq:GFR4ww4}.

More interesting are the non-generic fusion rules that result in reducible but indecomposable direct summands.  The structures of these indecomposables are not visible in the Grothendieck fusion rules of \eqref{eq:grfusionrules}, but they can now be predicted from the Loewy diagrams of the indecomposable projectives of $\affcat$ depicted in \cref{fig:LoewyAff}.

As an example, consider the tensor product rule
\begin{equation}
	\irr{3}{\lambda} \otimes \irr{8}{\mu} \cong \pro{24}{\lambda+\mu-\alpha_3}
\end{equation}
from \cref{prop:TP3xM}, which holds if $\lambda_i, \lambda_3 \in 2\ZZ+1$ and $\mu_1, \mu_2, \mu_3 \in 2\ZZ$.  Take $i=2$, so that these conditions are equivalent to $\lambda \in \omega_1 + 2P$ and $\mu \in \rho + 2P$.  As $\lambda+\mu-\alpha_3 \in \omega_1 + 2P$, the correspondence \eqref{eq:KLequiv} converts this tensor product rule into the fusion rule
\begin{equation}
	\singirr{3}{\lambda'} \singfuse \singirr{8}{\mu'} \cong \singpro{24}{\lambda'+\mu'},
\end{equation}
where $\lambda' = \varphi^{-1}(\lambda-\rho) \in -\frac{1}{2}\rho + Q$ and $\mu' = \varphi^{-1}(\mu-\rho) \in Q$.  Tensoring with $\fock{\lambda'+\mu'}$ and inducing now predicts the following fusion rule of $\affvoa$-modules:
\begin{equation}
	\airr{-\frac{1}{2}\rho} \afffuse \arel{0} \cong \airrpro{-\frac{1}{2}\rho}.
\end{equation}
One can use \eqref{eq:Rdegen2} to check that the composition factors of this fusion product match those given in the corresponding Grothendieck fusion rule \eqref{eq:GFR38}.

As a second example, suppose that one wished to identify the fusion product of the irreducible $\affvoa$-module $\arel{\lambda}$ and its conjugate $\conjmod{\arel{\lambda}} \cong \arel{-\lambda}$.  This can be computed from the $\singvoa$-module fusion product $\singirr{8}{\lambda} \singfuse \singirr{8}{-\lambda}$ which corresponds, under the Kazhdan--Lusztig correspondence \eqref{eq:KLequiv} to the $\qgrp$ tensor product $\irr{8}{\lambda'} \otimes \irr{8}{\mu'}$, where $\lambda' = \varphi(\lambda)+\rho$ and $\mu' = -\varphi(\lambda)+\rho$.  It is easy to verify that $\lambda'_i + \mu'_i \in 2\ZZ$, for $i=1,2,3$, hence that the required tensor product rule is
\begin{equation}
	\irr{8}{\varphi(\lambda)+\rho} \otimes \irr{8}{-\varphi(\lambda)+\rho} \cong 2\irr{8}{\rho} \oplus \pro{48}{0},
\end{equation}
by \cref{prop:TP8xM}.  Using the correspondence once again, specifically \eqref{eq:KLequiv'}, the $\singvoa$ fusion product is
\begin{equation}
	\singirr{8}{\lambda} \singfuse \singirr{8}{-\lambda} \cong 2\singirr{8}{0} \oplus \singpro{48}{0}.
\end{equation}
Tensoring with Fock spaces and inducing then gives the desired fusion rule for $\affvoa$:
\begin{equation}
	\arel{\lambda} \afffuse \conjmod{\arel{\lambda}} \cong 2\arel{0} \oplus \airrpro{0}.
\end{equation}
Many other affine fusion rules follow similarly from the tensor product computations of \cref{tensorsec}.

\subsection{An octuplet algebra}\label{sec:KLoctuplet}

In \cite{SemVir11,SemNot13}, Semikhatov introduced and studied $\slthree$ analogues of the well known triplet algebras $\triplet$ of Kausch \cite{Kausch:1990vg} (which correspond to $\sltwo$).  As these analogues are generated by $8$ $W_3$ primary fields, rather than $3$ Virasoro ones, he dubbed these the \emph{octuplet} algebras.  The first nontrivial octuplet algebra $\octvoa$ has central charge $c=-10$ and turns out to be a simple current extension of $\singvoa$, just as $\triplet$ is a simple current extension of $\singlet$.

Each of the $\singirr{1}{\lambda}$, $\lambda \in Q$, is a simple current, by \eqref{eq:vacdecomp}.  The conformal weight $\Delta_{\lambda}$ of its highest-weight vector is given by the difference between the minimal conformal weight of a vector of weight $\lambda$ in $\affvoa$ and the conformal weight $\frac{1}{3} \norm{\lambda}^2$ of the highest-weight vector of $\fock{\lambda}$.  Obviously, $\Delta_0 = 0$.  The next-smallest conformal weights occur when $\lambda = \pm\alpha_i$, $i=1,2,3$, in which case $\Delta_{\lambda} = \frac{5}{3}$, and when $\lambda = \pm3\omega_i$, $i=1,2,3$, giving $\Delta_{\lambda} = 4$.  It is this latter set of six simple currents that define $\octvoa$:
\begin{equation}
	\octvoa = \bigoplus_{\lambda \in 3P} \singirr{1}{\lambda}.
\end{equation}
The remaining two octuplet fields have weight $0$.  Their existence is easily checked by using the decomposition \eqref{eq:vacdecomp} to compute the character
\begin{equation}
	\ch{\singirr{1}{0}} = q^{5/12} \bigl( 1 + q^2 + 2q^3 + 5q^4 + \cdots \bigr)
\end{equation}
and noting that the $W_3$-symmetry can only account for $3$ of the $5$ fields of conformal weight $4$.

It is now straightforward to induce an irreducible $\singvoa$-module to obtain a (possibly twisted) irreducible $\octvoa$-module.  As $Q/3P = \bigl\{ [0], [\rho], [-\rho] \bigr\}$, the $\singirr{1}{\lambda}$ yield three induced modules $\octirr{1}{0}$, $\octirr{1}{\rho}$ and $\octirr{1}{-\rho}$, respectively, by fusing with $\octvoa$.  The first is the vacuum module of $\octvoa$ and the others are order-$3$ simple currents, by \eqref{eq:indfusion}.

To determine if an irreducible induced module is untwisted, it suffices to check if the conformal weights of the component $\singvoa$-modules differ by integers.  This will be the case if and only if the corresponding Fock space conformal weights do.  Since
\begin{equation}
	\tfrac{1}{3} \norm{\lambda \pm 3 \omega_i}^2 - \tfrac{1}{3} \norm{\lambda}^2 = 2 \pm \killing{2\lambda}{\omega_i}, \quad i=1,2,3,
\end{equation}
$\singirr{d}{\lambda}$ induces to an untwisted irreducible $\octvoa$-module if and only if $\lambda \in \frac{1}{2} Q$.  In particular, $\octirr{1}{0}$, $\octirr{1}{\alpha_3}$ and $\octirr{1}{-\alpha_3}$ are untwisted $\octvoa$-modules.

Similar analyses result in further untwisted irreducible $\octvoa$-modules.  Their top-space dimensions, $\slthree$-weights and conformal weights are summarised in \cref{tab:octuplet}.  Note that the $\singirr{8}{\lambda}$ with $\lambda \in \frac{1}{2} Q$ are reducible unless $\lambda \in Q$.  Similarly, the $\singirr{4}{\lambda}$ with $\lambda \in \frac{1}{2} Q$ are always reducible.  The number of irreducibles is finite in accord with the expectation that the octuplet algebra $\octvoa$ is $C_2$-cofinite.

\begin{table}
	\renewcommand{\arraystretch}{1.5}
	\centering
	\begin{tabular}{CCCCC}
		\singvoa \text{-mod.} & \octvoa \text{-mod.} & \text{top-space dim.} & \slthree \text{-weights} & \text{conf.\ weight} \\
		\hline\hline
		\multirow{3}{*}{$\singirr{1}{\lambda}$} & \octirr{1}{0} & 1 & 0 & 0 \\
		 & \octirr{1}{\rho} & 3 & \rho,\ \rho-3\omega_1,\ \rho-3\omega_2 & \frac{5}{3} \\
		 & \octirr{1}{-\rho} & 3 & -\rho,\ -\rho+3\omega_1,\ -\rho+3\omega_2 & \frac{5}{3} \\
		\hline
		\multirow{3}{*}{$\singirr{3}{\lambda}$} & \octirr{3}{-\frac{1}{2}\rho} & 1 & -\frac{1}{2}\rho & -\frac{1}{3} \\
		 & \octirr{3}{\frac{1}{2}\rho} & 3 & \frac{1}{2}\rho,\ \frac{1}{2}\rho-3\omega_1,\ \frac{1}{2}\rho-3\omega_2 & \frac{2}{3} \\
		 & \octirr{3}{\frac{3}{2}\rho} & 3 & -\frac{3}{2}\rho,\ -\frac{3}{2}\rho+3\omega_1,\ -\frac{3}{2}\rho+3\omega_2 & 1 \\
		\hline
		\multirow{3}{*}{$\singirr{\overline{3}}{\lambda}$} & \octirr{\overline{3}}{-\frac{1}{2}\rho} & 3 & -\frac{1}{2}\rho,\ -\frac{1}{2}\rho+3\omega_1,\ -\frac{1}{2}\rho+3\omega_2 & -\frac{1}{3} \\
		 & \octirr{\overline{3}}{\frac{1}{2}\rho} & 1 & \frac{1}{2}\rho & \frac{2}{3} \\
		 & \octirr{\overline{3}}{\frac{3}{2}\rho} & 3 & \frac{3}{2}\rho,\ \frac{3}{2}\rho-3\omega_1,\ \frac{3}{2}\rho-3\omega_2 & 1 \\
		\hline
		\multirow{3}{*}{$\singirr{8}{\lambda}$} & \octirr{8}{0} & 1 & 0 & -\frac{1}{2} \\
		 & \octirr{8}{\rho} & 3 & \rho,\ \rho-3\omega_1,\ \rho-3\omega_2 & \frac{1}{6} \\
		 & \octirr{8}{-\rho} & 3 & -\rho,\ -\rho+3\omega_1,\ -\rho+3\omega_2 & \frac{1}{6} \\
		\hline
	\end{tabular}
\caption{The dimensions, $\slthree$-weights and conformal weights of the top spaces of the (untwisted) irreducible $\octvoa$-modules $\octirr{d}{\lambda}$, where $[\lambda] \in Q/3P$.} \label{tab:octuplet}
\end{table}

Predictions for the fusion rules of these irreducible $\octvoa$-modules now follow by applying \eqref{eq:indfusion} to the conjectural fusion rules of the irreducible $\singvoa$-modules, themselves deduced from the tensor product computations of \cref{tensorsec} and the Kazhdan--Lusztig correspondence of \cref{conj:bte}.  Of course, $\octirr{1}{0}$ is the vacuum module and $\octirr{1}{\pm\rho}$ is easily seen to be an order-$3$ simple current:
\begin{equation}
	\octirr{1}{\pm\rho} \octfuse \octirr{1}{\pm\rho} \cong \octirr{1}{\mp\rho}, \quad
	\octirr{1}{\pm\rho} \octfuse \octirr{1}{\mp\rho} \cong \octirr{1}{0}.
\end{equation}
The orbits of the irreducible $\octvoa$-modules under fusion with $\octirr{1}{\rho}$ are as follows:
\begin{equation}
	\begin{gathered}
		\octirr{1}{0} \to \octirr{1}{\rho} \to \octirr{1}{-\rho} \to \octirr{1}{0}, \\
		\octirr{8}{0} \to \octirr{8}{\rho} \to \octirr{8}{-\rho} \to \octirr{1}{0},
	\end{gathered}
	\qquad
	\begin{gathered}
		\octirr{3}{-\frac{1}{2}\rho} \to \octirr{3}{\frac{1}{2}\rho} \to \octirr{3}{\frac{3}{2}\rho} \to \octirr{3}{-\frac{1}{2}\rho}, \\
		\octirr{\overline{3}}{-\frac{1}{2}\rho} \to \octirr{\overline{3}}{\frac{1}{2}\rho} \to \octirr{\overline{3}}{\frac{3}{2}\rho} \to \octirr{\overline{3}}{-\frac{1}{2}\rho}.
	\end{gathered}
\end{equation}

As fusion is associative, the remaining irreducible fusion rules are determined by the six involving the orbit representatives $\octirr{3}{\frac{3}{2}\rho}$, $\octirr{\overline{3}}{\frac{3}{2}\rho}$ and $\octirr{8}{0}$.  These rules are
\begin{equation}
	\begin{aligned}
		\octirr{3}{\frac{3}{2}\rho} \octfuse \octirr{3}{\frac{3}{2}\rho} &\cong \octquo{\overline{9}}{\frac{3}{2}\rho}, &
		\octirr{\overline{3}}{\frac{3}{2}\rho} \octfuse \octirr{\overline{3}}{\frac{3}{2}\rho} &\cong \octquo{9}{\frac{3}{2}\rho}, &
		\octirr{3}{\frac{3}{2}\rho} \octfuse \octirr{\overline{3}}{\frac{3}{2}\rho} &\cong \octirr{8}{0} \oplus \octirr{1}{0}, \\
		\octirr{3}{\frac{3}{2}\rho} \octfuse \octirr{8}{0} &\cong \octpro{24}{\frac{3}{2}\rho}, &
		\octirr{\overline{3}}{\frac{3}{2}\rho} \octfuse \octirr{8}{0} &\cong \octpro{\overline{24}}{\frac{3}{2}\rho}, &
		\octirr{8}{0} \octfuse \octirr{8}{0} &\cong 2\octirr{8}{0} \oplus \octpro{48}{0}.
	\end{aligned}
\end{equation}
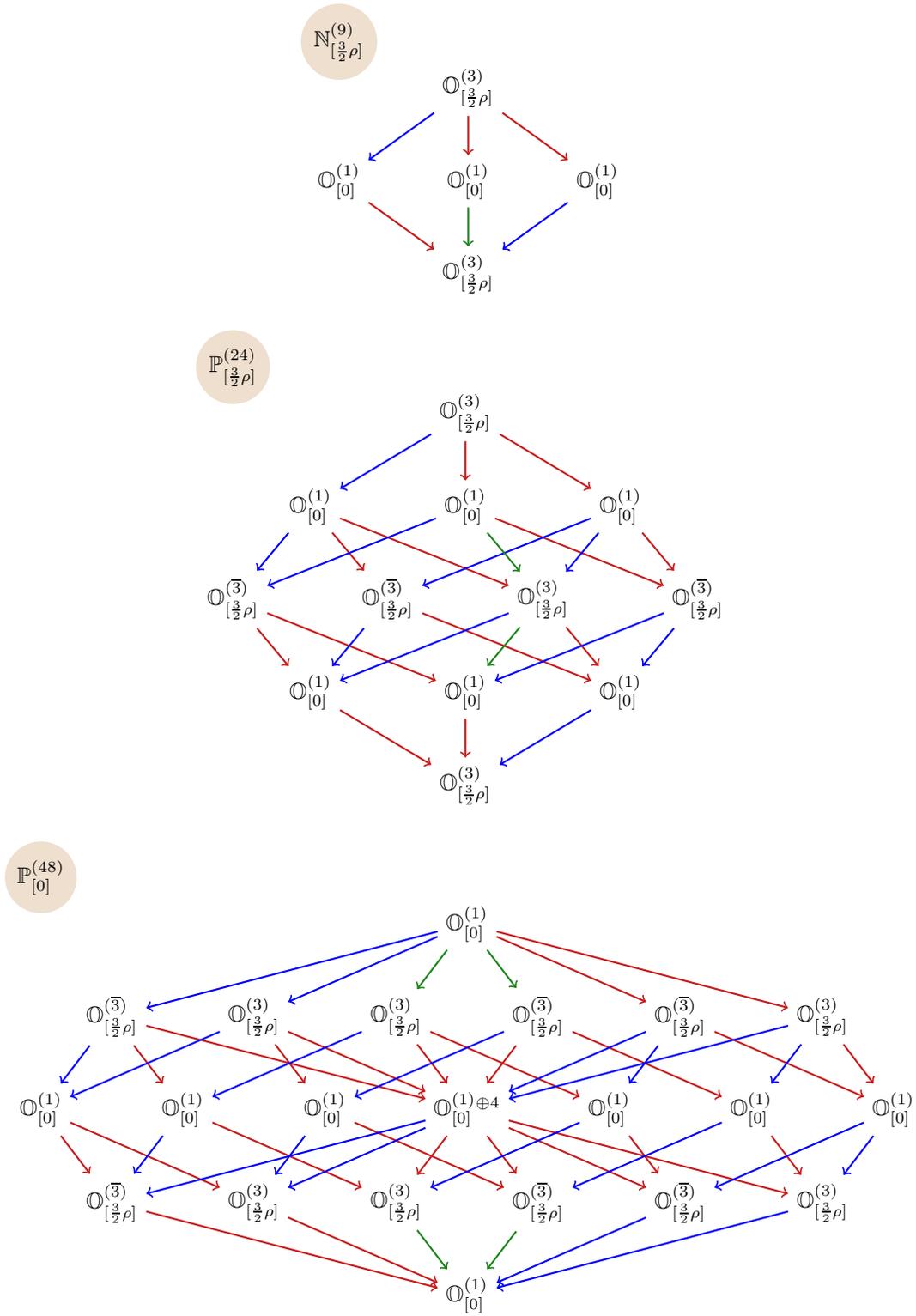
\begin{figure}
	\centering
	\begin{tikzpicture}[->, thick, xscale=0.5, yscale=0.72]
		\node (1) at (0,4) [] {$\octirr{3}{\frac{3}{2}\rho}$};
		\node (21) at (-4,2) [] {$\octirr{1}{0}$};
		\node (22) at (0,2) [] {$\octirr{1}{0}$};
		\node (23) at (4,2) [] {$\octirr{1}{0}$};
		\node (33) at (0,0) [] {$\octirr{3}{\frac{3}{2}\rho}$};
		\draw[verma] (1) -- (22);
		\draw[verma] (1) -- (23);
		\draw[dverma] (1) -- (21);
		\draw[verma] (21) -- (33);
		\draw[other] (22) -- (33);
		\draw[dverma] (23) -- (33);
		\node[nom] at (-4,5) {$\octquo{9}{\frac{3}{2}\rho}$};
	\end{tikzpicture}
	\\ \bigskip
	\begin{tikzpicture}[->, thick, xscale=0.6, yscale=0.72]
			\node (1) at (0,4) [] {$\octirr{3}{\frac{3}{2}\rho}$};
			\node (21) at (-4,2) [] {$\octirr{1}{0}$};
			\node (22) at (0,2) [] {$\octirr{1}{0}$};
			\node (23) at (4,2) [] {$\octirr{1}{0}$};
			\node (31) at (-6,0) [] {$\octirr{\overline{3}}{\frac{3}{2}\rho}$};
			\node (32) at (-2,0) [] {$\octirr{\overline{3}}{\frac{3}{2}\rho}$};
			\node (33) at (2,0) [] {$\octirr{3}{\frac{3}{2}\rho}$};
			\node (34) at (6,0) [] {$\octirr{\overline{3}}{\frac{3}{2}\rho}$};
			\node (41) at (-4,-2) [] {$\octirr{1}{0}$};
			\node (42) at (0,-2) [] {$\octirr{1}{0}$};
			\node (43) at (4,-2) [] {$\octirr{1}{0}$};
			\node (5) at (0,-4) [] {$\octirr{3}{\frac{3}{2}\rho}$};
			\draw[verma] (1) -- (22);
			\draw[verma] (1) -- (23);
			\draw[verma] (22) -- (34);
			\draw[verma] (23) -- (34);
			\draw[verma] (21) -- (32);
			\draw[verma] (21) -- (33);
			\draw[verma] (32) -- (43);
			\draw[verma] (33) -- (43);
			\draw[verma] (31) -- (41);
			\draw[verma] (31) -- (42);
			\draw[verma] (41) -- (5);
			\draw[verma] (42) -- (5);
			\draw[dverma] (1) -- (21);
			\draw[dverma] (21) -- (31);
			\draw[dverma] (22) -- (31);
			\draw[dverma] (23) -- (32);
			\draw[dverma] (23) -- (33);
			\draw[dverma] (32) -- (41);
			\draw[dverma] (33) -- (41);
			\draw[dverma] (34) -- (42);
			\draw[dverma] (34) -- (43);
			\draw[dverma] (43) -- (5);
			\draw[other] (22) -- (33);
			\draw[other] (33) -- (42);
			\node[nom] at (-6,5) {$\octpro{24}{\frac{3}{2}\rho}$};
		\end{tikzpicture}
	\\ \bigskip
	\begin{tikzpicture}[->, thick, xscale=1.1, yscale=0.72]
			\node (1) at (0,4) [] {$\octirr{1}{0}$};
			\node (21) at (-5,2) [] {$\octirr{\overline{3}}{\frac{3}{2}\rho}$};
			\node (22) at (-3,2) [] {$\octirr{3}{\frac{3}{2}\rho}$};
			\node (23) at (-1,2) [] {$\octirr{3}{\frac{3}{2}\rho}$};
			\node (24) at (1,2) [] {$\octirr{\overline{3}}{\frac{3}{2}\rho}$};
			\node (25) at (3,2) [] {$\octirr{\overline{3}}{\frac{3}{2}\rho}$};
			\node (26) at (5,2) [] {$\octirr{3}{\frac{3}{2}\rho}$};
			\node (31) at (-6,0) [] {$\octirr{1}{0}$};
			\node (32) at (-4,0) [] {$\octirr{1}{0}$};
			\node (33) at (-2,0) [] {$\octirr{1}{0}$};
			\node (34) at (0,0) [] {${\octirr{1}{0}}{}^{\oplus4}$};
			\node (35) at (2,0) [] {$\octirr{1}{0}$};
			\node (36) at (4,0) [] {$\octirr{1}{0}$};
			\node (37) at (6,0) [] {$\octirr{1}{0}$};
			\node (41) at (-5,-2) [] {$\octirr{\overline{3}}{\frac{3}{2}\rho}$};
			\node (42) at (-3,-2) [] {$\octirr{3}{\frac{3}{2}\rho}$};
			\node (43) at (-1,-2) [] {$\octirr{3}{\frac{3}{2}\rho}$};
			\node (44) at (1,-2) [] {$\octirr{\overline{3}}{\frac{3}{2}\rho}$};
			\node (45) at (3,-2) [] {$\octirr{\overline{3}}{\frac{3}{2}\rho}$};
			\node (46) at (5,-2) [] {$\octirr{3}{\frac{3}{2}\rho}$};
			\node (5) at (0,-4) [] {$\octirr{1}{0}$};
			\draw[verma] (1) -- (25);
			\draw[verma] (1) -- (26);
			\draw[verma] (25) -- (37);
			\draw[verma] (26) -- (37);
			\draw[verma] (21) -- (32);
			\draw[verma] (21) -- (34);
			\draw[verma] (32) -- (43);
			\draw[verma] (34) -- (43);
			\draw[verma] (22) -- (33);
			\draw[verma] (22) -- (34);
			\draw[verma] (33) -- (44);
			\draw[verma] (34) -- (44);
			\draw[verma] (23) -- (34);
			\draw[verma] (23) -- (35);
			\draw[verma] (34) -- (45);
			\draw[verma] (35) -- (45);
			\draw[verma] (24) -- (34);
			\draw[verma] (24) -- (36);
			\draw[verma] (34) -- (46);
			\draw[verma] (36) -- (46);
			\draw[verma] (31) -- (41);
			\draw[verma] (31) -- (42);
			\draw[verma] (41) -- (5);
			\draw[verma] (42) -- (5);
			\draw[dverma] (1) -- (21);
			\draw[dverma] (1) -- (22);
			\draw[dverma] (21) -- (31);
			\draw[dverma] (22) -- (31);
			\draw[dverma] (23) -- (32);
			\draw[dverma] (32) -- (41);
			\draw[dverma] (34) -- (41);
			\draw[dverma] (24) -- (33);
			\draw[dverma] (33) -- (42);
			\draw[dverma] (34) -- (42);
			\draw[dverma] (25) -- (34);
			\draw[dverma] (25) -- (35);
			\draw[dverma] (35) -- (43);
			\draw[dverma] (26) -- (34);
			\draw[dverma] (26) -- (36);
			\draw[dverma] (36) -- (44);
			\draw[dverma] (37) -- (45);
			\draw[dverma] (37) -- (46);
			\draw[dverma] (45) -- (5);
			\draw[dverma] (46) -- (5);
			\draw[other] (1) -- (23);
			\draw[other] (1) -- (24);
			\draw[other] (43) -- (5);
			\draw[other] (44) -- (5);
			\node[nom] at (-6,5) {$\octpro{48}{0}$};
		\end{tikzpicture}
	\caption{Conjectural Loewy diagrams for the $\octvoa$-modules $\octquo{9}{\frac{3}{2}\rho}$, $\octpro{24}{\frac{3}{2}\rho}$ and $\octpro{48}{0}$.} \label{fig:LoewyOct}
\end{figure}

Here, $\octpro{24}{\frac{3}{2}\rho}$, $\octpro{\overline{24}}{\frac{3}{2}\rho}$ and $\octpro{48}{0}$ are the (conjectural) projective covers of $\octirr{3}{\frac{3}{2}\rho}$, $\octirr{\overline{3}}{\frac{3}{2}\rho}$ and $\octirr{1}{0}$, respectively, while $\octquo{9}{\frac{3}{2}\rho}$ and $\octquo{\overline{9}}{\frac{3}{2}\rho}$ are quotients of $\octirr{3}{\frac{3}{2}\rho}$ and $\octirr{\overline{3}}{\frac{3}{2}\rho}$, respectively.

Loewy diagrams for $\octquo{9}{\frac{3}{2}\rho}$, $\octpro{24}{\frac{3}{2}\rho}$ and $\octpro{48}{0}$ are pictured in \cref{fig:LoewyOct}.
Those of $\octquo{\overline{9}}{\frac{3}{2}\rho}$ and $\octpro{\overline{24}}{\frac{3}{2}\rho}$ are obtained from their unbarred cousins by exchanging $3$ with $\overline{3}$ everywhere.  Those of the projective covers of the remaining irreducibles are then obtained by fusing each composition factor with whichever of the simple currents $\octirr{1}{\rho}$ or $\octirr{1}{-\rho}$ is appropriate.  Finally, the irreducibles $\octirr{8}{0}$, $\octirr{8}{\rho}$ and $\octirr{8}{-\rho}$ are their own projective covers.  The Virasoro zero-mode is expected to act with Jordan blocks of rank $3$ on $\octpro{24}{\frac{3}{2}\rho}$ and $\octpro{48}{0}$.

\footnotesize
\flushleft


\providecommand{\opp}[2]{\textsf{arXiv:\mbox{#2}/#1}}
\providecommand{\pp}[2]{\textsf{arXiv:#1 [\mbox{#2}]}}

\end{document}